\documentclass[twoside,11pt]{article}
\usepackage{jmlr2e}
\usepackage{amsmath}
\usepackage{pifont}
\usepackage{xcolor}
\usepackage{hyperref,url,multirow,multicol}
\usepackage{algorithm}
\usepackage{algpseudocode}
\usepackage{amssymb,makeidx,mathrsfs}
\usepackage{fixltx2e}
\usepackage[T1]{fontenc} % Use 8-bit encoding that has 256 glyphs
\usepackage[a4paper]{geometry}
\usepackage{multicol} % Used for the two-column layout of the document
\usepackage[hang, small,labelfont=bf,up,textfont=it,up]{caption} % Custom captions under/above floats in tables or figures
\usepackage{booktabs} % Horizontal rules in tables
\usepackage{float} % Required for tables and figures in the multi-column environment - they need to be placed in specific locations with the [H] (e.g. \begin{table}[H])
\usepackage{hyperref} % For hyperlinks in the PDF
\usepackage{abstract} % Allows abstract customization
\usepackage{fancyhdr} % Headers and footers
\usepackage{ulem}

\newtheorem{assumption}{Assumption}[section]
\makeatletter
\renewcommand\@biblabel[1]{#1.}
\makeatother

\newcommand{\R}{\mathbb{R}}
\newcommand{\Rext}{\R\cup\{+\infty\}}

\newcommand{\abs}[1]{\left\vert#1\right\vert}

\newcommand{\set}[1]{\left\{#1\right\}}
\newcommand{\sets}[1]{\{#1\}}
\newcommand{\norm}[1]{\left\Vert#1\right\Vert}
\newcommand{\norms}[1]{\Vert#1\Vert}

\newcommand{\Eproof}{\hfill $\square$}
\newcommand{\prox}{\mathrm{prox}}

\newcommand{\proj}{\mathrm{proj}}

\newcommand{\dom}[1]{\mathrm{dom}(#1)}

\newcommand{\zero}[1]{{\boldsymbol{0}}}
\newcommand{\Exps}[2]{\mathbb{E}_{#1}\left[#2\right]}
\newcommand{\Prob}[1]{\mathbf{Prob}\left(#1\right)}

\newcommand{\Ac}{\mathcal{A}}

\newcommand{\Pbb}{\mathbb{P}}

\newcommand{\Bc}{\mathcal{B}}

\newcommand{\Ub}{\mathbf{U}}

\newcommand{\Xc}{\mathcal{X}}

\newcommand{\Sc}{\mathcal{S}}

\newcommand{\Tc}{\mathcal{T}}

\newcommand{\Fc}{\mathcal{F}}

\newcommand{\bsigma}{\boldsymbol{\sigma}}

\newcommand{\iprods}[1]{\langle #1\rangle}

\newcommand{\Exp}[1]{\mathbb{E}\left[#1\right]}
\newcommand{\cExp}[2]{\mathbb{E}\left[#1\mid #2\right]}

\newcommand{\BigO}[1]{\mathcal{O}\left(#1\right)}

\newcommand{\cmark}{\ding{51}}%
\newcommand{\xmark}{\ding{55}}%
\newcommand{\done}{\rlap{$\square$}{\raisebox{2pt}{\large\hspace{1pt}\cmark}}%
\hspace{-2.5pt}}
\newcommand{\wontfix}{\rlap{$\square$}{\large\hspace{1pt}\xmark}}

\newcommand{\nhanp}[1]{{#1}}

\newcommand{\beforesec}{\vspace{-2.25ex}}
\newcommand{\aftersec}{\vspace{-1.5ex}}
\newcommand{\beforesubsec}{\vspace{-2ex}}
\newcommand{\aftersubsec}{\vspace{-1.25ex}}

\newcommand{\beforepara}{\vspace{-1.0ex}}

\ShortHeadings{ProxSARAH Algorithms for Stochastic Composite Nonconvex Optimization}{N. H. Pham, L. M. Nguyen, D. T. Phan, and Q. Tran-Dinh}
\firstpageno{1}

\begin{document}

\title{ProxSARAH: An Efficient Algorithmic Framework for Stochastic Composite Nonconvex Optimization}

\author{\name Nhan H. Pham$^{\dagger}$ \email nhanph@live.unc.edu
       \AND
       \name Lam M. Nguyen$^{\ddagger}$ \email lamnguyen.mltd@ibm.com \\
       \name Dzung T. Phan$^{\ddagger}$  \email phandu@us.ibm.com \\
       \addr $^{\ddagger}\!$IBM Research, Thomas J. Watson Research Center \\ Yorktown Heights, NY10598, USA
       \AND
       \name Quoc Tran-Dinh$^{\dagger}$ \email quoctd@email.unc.edu \\
       \addr $^{\dagger}\!$Department of Statistics and Operations Research\\
       University of North Carolina at Chapel Hill, Chapel Hill, NC27599, USA.
       }

\editor{~}
%\vspace{-2ex}
\maketitle
\vspace{-4ex}
\begin{abstract}
\vspace{-1ex}
We propose a new stochastic first-order algorithmic framework to solve stochastic  composite nonconvex optimization problems that covers both finite-sum and expectation settings.
Our algorithms rely on the SARAH estimator introduced in \citep{Nguyen2017_sarah} and consist of two steps: a proximal gradient  and an averaging step making them different from existing nonconvex proximal-type algorithms.
The algorithms only require an average smoothness assumption of the nonconvex objective term and additional bounded variance assumption if applied to expectation problems.
They work with both  constant and adaptive step-sizes, while allowing single sample and mini-batches. 
In all these cases,  we prove that our algorithms can  achieve the best-known complexity bounds.
One key step of our methods is new constant and adaptive step-sizes that help to achieve desired complexity bounds while improving practical performance.
Our constant step-size is much larger than existing methods including proximal SVRG schemes in the single sample case. 
We also specify the algorithm to the non-composite case that covers  existing state-of-the-arts in terms of complexity bounds.
Our  update also allows one to trade-off between step-sizes and mini-batch sizes to improve performance.
We test the proposed algorithms on two composite nonconvex problems and neural networks using several well-known datasets. 

\noindent\textbf{The first version of this paper was online on Arxiv on February 15, 2019.} 
\vspace{-2ex}
\end{abstract}

\begin{keywords}
Stochastic proximal gradient descent; optimal convergence rate; composite nonconvex optimization; finite-sum minimization; expectation minimization.
\end{keywords}

%%% 1. Introduction.    
\beforesec
\section{Introduction}\label{sec:intro}
\aftersec 
In this paper, we consider the following stochastic composite, nonconvex, and possibly nonsmooth optimization problem:
\begin{equation}\label{eq:sopt_prob}
\min_{w\in\R^d}\Big\{ F(w) := f(w) + \psi(w) \equiv  \Exp{f(w; \xi)} + \psi(w) \Big\},
\end{equation}
where $f(w) := \Exp{f(w; \xi)}$ is the expectation of a stochastic function $f(w; \xi)$ depending on a random vector $\xi$ in a given probability space $(\Omega,\Pbb)$, and $\psi : \R^d\to\Rext$ is a proper, closed, and convex function.

As a special case of \eqref{eq:sopt_prob}, if $\xi$ is a uniformly random vector defined on a finite support set $\Omega := \set{\xi_1, \xi_2, \cdots, \xi_n}$, then \eqref{eq:sopt_prob} reduces to the following composite nonconvex finite-sum minimization problem:
\vspace{-1ex}
\begin{equation}\label{eq:finite_sum}
\min_{w\in\R^d}\Big\{ F(w) := f(w) + \psi(w) \equiv \frac{1}{n}\sum_{i=1}^nf_i(w) + \psi(w) \Big\},
\vspace{-1ex}
\end{equation}
where $f_i(w) := f(w; \xi_i)$ for $i=1,\cdots, n$.
Problem \eqref{eq:finite_sum} is  often referred to as a regularized empirical risk minimization in machine learning and finance.

\beforepara
\paragraph{Motivation:}
Problems \eqref{eq:sopt_prob} and \eqref{eq:finite_sum} cover a broad range of applications in machine learning and statistics, especially in neural networks, see, e.g. \citep{Bottou1998,bottou2010large,Bottou2018,goodfellow2016deep,sra2012optimization}. 
Hitherto, state-of-the-art numerical optimization methods for solving these problems rely on stochastic approaches, see, e.g. \citep{johnson2013accelerating,schmidt2017minimizing,Shapiro2009,SAGA}.
In the convex case, both non-composite and composite settings  \eqref{eq:sopt_prob} and \eqref{eq:finite_sum}  have been intensively studied with different schemes such as standard stochastic gradient \citep{robbins1951stochastic}, proximal stochastic gradient \citep{ghadimi2013stochastic,Nemirovski2009a}, stochastic dual coordinate descent \citep{shalev2013stochastic}, variance reduction methods (e.g., SVRG and SAGA) \citep{allen2016katyusha,SAGA,johnson2013accelerating,nitanda2014stochastic,schmidt2017minimizing,Xiao2014}, stochastic conditional gradient (Frank-Wolfe) methods \citep{reddi2016stochastic}, and stochastic primal-dual methods \citep{chambolle2017stochastic}.
Thanks to variance reduction techniques, several efficient methods with constant step-sizes have been developed for convex settings that match the lower-bound worst-case complexity \citep{agarwal2010information}.
However, variance reduction methods for nonconvex settings are still limited and heavily focus on the non-composite form of \eqref{eq:sopt_prob} and \eqref{eq:finite_sum}, i.e. $\psi = 0$, and the SVRG estimator.

Theory and stochastic methods for nonconvex problems are still in progress and require substantial effort to obtain efficient algorithms with rigorous convergence guarantees. 
It is shown in \citep{fang2018spider,zhou2019lower} that there is still a gap between the upper-bound complexity in state-of-the-art methods and the lower-bound worst-case complexity for the nonconvex problem \eqref{eq:finite_sum} 	under standard smoothness assumption.
Motivated by this fact, we make an attempt to develop a new algorithmic framework that can reduce  and at least  nearly close this gap in the composite finite-sum setting \eqref{eq:finite_sum}.
In addition to the best-known complexity bounds, we expect to \nhanp{design} practical algorithms \nhanp{advancing} beyond existing methods by providing an adaptive rule to update step-sizes with rigorous complexity analysis.
Our algorithms rely on a recent biased stochastic estimator for the objective gradient, called SARAH, introduced in \citep{Nguyen2017_sarah} for convex problems.

\beforepara
\paragraph{Related work:}
In the nonconvex case, both problems \eqref{eq:sopt_prob} and \eqref{eq:finite_sum} have been intensively studied in recent years with a  vast number of research papers. 
While numerical algorithms for solving the non-composite setting, i.e. $\psi = 0$, are well-developed and have received considerable attention \citep{allen2017natasha2,allen2018neon2,SVRG++,fang2018spider,lei2017non,Nguyen2017_sarahnonconvex,Nguyen2018_iSARAH,Nguyen2019_SARAH,reddi2016proximal,zhou2018stochastic}, methods for composite setting remain limited \citep{reddi2016proximal,wang2018spiderboost}.
In terms of algorithms, \citep{reddi2016proximal} studies a non-composite finite-sum problem as a special case of \eqref{eq:finite_sum} using SVRG estimator from \citep{johnson2013accelerating}.
Additionally, they extend their method to the composite setting by simply applying the proximal operator of $\psi$ as in the well-known forward-backward scheme.
Another related work using SVRG estimator can be found in \citep{li2018simple}.
These algorithms have some limitation as  will be discussed later.
The same technique was applied  in \citep{wang2018spiderboost} to develop other variants for both \eqref{eq:sopt_prob} and \eqref{eq:finite_sum}, but using the SARAH estimator from \citep{Nguyen2017_sarah}.
The authors derive a  large constant step-size, but at the same time control mini-batch size to achieve desired complexity bounds.
Consequently, it has an essential limitation as will also be discussed in Subsection~\ref{subsubsec:mini_batch_step_size}.
Both algorithms achieve the best-known complexity bounds for solving \eqref{eq:sopt_prob} and \eqref{eq:finite_sum}.
In \citep{reddi2016stochastic}, the authors propose a stochastic Frank-Wolfe method that can handle constraints as special cases of \eqref{eq:finite_sum}.
Recently, a stochastic  variance reduction method with momentum was studied in \citep{zhou2019momentum} for solving \eqref{eq:finite_sum} which can be viewed as a modification of SpiderBoost in \citep{wang2018spiderboost}.

Our algorithm remains a variance reduction stochastic method, but it is different from these works at two major points: an additional averaging step and two  different  step-sizes.
Having two step-sizes allows us to flexibly trade-off them and develop an adaptive update rule.
Note that our averaging step looks similar to the robust stochastic gradient method in \citep{Nemirovski2009a}, but fundamentally different since it evaluates the proximal step at the averaging point.
In fact, it is closely related to averaged fixed-point schemes in the literature, see, e.g. \citep{Bauschke2011}. 

In terms of theory, many researchers have focused on theoretical aspects of existing algorithms.
For example, \citep{ghadimi2013stochastic} appears to be one of the first pioneering works studying convergence rates of stochastic gradient descent-type methods for nonconvex and non-composite finite-sum problems.
They later extend it to the composite setting in \citep{ghadimi2016mini}.
\citep{wang2018spiderboost}  also investigate the gradient dominance case, and \citep{karimi2016linear} consider both finite-sum and composite finite-sum under different assumptions.

Whereas many researchers have been trying to improve complexity upper bounds of stochastic first-order methods using different techniques \citep{allen2017natasha2,allen2018neon2,SVRG++,fang2018spider},
other researchers attempt to  construct examples for lower-bound complexity estimates.
In the convex case, there exist numerous research papers including \citep{agarwal2010information,Nemirovskii1983,Nesterov2004}.
In \citep{fang2018spider,zhou2019lower}, the authors have constructed a lower-bound complexity for nonconvex finite-sum problem covered by \eqref{eq:finite_sum}.
They showed that the lower-bound complexity for any stochastic gradient method using only smoothness assumption to achieve an $\varepsilon$-stationary point in expectation  is $\Omega\left({n^{1/2}\varepsilon^{-2}}\right)$ given that the number of objective components $n$ does not exceed $\BigO{\varepsilon^{-4}}$. 

For the expectation problem \eqref{eq:sopt_prob}, the best-known complexity bound to achieve an $\varepsilon$-stationary point in expectation is $\BigO{\sigma\varepsilon^{-3} + \sigma^2\varepsilon^{-2}}$ as shown in \citep{fang2018spider,wang2018spiderboost}, where $\sigma$ is an upper bound of the variance (see Assumption~\ref{as:A3}).
Unfortunately, we have not seen any lower-bound complexity for the nonconvex setting of \eqref{eq:sopt_prob} under standard assumptions in the literature.

\beforepara
\paragraph{Our approach and contribution:}
We exploit the SARAH estimator, a biased stochastic recursive gradient estimator, in \citep{Nguyen2017_sarah}, to design new proximal variance reduction stochastic gradient algorithms to solve both composite  expectation and finite-sum problems  \eqref{eq:sopt_prob} and \eqref{eq:finite_sum}.
The SARAH algorithm is simply a double-loop stochastic gradient method with a flavor of SVRG \citep{johnson2013accelerating}, but using a novel \nhanp{biased} estimator that is different from SVRG. 
SARAH is a recursive method as SAGA \citep{SAGA}, but can avoid the major issue of storing gradients as in SAGA.
Our method will rely on the SARAH estimator as in SPIDER and SpiderBoost combining with an averaging proximal-gradient scheme to solve both \eqref{eq:sopt_prob} and \eqref{eq:finite_sum}.

The contribution of this paper is a new algorithmic framework that covers different variants with constant and adaptive step-sizes, single sample and mini-batch, and achieves best-known theoretical complexity bounds.
More specifically, our main contribution can be summarized as follows:

\begin{itemize}
\vspace{-1.2ex}
\item[$\mathrm{(a)}$]\textbf{Composite settings:}
We propose a general stochastic variance reduction framework relying on the SARAH estimator to solve both expectation and finite-sum problems \eqref{eq:sopt_prob} and \eqref{eq:finite_sum} in composite settings.
We analyze our framework to design appropriate constant step-sizes instead of diminishing step-sizes as in standard stochastic gradient descent methods.
As usual, the algorithm has double loops, where the outer loop can either take full gradient or mini-batch to reduce computational burden in large-scale and expectation settings.
The inner loop can work with single sample or a broad range of mini-batch sizes.

\vspace{-1.2ex}
\item[$\mathrm{(b)}$]\textbf{Best-known complexity:}
In the finite-sum setting \eqref{eq:finite_sum}, our method achieves $\BigO{n + n^{1/2}\varepsilon^{-2}}$ complexity bound to attain an $\varepsilon$-stationary point in expectation under only the smoothness of $f$.
This complexity matches the lower-bound worst-case  complexity in \citep{fang2018spider,zhou2019lower} up to a constant factor when $n \leq \BigO{\varepsilon^{-4}}$.
In the expectation setting \eqref{eq:sopt_prob}, our algorithm requires $\BigO{\sigma^2\varepsilon^{-2} + \sigma\varepsilon^{-3}}$ first-order oracle calls of $f$ to achieve an $\varepsilon$-stationary point in expectation under only the smoothness of $f$ and bounded variance $\sigma^2$.
To the best of our knowledge, this is the best-known complexity so far for  \eqref{eq:sopt_prob} under standard assumptions in both the single sample and mini-batch cases.

\vspace{-1.2ex}
\item[$\mathrm{(c)}$]\textbf{Adaptive step-sizes:}
Apart from constant step-size algorithms, we also specify our framework to obtain adaptive step-size variants for both composite and non-composite settings in both single sample and mini-batch cases.
Our adaptive step-sizes are increasing along the inner iterations rather than diminishing as in stochastic proximal gradient descent methods.
The adaptive variants often outperform the constant step-sizes schemes in several test cases.
\vspace{-1.5ex}
\end{itemize}

Our result covers the non-composite setting in the finite-sum case \citep{Nguyen2019_SARAH}, and matches the best-known complexity in \citep{fang2018spider,wang2018spiderboost} for both problems \eqref{eq:sopt_prob} and \eqref{eq:finite_sum}. 
Since the composite setting covers a broader class of nonconvex problems including convex constraints, we believe that our method has better chance to handle new applications than non-composite methods.
It also allows one to deal with composite problems under different type of regularizers such as sparsity or constraints on weights as in neural network training applications. 

%%%% Comparison:
\beforepara
\paragraph{Comparison:}
Hitherto, we have found three different variance reduction algorithms of the stochastic proximal gradient method for nonconvex problems that are most related to our work: proximal SVRG (called ProxSVRG) in  \citep{reddi2016proximal}, ProxSVRG+ in \citep{li2018simple}, and ProxSpiderBoost in \citep{wang2018spiderboost}.
Other methods such as proximal stochastic gradient descent (ProxSGD) scheme \citep{ghadimi2016mini}, ProxSAGA in \citep{reddi2016proximal}, and Natasha variants in \citep{allen2017natasha2} are quite different and already intensively compared in previous works \citep{li2018simple,reddi2016proximal,wang2018spiderboost}, and hence we do not include them here.

In terms of theory, Table~\ref{tbl:SFO_compare} compares different methods for solving \eqref{eq:sopt_prob} and \eqref{eq:finite_sum} regarding the stochastic first-order oracle calls (SFO), the applicability to finite-sum and/or expectation and composite settings, step-sizes, and the use of adaptive step-sizes.

%%%% Table 1. Comparison ...
\begin{table}[hpt!]
\newcommand{\cellb}[1]{{\!\!}{\color{blue}#1}{\!\!}}
\newcommand{\cellr}[1]{{\!\!}{\color{red}#1}{\!\!}}
\newcommand{\cell}[1]{{\!\!\!}#1{\!\!\!}}
\begin{center}
\resizebox{\textwidth}{!}{
\begin{tabular}{l ccccc}\toprule
\cell{~~~~~~~~~\bf Algorithms} & \cell{\bf Finite-sum} & \cell{\bf Expectation} & \cell{\bf Composite} & \cell{\bf Step-size} & \cell{{\!\!\!}\bf Adaptive step-size} \\ \toprule
\cell{GD \citep{Nesterov2004}} & \cell{$\BigO{n\varepsilon^{-2}}$} & NA & \cellb{\done} & $\BigO{L^{-1}}$ & \cellb{Yes}  \\ \midrule
\cell{SGD \citep{ghadimi2013stochastic}} & NA & \cell{$\BigO{\sigma^2\varepsilon^{-4}}$} & \cellb{\done}  & $\BigO{L^{-1}}$ & \cellb{Yes}  \\ \midrule
\cell{SVRG/SAGA \citep{reddi2016proximal}} & $\BigO{n + n^{2/3}\varepsilon^{-2}}$ & NA & \cellb{\done}  & \cell{$\BigO{(nL)^{-1}}\to\BigO{L^{-1}}$} & \cellr{No}  \\ \midrule
\cell{SVRG+ \citep{li2018simple}} & \cell{$\BigO{n+n^{2/3}\varepsilon^{-2}}$} & \cell{$\BigO{\sigma^2\varepsilon^{-10/3}}$} & \cellb{\done}  & \cell{$\BigO{(nL)^{-1}}\to\BigO{L^{-1}}$} & \cellr{No} \\ \midrule
\cell{SCSG \citep{lei2017non}} & \cell{$\BigO{n + n^{2/3}\varepsilon^{-2}}$} &\cell{ $\BigO{\sigma^2\varepsilon^{-2}+\sigma\varepsilon^{-10/3}}$} & \cellr{\wontfix}  & \cell{$\BigO{L^{-1}(n^{-2/3}\wedge \varepsilon^{4/3})}$} & \cellr{No} \\ \midrule
\cell{SNVRG \citep{zhou2018stochastic}} & \cell{$\BigO{(n+n^{1/2}\varepsilon^{-2})\log(n)}$} & $\BigO{(\sigma^2\varepsilon^{-2}+\sigma\varepsilon^{-3})\log(\varepsilon^{-1})}$ & \cellr{\wontfix}  & $\BigO{L^{-1}q^{-1/2}}$ & \cellr{No}  \\ \midrule
\cell{SPIDER \citep{fang2018spider}} & \cell{$\BigO{n+n^{1/2}\varepsilon^{-2}}$} & \cell{$\BigO{\sigma^2\varepsilon^{-2}+\sigma\varepsilon^{-3}}$} & \cellr{\wontfix}  & \cell{$\BigO{L^{-1}\varepsilon}$} & \cellb{Yes} \\ \midrule
\cell{SpiderBoost \citep{wang2018spiderboost}} & \cell{$\BigO{n+n^{1/2}\varepsilon^{-2}}$} & \cell{$\BigO{\sigma^2\varepsilon^{-2}+\sigma\varepsilon^{-3}}$} & \cellb{\done}  & \cell{$\BigO{L^{-1}}$} & \cellr{No} \\ \midrule
\cellb{\bf ProxSARAH (This work)} & \cellb{$\BigO{n+{n^{1/2}}\varepsilon^{-2}}$} & \cellb{$\BigO{\sigma^2\varepsilon^{-2} + \sigma \varepsilon^{-3}}$} & \cellb{\done}  & \cellb{$\BigO{L^{-1}m^{-1/2}} \to \BigO{L^{-1}}$} &  \cellb{Yes} \\ 
\bottomrule
\end{tabular}}
\caption{Comparison of results on SFO $($stochastic first-order oracle$)$  complexity for nonsmooth nonconvex optimization $($both non-composite and composite case$)$.
{\small
Here, $m$ is the number of inner iterations $($epoch length$)$ and $\sigma$ is the variance in Assumption~\ref{as:A3}.
% and ``required'' means that the algorithm uses  \nhanp{specified/predefined mini-batch size} to achieve the best complexity.
Note that all the complexity bounds here must depend on the Lipschitz constant $L$ of the smooth components and $F(\widetilde{w}^0) - F^{\star}$, the difference between the initial objective value $F(\widetilde{w}^0)$ and the lower-bound $F^{\star}$.
For the sake of presentation, we assume that $L = \BigO{1}$ and ignore these quantities in the complexity bounds
}
}\label{tbl:SFO_compare}
\vspace{-4ex}
\end{center}
\end{table}

Now, let us compare in detail our algorithms and four methods: ProxSVRG, ProxSVRG+, SPIDER, and ProxSpiderBoost for solving \eqref{eq:sopt_prob} and \eqref{eq:finite_sum}, or their special cases.

\beforepara
\paragraph{\textbf{Assumptions:}}
In the finite-sum setting \eqref{eq:finite_sum}, ProxSVRG, ProxSVRG+, and ProxSpiderBoost all use the smoothness of each component $f_i$ in \eqref{eq:finite_sum}, which is stronger than the average smoothness in Assumption~\ref{as:A2} stated below.
They did not consider \eqref{eq:finite_sum} under Assumption~\ref{as:A2}.
%It is not clear if these methods can still achieve the same complexity bounds as they stated under Assumption~\ref{as:A2}.

\beforepara
\paragraph{\textbf{Single sample for the finite-sum case:}}
The performance of gradient descent-type algorithms crucially depends on the step-size (i.e., learning rate).
Let us make a comparison between different methods in terms of step-size for single sample case, and the corresponding complexity bound.
\begin{itemize}
\vspace{-1.5ex}
\item As shown in  \citep[Theorem 1]{reddi2016proximal}, in the single sample case, i.e. the mini-batch size of the inner loop $\hat{b} = 1$,  ProxSVRG for solving \eqref{eq:finite_sum} has a small step-size $\eta = \frac{1}{3Ln}$, and its corresponding complexity is $\BigO{n\varepsilon^{-2}}$, see \citep[Corollary 1]{reddi2016proximal}, which is the same as in standard proximal gradient methods.
\item ProxSVRG+ in \citep[Theorem 3]{li2018simple} is a variant of ProxSVRG, and in the single sample case, it uses a different step-size $\eta =  \min\set{\frac{1}{6L}, \frac{1}{6mL}}$. 
This step-size is only better than that of ProxSVRG if $2m < n$. 
With this step-size, the complexity of ProxSVRG+ remains $\BigO{n^{2/3}\varepsilon^{-2}}$ as in ProxSVRG.

\vspace{-1.5ex}
\item In the non-composite case, SPIDER \citep{fang2018spider} relies on an adaptive step-size $\eta_t := \min\set{\frac{\varepsilon}{L\Vert v_t\Vert\sqrt{n}}, \frac{1}{2L\sqrt{n}}}$, where $v_t$ is the SARAH stochastic estimator. 
Clearly, this step-size is very small if the target accuracy $\varepsilon$ is small, and/or $\Vert v_t\Vert$ is large.
However, SPIDER achieves $\BigO{n+ n^{1/2}\varepsilon^{-2}}$ complexity bound, which is nearly optimal.
Note that this step-size is problem dependent since it depends on $v_t$.
We also emphasize that SPIDER did not consider the composite problems.

\vspace{-1.5ex}
\item In our constant step-size ProxSARAH variants, we use two step-sizes: averaging step-size $\gamma = \frac{\sqrt{2}}{\sqrt{3m}L}$ and proximal-gradient step-size $\eta = \frac{2\sqrt{3m}}{4\sqrt{3m} + \sqrt{2}}$, and their product presents a combined step-size, which is $\hat{\eta} := \gamma\eta = \frac{2}{L(4\sqrt{3m} + \sqrt{2})}$ (see \eqref{eq:grad_representation} for our definition of step-size).
Clearly, our step-size $\hat{\eta}$ is much larger than that of both ProxSVRG and ProxSVRG+.
It can be larger than that of SPIDER if $\varepsilon$ is small and $\Vert v_t\Vert$ is large.
With these step-sizes, our complexity bound is $\BigO{n+ n^{1/2}\varepsilon^{-2}}$, and if $\varepsilon \leq \BigO{n^{-1/4}}$, then it reduces to $\BigO{n^{1/2}\varepsilon^{-2}}$, which is also nearly optimal.

\vspace{-1.5ex}
\item 
As we can observe from Algorithm~\ref{alg:prox_sarah} in the sequel, the number of proximal operator calls in our method remains the same as in ProxSVRG and ProxSVRG+.
\vspace{-1.5ex}
\end{itemize}

\vspace{-1.25ex}
\paragraph{\textbf{Mini-batch for the finite-sum case:}}
Now, we consider the case of using mini-batch.
\begin{itemize}
\vspace{-1.5ex}
\item As indicated in \citep[Theorem 2]{reddi2016proximal}, if we choose the batch size $\hat{b} = \lfloor n^{2/3}\rfloor$ and $m = \lfloor n^{1/3}\rfloor$, then the step-size $\eta$ can be chosen as $\eta = \frac{1}{3L}$, and its complexity is improved up to $\BigO{n + n^{2/3}\varepsilon^{-2}}$ for ProxSVRG.
However, the mini-batch size $n^{2/3}$ is close to the full dataset $n$.

\vspace{-1.5ex}
\item For ProxSVRG+ in \citep{li2018simple}, based on Theorem 1, we need to set \nhanp{$\hat{b} =  \lfloor n^{2/3} \rfloor$}  and $m = \lfloor\sqrt{\hat{b}}\rfloor = \lfloor n^{1/3}\rfloor$ to obtain the best complexity bound for this method, which is $\BigO{n+ n^{2/3}\varepsilon^{-2}}$.
Nevertheless, its step-size is $\eta = \frac{1}{6L}$, which is twice smaller than that of ProxSVRG.
In addition, ProxSVRG requires the bounded variance assumption for \eqref{eq:finite_sum}.

\vspace{-1.5ex}
\item For SPIDER, again in the non-composite setting, if we choose the batch-size $\hat{b} = \lfloor n^{1/2}\rfloor$, then its step-size is $\eta_t := \min\set{\frac{\varepsilon}{L\Vert v_t\Vert}, \frac{1}{2L}}$.
In addition, SPIDER limits the batch size $\hat{b}$ in the range of $[1, n^{1/2}]$, and did not consider larger mini-batch sizes.

\vspace{-1.5ex}
\item For SpiderBoost in \citep{wang2018spiderboost}, it requires to  properly set mini-batch size to achieve $\BigO{n + n^{1/2}\varepsilon^{-2}}$ complexity for solving \eqref{eq:finite_sum}.
More precisely, from \citep[Theorem 1]{wang2018spiderboost}, we can see that one needs to set $m = \lfloor\sqrt{n}\rfloor$ and $\hat{b} = \lfloor\sqrt{n}\rfloor$ to achieve such a complexity.
This mini-batch size can be large if $n$ is large, and less flexible to adjust the performance of the algorithm. 
Unfortunately, ProxSpiderBoost does not have theoretical guarantee for the single sample case.

\vspace{-1.5ex}
\item In our methods, it is flexible to choose the epoch length $m$ and the batch size $\hat{b}$ such that we can obtain different step-sizes and complexity bounds.
Our batch-size $\hat{b}$ can be any value in $[1, n-1]$ for \eqref{eq:finite_sum}.
Given $\hat{b} \in [1, \sqrt{n}]$, we can properly choose $m = \BigO{n/\hat{b}}$  to obtain the best-known complexity bound $\BigO{n+ n^{1/2}\varepsilon^{-2}}$ when $n > \BigO{\varepsilon^{-4}}$ and $\BigO{n^{1/2}\varepsilon^{-2}}$, otherwise. 
More details can be found in Subsection~\ref{subsubsec:mini_batch_step_size}.
\vspace{-1.5ex}
\end{itemize}

\vspace{-1.25ex}
\paragraph{\textbf{Online or expectation problems:}}
For online or expectation problems, a mini-batch is required to evaluate snapshot gradient estimators for the outer loop.
\begin{itemize}
\vspace{-1.5ex}
\item In the online or expectation case \eqref{eq:sopt_prob}, SPIDER in \citep[Theorem 1]{fang2018spider} achieves an $\BigO{\sigma \varepsilon^{-3} + \sigma^2 \varepsilon^{-2}}$ complexity.
In the single sample case, SPIDER's step-size becomes $\eta_t := \min\set{\frac{\varepsilon^2}{2\sigma L\Vert v_t\Vert}, \frac{\varepsilon}{4\sigma L}}$, which can be very small, and depends on $v_t$ and $\sigma$.
Note that $\sigma$ is often unknown or hard to estimate.
Moreover, in early iterations, $\Vert v_t\Vert$ is often large potentially making this method slow.

\vspace{-1.5ex}
\item ProxSpiderBoost in \citep{wang2018spiderboost} achieves the same complexity bound as SPIDER for the composite problem \eqref{eq:sopt_prob}, but requires to set the mini-batch for both outer and inner loops.
The size of these mini-batches has to be fixed a priori in order to use a constant step-size, which is certainly less flexible.
The total complexity of this method is $\BigO{\sigma \varepsilon^{-3} + \sigma^2 \varepsilon^{-2}}$.

\vspace{-1.5ex}
\item 
As shown in Theorem~\ref{th:convergence_composite_expectation}, our complexity is $\BigO{\sigma \varepsilon^{-3}}$ given that $\sigma \leq \BigO{\varepsilon^{-1}}$.
Otherwise, it is $\BigO{\sigma \varepsilon^{-3} + \sigma^2 \varepsilon^{-2}}$, which is the same as in ProxSpiderBoost.
Note that our complexity can \nhanp{be achieved} for both single sample and a wide range of mini-batch sizes as opposed to a predefined mini-batch size of ProxSpiderBoost.
\vspace{-1.5ex}
\end{itemize}

From an algorithmic point of view, our method is fundamentally different from existing methods due to its averaging step and large step-sizes in the composite settings.
Moreover, our methods have more chance to improve \nhanp{the} performance due to the use of adaptive step-sizes and an additional damped step-size $\gamma_t$, and \nhanp{the} flexibility to choose  the epoch length $m$, the inner mini-batch size $\hat{b}$, and the snapshot batch size $b_s$.

%% Paper roadmap.
\beforepara
\paragraph{Paper organization:}
The rest of this paper is organized as follows.
Section~\ref{sec:assumption_optimality} discusses the fundamental assumptions and optimality conditions.
Section~\ref{sec:alg_section} presents the main algorithmic framework and its convergence results for two settings.
Section~\ref{sec:extensions} considers extensions and special cases of our algorithms.
Section~\ref{sec:num_experiments} provides some numerical examples to verify our methods and compare them with existing state-of-the-arts.

%%%%%%%%%%%%%%%%%%%%%%%%%%%%%
%%%%% 2. Fundamental assumptions and Optimality
\beforesec
\section{Mathematical tools and preliminary results}\label{sec:assumption_optimality}
\aftersec
Firstly, we recall some basic notation and concepts in optimization, which  can be found in  \citep{Bauschke2011,Nesterov2004}.
Next, we state our blanket assumptions and discuss the optimality condition of \eqref{eq:sopt_prob} and \eqref{eq:finite_sum}.
Finally, we provide  preliminary results \nhanp{needed} in the sequel.

%%% 2.1. Basic notation and concepts
\beforesubsec
\subsection{Basic notation and concepts}
\aftersubsec
We work with finite dimensional spaces, $\R^d$, equipped with standard inner product $\iprods{\cdot,\cdot}$ and Euclidean norm $\norms{\cdot}$.
Given a function $f : \R^d\to\Rext$, we use $\dom{f} := \set{w\in\R^d \mid f(w) < +\infty}$ to denote its (effective) domain.
If $f$ is proper, closed, and convex, $\partial{f}(w) := \set{v \in \R^d \mid f(z) \geq f(w) + \iprods{v, z - w},~~\forall z\in\dom{f}}$ denotes its subdifferential at $w$, and $\prox_f(w) := \mathrm{arg}\min_{z}\set{ f(z) + (1/2)\norms{z-w}^2 }$ denotes its proximal operator. 
Note that if $f$ is the indicator of a nonempty, closed, and convex set $\Xc$, i.e. $f(w) = \delta_{\Xc}(w)$, then $\prox_f(\cdot) = \proj_{\Xc}(\cdot)$, the projection of $w$ onto $\Xc$.
Any element $\nabla{f}(w)$ of $\partial{f}(w)$ is called a subgradient of $f$ at $w$. If $f$ is differentiable at $w$, then $\partial{f}(w) = \set{\nabla{f}(w)}$, the gradient of $f$ at $w$.
A continuous differentiable function $f : \R^d\to\R$ is said to be $L_f$-smooth if $\nabla{f}$ is Lipschitz continuous on its domain, i.e. $\norms{\nabla{f}(w) - \nabla{f}(z)} \leq L_f\norms{w - z}$ for $w,z\in\dom{f}$.
We use $\Ub_p(S)$ to denote a finite set $S := \set{s_1, s_2, \cdots, s_n}$ equipped with a probability distribution $p$ over $S$. If $p$ is uniform, then we simply use $\Ub(S)$.
For any real number $a$, $\lfloor a\rfloor$ denotes the largest integer less than or equal to $a$.
We use $[n]$ to denote the set $\set{1,2,\cdots, n}$.

%%% 2.1. Fundamental assumptions
\beforesubsec
\subsection{Fundamental assumptions}
\aftersubsec
To develop numerical methods for solving \eqref{eq:sopt_prob} and \eqref{eq:finite_sum}, we rely on some basic assumptions usually used in stochastic optimization methods.
%%% Assumption A.1.
\begin{assumption}[Bounded from below]\label{as:A1}
Both problems \eqref{eq:sopt_prob} and \eqref{eq:finite_sum} are bounded from below.
That is $F^{\star} := \inf_{w\in\R^d}F(w) > -\infty$.
Moreover, $\dom{F} := \dom{f}\cap\dom{\psi} \neq\emptyset$.
\end{assumption}
This assumption usually holds in practice since $f$ often represents a loss function which is nonnegative or bounded from below. 
In addition, the regularizer $\psi$ is also nonnegative or bounded from below, and its domain intersects $\dom{f}$.

Our next assumption is the smoothness of $f$ with respect to the argument $w$.
%%% Assumption A.2.
\begin{assumption}[$L$-average smoothness]\label{as:A2}
In the expectation setting \eqref{eq:sopt_prob}, for any realization of $\xi\in\Omega$, $f(\cdot; \xi)$ is $L$-smooth $($on average$)$, i.e. $f(\cdot;\xi)$ is continuously differentiable and its gradient $\nabla_w{f}(\cdot;\xi)$ is Lipschitz continuous with the same Lipschitz constant $L \in (0, +\infty)$, i.e.:
\begin{equation}\label{eq:L_smooth}
\Exps{\xi}{\norms{\nabla_w{f}(w;\xi) - \nabla_w{f}(\hat{w};\xi)}^2} \leq L^2\norms{w - \hat{w}}^2,~~w,\hat{w}\in\dom{f}.
\end{equation}
In the finite-sum setting \eqref{eq:finite_sum}, the condition \eqref{eq:L_smooth} reduces to
\begin{equation}\label{eq:L_smooth_fi}
\frac{1}{n}\sum_{i=1}^n\norms{\nabla{f_i}(w) - \nabla{f_i}(\hat{w})}^2 \leq L^2\norms{w - \hat{w}}^2,~~w,\hat{w}\in\dom{f},~~i=1,\cdots, n.
\end{equation}
\end{assumption}
We can write \eqref{eq:L_smooth_fi} as $\Exps{i}{\norms{\nabla{f_i}(w) - \nabla{f_i}(\hat{w})}^2} \leq L^2\norms{w - \hat{w}}^2$.
Note that \eqref{eq:L_smooth_fi} is weaker than assuming that each component $f_i$ is $L_i$-smooth, i.e., $\Vert \nabla{f_i}(w) - \nabla{f_i}(\hat{w})\Vert \leq L_i\Vert w - \hat{w}\Vert$ for all $w, \hat{w}\in\dom{f}$.
Indeed, the individual $L_i$-smoothness implies \eqref{eq:L_smooth_fi} with $L^2 := \frac{1}{n}\sum_{i=1}^nL_i^2$.
Conversely, if \eqref{eq:L_smooth_fi} holds, then $\Vert \nabla{f_i}(w) - \nabla{f_i}(\hat{w})\Vert^2 \leq \sum_{i=1}\Vert \nabla{f_i}(w) - \nabla{f_i}(\hat{w})\Vert^2 \leq nL^2\Vert w - \hat{w}\Vert^2$ for $i=1,\cdots, n$.
Therefore, each component $f_i$ is $\sqrt{n}L$-smooth, which is larger than \eqref{eq:L_smooth_fi} within a factor of  $\sqrt{n}$ in the worst-case. 
We emphasize that ProxSVRG, ProxSVRG+, and ProxSpiderBoost all require the $L$-smoothness of each component $f_i$ in \eqref{eq:finite_sum}.

It is well-known that the $L$-smooth condition leads to the following bound
\begin{equation}\label{eq:upper_bound}
\Exps{\xi}{f(\hat{w};\xi)} \leq \Exps{\xi}{f(w;\xi)} + \Exps{\xi}{\iprods{\nabla_w{f}(w;\xi), \hat{w} - w}} + \frac{L}{2}\norms{\hat{w} - w}^2,~~w,\hat{w}\in\dom{f}.
\end{equation}
Indeed, from \eqref{eq:L_smooth}, we have
\begin{equation*}
\begin{array}{ll}
\norms{\nabla{f}(w) - \nabla{f}(\hat{w})}^2 &= \norms{\Exps{\xi}{\nabla_w{f}(w;\xi) - \nabla_w{f}(\hat{w};\xi)}}^2 \vspace{1ex}\\
& \leq \Exps{\xi}{\norms{\nabla_w{f}(w;\xi) - \nabla_w{f}(\hat{w};\xi)}^2} \vspace{1ex}\\
& \leq L^2\norms{w - \hat{w}}^2,
\end{array}
\end{equation*}
which shows that $\norms{\nabla{f}(w) - \nabla{f}(\hat{w})} \leq L\norms{w - \hat{w}}$.
Hence, using either \eqref{eq:L_smooth} or \eqref{eq:L_smooth_fi}, we get
\begin{equation}\label{eq:upper_bound_fw}
f(\hat{w}) \leq f(w) + \iprods{\nabla{f}(w), \hat{w} - w} + \frac{L}{2}\norms{\hat{w} - w}^2,~~w,\hat{w}\in\dom{f}.
\end{equation}
In the expectation setting \eqref{eq:sopt_prob}, we need the following bounded variance condition:

%%% Assumption 3.
\begin{assumption}[Bounded variance]\label{as:A3}
For the expectation problem \eqref{eq:sopt_prob}, there exists a uniform constant $\sigma \in (0, +\infty)$ such that
\begin{equation}\label{eq:bounded_variance}
\Exps{\xi}{\norms{\nabla_w{f}(w;\xi) - \nabla{f}(w)}^2} \leq \sigma^2, ~~\forall w\in\R^d.
\end{equation}
\end{assumption}
This assumption is standard in stochastic optimization and often required in almost any solution method for solving \eqref{eq:sopt_prob}, see, e.g. \citep{ghadimi2013stochastic}.
For problem \eqref{eq:finite_sum}, if $n$ is extremely large, passing over $n$ data points is exhaustive or impossible.
We refer to this case as the online case mentioned in \citep{fang2018spider}, and can be cast into Assumption~\ref{as:A3}.
Therefore, we do not consider this case separately.
However, our theory and algorithms developed in this paper do apply to such a setting.

%%% 2.2. Optimality condition.
\beforesubsec
\subsection{Optimality conditions}
\aftersubsec
Under Assumption~\ref{as:A1}, we have $\dom{f}\cap\dom{\psi}\neq\emptyset$.
When $f(\cdot;\xi)$ is nonconvex in $w$, the first order optimality condition of \eqref{eq:sopt_prob} can be stated as
\begin{equation}\label{eq:sopt_opt_cond}
0 \in \partial{F}(w^{\star}) \equiv \nabla{f}(w^{\star}) + \partial{\psi}(w^{\star}) \equiv \Exps{\xi}{\nabla_wf(w^{\star};\xi)} + \partial{\psi}(w^{\star}).
\end{equation}
Here, $w^{\star}$ is called a stationary point of $F$. 
We denote $\Sc^{\star}$ the set of all stationary points.
The condition \eqref{eq:sopt_opt_cond} is called the first-order optimality condition, and also holds for \eqref{eq:finite_sum}.

Since $\psi$ is proper, closed, and convex, its proximal operator $\prox_{\eta\psi}$ satisfies the nonexpansiveness, i.e.  $\norms{\prox_{\eta\psi}(w) - \prox_{\eta\psi}(z)} \leq \norms{w - z}$ for all $w, z\in\R^d$.

Now, for any fixed $\eta > 0$, we define the following quantity
\begin{equation}\label{eq:gradient_mapping}
G_{\eta}(w) := \frac{1}{\eta}\big(w - \prox_{\eta\psi}(w - \eta\nabla{f}(w))\big).
\end{equation}
This quantity is called the gradient mapping of $F$ \citep{Nesterov2004}.
Indeed, if $\psi\equiv 0$, then $G_{\eta}(w) \equiv \nabla{f}(w)$, which is exactly the gradient of $f$.
By using $G_{\eta}(\cdot)$, the optimality condition \eqref{eq:sopt_opt_cond} can be equivalently written as
\begin{equation}\label{eq:sopt_opt_fixed_point}
\norms{G_{\eta}(w^{\star})}^2 = 0.
\end{equation}
If we apply gradient-type methods to solve \eqref{eq:sopt_prob} or \eqref{eq:finite_sum}, then we can only aim at finding an $\varepsilon$-approximate stationary point $\widetilde{w}_T$ to $w^{\star}$ in \eqref{eq:sopt_opt_fixed_point} after at most $T$ iterations within a given accuracy $\varepsilon > 0$, i.e.:
\begin{equation}\label{eq:approximate_opt_cond}
\Exp{\norms{G_{\eta}(\widetilde{w}_T)}^2} \leq \varepsilon^2.
\end{equation}
The condition \eqref{eq:approximate_opt_cond} is standard in stochastic nonconvex optimization methods.
Stronger results such as approximate second-order optimality or strictly local minimum require additional assumptions and more sophisticated optimization methods such as cubic regularized Newton-type schemes, see, e.g., \citep{Nesterov2006a}.

%%%%%%%%%%%%%%%%%%%%%
%%% 2.4. Stochastic gradient estimators
\beforesubsec
\subsection{Stochastic gradient estimators}
\aftersubsec
One key step to design a stochastic gradient method for \eqref{eq:sopt_prob} or \eqref{eq:finite_sum} is to query  an estimator for the gradient $\nabla{f}(w)$ at any $w$.
Let us recall some existing stochastic estimators.

%% 2.4.a. Single sample estimators
\beforepara
\paragraph{Single sample estimators:}
A simple estimator of $\nabla{f}(w)$ can be computed as follows:
\begin{equation}\label{eq:stochastic_estimator}
\widetilde{\nabla}f(w_t) := \nabla_w{f}(w_t;\xi_t),
\end{equation}
where $\xi_t$ is a realization of $\xi$.
This estimator is unbiased, i.e., $\cExp{\widetilde{\nabla}f(w_t)}{\Fc_t} = \nabla{f}(w_t)$, but its variance is fixed for any $w_t$, where $\Fc_t$ is the history of randomness collected up to the $t$-th iteration, i.e.:
\begin{equation}\label{eq:history_Ft}
\Fc_t := \sigma\big(w_0, w_1,\cdots, w_{t} \big).
\end{equation}
This is a $\sigma$-field generated by random variables $\set{w_0, w_1,\cdots, w_{t}}$.
In the finite-sum setting \eqref{eq:finite_sum}, we have $\widetilde{\nabla}f(w_t) := \nabla{f}_{i_t}(w_t)$, where $i_t \sim \Ub([n])$ with $[n] := \set{1,2,\cdots, n}$.

In recent years, there has been  huge interest in designing stochastic estimators with variance reduction properties. 
The first variance reduction method was perhaps proposed in \citep{schmidt2017minimizing} since 2013, and then in \citep{SAGA} for convex optimization.
However, the most well-known method is SVRG introduced by Johnson and Zhang in \citep{johnson2013accelerating} that works for both convex and nonconvex problems.
The SVRG estimator for $\nabla{f}$ in \eqref{eq:finite_sum} is given as
\begin{equation}\label{eq:svrg_estimator}
\widetilde{\nabla}f(w_t) := \nabla{f}(\widetilde{w}) + \nabla{f_{i_t}}(w_t) - \nabla{f_{i_t}}(\widetilde{w}), 
\end{equation}
where $\nabla{f}(\widetilde{w})$ is the full gradient of $f$ at a snapshot point $\widetilde{w}$, and $i_t$ is a uniformly random index in $[n]$.
It is clear that $\cExp{\widetilde{\nabla}f(w_t)}{\Fc_t} = \nabla{f}(w_t)$, which shows that $\widetilde{\nabla}f(w_t)$ is an unbiased estimator of $\nabla{f}(w_t)$.
Moreover, its variance is reduced along the snapshots.

Our methods rely on the SARAH estimator introduced in \citep{Nguyen2017_sarah} for the non-composite convex problem instances of \eqref{eq:finite_sum}.
We instead consider it in a more general setting to cover both \eqref{eq:finite_sum} and \eqref{eq:sopt_prob}, which is defined as follows:
\begin{equation}\label{eq:sarah_estimator}
v_t := v_{t-1} + \nabla_w{f}(w_t; \xi_{t}) - \nabla_w{f}(w_{t-1}; \xi_t),
\end{equation}
for a given realization $\xi_t$ of $\xi$.
Each evaluation of $v_t$ requires two gradient evaluations.
Clearly, the SARAH estimator is biased, since $\cExp{v_t}{\Fc_t} = v_{t-1} + \nabla{f}(w_t) - \nabla{f}(w_{t-1}) \neq \nabla{f}(w_t)$.
But it has a variance reduced property.

%% 2.4.b. Minibatch estimators.
\beforepara
\paragraph{Mini-batch estimators:}
We consider a mini-batch estimator of the gradient $\nabla{f}$ in \eqref{eq:stochastic_estimator} and of the SARAH estimator \eqref{eq:sarah_estimator} respectively as follows:
\begin{equation}\label{eq:mini_batch}
{\!\!\!}\widetilde{\nabla}f_{\Bc_t}(w_t) := \frac{1}{b_t}\sum_{i\in\Bc_t}\nabla_w{f}(w_t;\xi_i)~~\text{and}~~v_t := v_{t-1} + \frac{1}{b_t}\sum_{i\in\Bc_t}\left(\nabla_w{f}(w_t;\xi_i) - \nabla_w{f}(w_{t-1};\xi_i)\right),{\!\!\!}
\end{equation}
where $\Bc_t$ is a mini-batch of the size $b_t := \abs{\Bc_t} \geq 1$.
For the finite-sum problem \eqref{eq:finite_sum}, we replace $f(\cdot;\xi_i)$ by $f_i(\cdot)$.
In this case, $\Bc_t$ is a uniformly random subset of $[n]$. 
Clearly, if $b_t = n$, then we take the full gradient $\nabla{f}$ as the exact estimator.

%%% 2.4. Basic property of SAHAR estimator
\beforesubsec
\subsection{Basic properties of stochastic and SARAH estimators}
\aftersubsec
We recall some basic properties of the standard stochastic and SARAH estimators for \eqref{eq:sopt_prob} and \eqref{eq:finite_sum}.
The following result was proved in \citep{Nguyen2017_sarah}.

\begin{lemma}\label{le:sarah_estimator}
Let $\set{v_t}_{t\geq 0}$ be defined by \eqref{eq:sarah_estimator} and $\Fc_t$ be defined by \eqref{eq:history_Ft}. Then
\begin{equation}\label{eq:biased}
\begin{array}{ll}
&\cExp{v_t}{\Fc_t} = \nabla{f}(w_t) + \epsilon_t \neq \nabla{f}(w_t),~~\text{where}~~\epsilon_t := v_{t-1} - \nabla{f}(w_{t-1}). \vspace{1.75ex}\\
&\cExp{\norms{v_t - \nabla{f}(w_t)}^2}{\Fc_t} = \norms{ v_{t-1} - \nabla{f}(w_{t-1})}^2 + \cExp{\norms{v_t - v_{t-1}}^2}{\Fc_t} \vspace{1ex}\\
&{~~~~~~~~~~~~~~~~~~~~~~~~~~~~~} - \norms{\nabla{f}(w_t) - \nabla{f}(w_{t-1})}^2.
\end{array}
\end{equation}
Consequently, for any $t \geq 0$, we have
\begin{equation}\label{eq:biased_sum}
\begin{array}{ll}
&\Exp{\norms{v_t - \nabla{f}(w_t)}^2} = \Exp{\norms{ v_0 - \nabla{f}(w_0)}^2} + \sum_{j=1}^{t}\Exp{\norms{v_j - v_{j-1}}^2} \vspace{1ex}\\
&{~~~~~~~~~~~~~~~~~~~~~~~~~~~~~} - \sum_{j=1}^{t}\Exp{\norms{\nabla{f}(w_j) - \nabla{f}(w_{j-1})}^2}.
\end{array}
\end{equation}
\end{lemma}
Our next result is some properties of the mini-batch estimators in \eqref{eq:mini_batch}.
Most of the proof is presented in \citep{harikandeh2015stopwasting,lohr2009sampling,Nguyen2017_sarahnonconvex,Nguyen2018sgd_dnn}, and we only provide the missing proof of \eqref{eq:mini_batch_est2} and  \eqref{eq:mini_batch_est2b} in Appendix~\ref{apdx:le:mini_batch}.

%% Lemma 2.2.
\begin{lemma}\label{le:mini_batch}
If $\widetilde{\nabla}f_{\Bc_t}(w_t)$ is generated by \eqref{eq:mini_batch}, then, under Assumption~\ref{as:A3}, we have
\begin{equation}\label{eq:mini_batch_est1}
\begin{array}{ll}
&\Exp{\widetilde{\nabla}f_{\Bc_t}(w_t) \mid \Fc_t} = \nabla{f}(w_t)~\text{and} \vspace{1.5ex}\\
&\Exp{\norms{\widetilde{\nabla}f_{\Bc_t}(w_t) - \nabla{f}(w_t)}^2 \mid \Fc_t} = \dfrac{1}{b_t}\Exp{\Vert\nabla_w{f}(w_t;\xi) - \nabla{f}(w_t)\Vert^2 \mid \Fc_t} \leq \dfrac{\sigma^2}{b_t}.
\end{array}
\end{equation}
If $\widetilde{\nabla}f_{\Bc_t}(w_t)$ is generated by \eqref{eq:mini_batch} for the finite support case $\vert\Omega\vert = n$, then
\begin{equation}\label{eq:mini_batch_est1b}
\begin{array}{ll}
& \Exp{\widetilde{\nabla}f_{\Bc_t}(w_t) \mid \Fc_t} = \nabla{f}(w_t) \vspace{1ex}\\
\text{and}~& \Exp{\norms{\widetilde{\nabla}f_{\Bc_t}(w_t) - \nabla{f}(w_t)}^2 \mid \Fc_t} \leq \frac{1}{b_t}\left(\frac{n-b_t}{n-1}\right)\sigma^2_n,
\end{array}
\end{equation}
where $\sigma_n^2$ is defined as
\begin{equation*}
\sigma_n^2 := \frac{1}{n}\sum_{i=1}^n\left[\norms{\nabla{f}_i(w_t)}^2 - \norms{\nabla{f}(w_t)}^2\right].
\end{equation*}
If $v_t$ is generated by \eqref{eq:mini_batch} for the case $\vert\Omega\vert = n$ in the finite-sum problem \eqref{eq:finite_sum}, then
\begin{equation}\label{eq:mini_batch_est2}
\begin{array}{ll}
\Exp{\norms{v_t - v_{t-1}}^2 \mid \Fc_t} &= \frac{n(b_t-1)}{b_t(n-1)}\Vert \nabla{f}(w_t) - \nabla{f}(w_{t-1})\Vert^2  \vspace{1ex}\\
& + {~} \frac{(n - b_t)}{b_t(n-1)}\cdot \frac{1}{n}\; \sum_{i=1}^n\norms{\nabla{f_i}(w_t) - \nabla{f_i}(w_{t-1})}^2.
\end{array}
\end{equation}
If $v_t$ is generated by \eqref{eq:mini_batch} for the case $\vert\Omega\vert \neq n$ in the expectation problem \eqref{eq:sopt_prob}, then
\begin{equation}\label{eq:mini_batch_est2b}
\begin{array}{ll}
\Exp{\norms{v_t - v_{t-1}}^2  \mid \Fc_t } &=  \left( 1 - \frac{1}{b_t} \right) \| \nabla f(w_t) - \nabla f(w_{t-1}) \|^2 \vspace{1ex}\\
&+ {~} \frac{1}{b_t}\Exp{\norms{\nabla_w{f}(w_t;\xi) - \nabla_w{f}(w_{t-1};\xi)}^2  \mid \Fc_t}.
\end{array}
\end{equation}
\end{lemma}
Note that if $b_t = n$, i.e., we take a full gradient estimate, then the second estimate of \eqref{eq:mini_batch_est1b} is vanished and independent of $\sigma_n$.
The second term of \eqref{eq:mini_batch_est2} is also vanished.

%%%%%%%%%%%%%%%%%%%%%%%%%%%%%
%%%% 3. The algorithms and convergence analysis
%%%%%%%%%%%%%%%%%%%%%%%%%%%%%
\beforesec
\section{ProxSARAH framework and convergence analysis}\label{sec:alg_section}
\aftersec
We describe our unified algorithmic framework and then specify it to solve different instances of \eqref{eq:sopt_prob} and \eqref{eq:finite_sum} under appropriate structures.
The general algorithm is described in Algorithm~\ref{alg:prox_sarah}, which is abbreviated by ProxSARAH.

\begin{algorithm}[hpt!]\caption{(Proximal SARAH with stochastic recursive gradient estimators)}\label{alg:prox_sarah}
\begin{algorithmic}[1]
   \State{\bfseries Initialization:} An initial point $\widetilde{w}_0$ and necessary parameters $\eta_t > 0$ and $\gamma_t \in (0, 1]$ (will be specified in the sequel).
   \vspace{0.5ex}
   \State{\bfseries Outer Loop:}~{\bfseries For $s := 1, 2, \cdots, S$ do}
   \vspace{0.5ex}   
   \State\hspace{3ex}\label{step:o2} Generate a snapshot $v_0^{(s)}$ at $w_0^{(s)} := \widetilde{w}_{s-1}$.
   \vspace{0.5ex}   
   \State\hspace{3ex}\label{step:o3} Update $\widehat{w}_1^{(s)} := \prox_{\eta_0\psi}(w_0^{(s)} - \eta_0v_0^{(s)})$ and $w_1^{(s)} := (1-\gamma_0)w_0^{(s)} + \gamma_0\widehat{w}_1^{(0)}$.
   \vspace{0.5ex}   
   \State\hspace{3ex}\label{step:o4}{\bfseries Inner Loop:}~{\bfseries For $t := 1,\cdots,m$ do}
   \vspace{0.5ex}   
   \State\hspace{6ex}\label{step:i1} Generate a proper single random  sample or mini-batch $\hat{\Bc}_t^{(s)}$. 
   \vspace{0.5ex}   
   \State\hspace{6ex}\label{step:i2} Evaluate $v_{t}^{(s)} := v_{t-1}^{(s)} + \frac{1}{\vert \hat{\Bc}_t^{(s)}\vert}\sum_{\xi_t^{(s)} \in \hat{\Bc}_t^{(s)}}\big[ \nabla_w{f}(w_{t}^{(s)}; \xi_t^{(s)}) - \nabla_w{f}(w_{t-1}^{(s)}; \xi_t^{(s)}) \big]$.
   \vspace{0.5ex}   
   \State\hspace{6ex}\label{step:i3} Update $\widehat{w}_{t+1}^{(s)} := \prox_{\eta_t\psi}(w_{t}^{(s)} - \eta_t v_{t}^{(s)})$ and $w_{t+1}^{(s)} := (1-\gamma_t)w_t^{(s)} + \gamma_t\widehat{w}_{t+1}^{(s)}$.
   \vspace{0.5ex}   
   \State\hspace{3ex}{\bfseries End For}
   \vspace{0.5ex}   
   \State\hspace{3ex}\label{step:o5} Set $\widetilde{w}_s := w_{m+1}^{(s)}$ 
   \vspace{0.5ex}   
   \State{\bfseries End For}
\end{algorithmic}
\end{algorithm} 

In terms of algorithm, ProxSARAH is different from SARAH where it has one proximal step  followed by an additional averaging step, Step~\ref{step:i3}. 
However, using the gradient mapping $G_{\eta}$ defined by \eqref{eq:gradient_mapping}, we can view Step~\ref{step:i3} as:
\begin{equation}\label{eq:grad_representation}
w_{t+1}^{(s)} := w_t^{(s)} - \eta_t\gamma_tG_{\eta_t}(w_t^{(s)}).
\end{equation}
Hence, this step is similar to a gradient step applying to the gradient mapping $G_{\eta_t}(w_t^{(s)})$.
In particular, if we set $\gamma_t = 1$, then we obtain a vanilla proximal SARAH variant which is similar to ProxSVRG, ProxSVRG+, and ProxSpiderBoost discussed above.
ProxSVRG, ProxSVRG+, and ProxSpiderBoost are simply vanilla proximal gradient-type methods in stochastic setttings.
If $\psi = 0$, then ProxSARAH is reduced to SARAH in \citep{Nguyen2017_sarah,Nguyen2017_sarahnonconvex,Nguyen2018_iSARAH} with a step-size $\hat{\eta}_t := \gamma_t\eta_t$. 
Note that Step~\ref{step:i3} can be represented as a weighted averaging step with given weights $\sets{\tau_j^{(s)}}_{j=0}^m$:
\vspace{-1ex}
\begin{equation*}
w_{t+1}^{(s)} := \frac{1}{\Sigma_t^{(s)}}\sum_{j=0}^t\tau_j^{(s)}\widehat{w}_{j+1}^{(s)}, ~~~\text{where}~~\Sigma_t^{(s)} := \sum_{j=0}^t\tau_j^{(s)} ~~\text{and}~~\gamma_j^{(s)} := \frac{\tau_j^{(s)}}{\Sigma_t^{(s)}}.
\vspace{-1ex}
\end{equation*}
Compared to \citep{ghadimi2012optimal,Nemirovski2009a}, ProxSARAH evaluates $v_t$ at the averaged point $w_t^{(s)}$ instead of $\widehat{w}_t^{(s)}$.
Therefore, it can be written as
\begin{equation*}
w_{t+1}^{(s)} := (1-\gamma_t)w_t^{(s)} +  \gamma_t\prox_{\eta_t\psi}(w_{t}^{(s)} - \eta_t v_{t}^{(s)}),
\end{equation*}
which is similar to averaged fixed-point schemes (e.g.  the Krasnosel'ski\u{i} -- Mann scheme) in the literature, see, e.g., \citep{Bauschke2011}. 

In addition, we will show in our analysis a key difference in terms of step-sizes $\eta_t$ and $\gamma_t$, mini-batch, and epoch length between ProxSARAH and existing methods, including SPIDER \citep{fang2018spider} and SpiderBoost \citep{wang2018spiderboost}. 

%%%%%%%%%%%%%%%%%%%%%%%%%%%%
%%% 3.1. Analysis of the inner-loop: Key estimates.
\beforesubsec
\subsection{Analysis of the inner-loop: Key estimates}
\aftersubsec
This subsection proves two key estimates of the inner loop for $t = 1$ to $m$.
We break our analysis into two different lemmas, which provide key estimates for our convergence analysis.
We assume that the mini-batch size $\hat{b} := \vert \hat{\Bc}_t^{(s)}\vert$ in the inner loop is fixed.

%%% Lemma 3.1.
\begin{lemma}\label{le:key_est1}
Let $\set{(w_t,\widehat{w}_t)}$ be generated by the inner-loop of Algorithm~\ref{alg:prox_sarah} with $\vert \hat{\Bc}_t^{(s)}\vert = \hat{b} \in [n-1]$ fixed. 
Then, under Assumption~\ref{as:A2}, we have
\begin{equation}\label{eq:key_est1}
\begin{array}{ll}
\Exp{F(w_{m+1}^{(s)})} {\!\!\!\!}&\leq  \Exp{F(w_0^{(s)})} + \displaystyle\frac{\rho L^2}{2}\sum_{t=0}^m\beta_t\displaystyle\sum_{j=1}^{t}\gamma_{j-1}^2\Exp{\Vert\widehat{w}_j^{(s)} - w_{j-1}^{(s)}\Vert^2}  \vspace{1ex}\\
& - {~} \displaystyle\frac{1}{2}\sum_{t=0}^m\kappa_t \Exp{\Vert\widehat{w}_{t+1}^{(s)} - w_t^{(s)}\Vert^2} + \dfrac{1}{2}\bar{\sigma}^{(s)}\Big(\displaystyle\sum_{t=0}^m\beta_t\Big)   \vspace{1ex}\\
& - {~} \displaystyle\sum_{t=0}^m\frac{s_t\eta_t^2}{2}\Exp{\Vert G_{\eta_t}(w_t^{(s)})\Vert^2} -  \displaystyle\sum_{t=0}^m\Exp{\sigma_t^{(s)}},
\end{array}
\end{equation}
where $\set{c_t}$, $\set{r_t}$, and $\set{s_t}$ are any given positive sequences, $\bar{\sigma}^{(s)} := \Exp{\norms{v_0^{(s)} - \nabla{f}(w_0^{(s)})}^2} \geq 0$,  $\sigma_t^{(s)} := \frac{\gamma_t}{2c_t}\Vert \nabla{f}(w_t^{(s)}) - v_t^{(s)} - c_t(\widehat{w}_{t+1}^{(s)} - w_t^{(s)})\Vert^2 \geq 0$, and 
\begin{equation}\label{eq:key_est1_param}
\beta_t := \frac{\gamma_t}{c_t} + (1+r_t)s_t\eta_t^2, ~~\text{and}~~\kappa_t := \frac{2\gamma_t}{\eta_t} -  L\gamma_t^2- \gamma_t c_t - s_t\left(1+\frac{1}{r_t} \right).
\end{equation}
Here, $\rho := \frac{1}{\hat{b}}$ if Algorithm~\ref{alg:prox_sarah} solves \eqref{eq:sopt_prob}, and $\rho := \frac{(n-\hat{b})}{\hat{b}(n-1)}$ if Algorithm~\ref{alg:prox_sarah} solves \eqref{eq:finite_sum}.
\end{lemma}

The proof of Lemma~\ref{le:key_est1} is deferred to Appendix \ref{apdx:le:key_est1}.
The next lemma shows how to choose constant step-sizes $\gamma$ and $\eta$ by fixing other parameters in Lemma~\ref{le:key_est1} to obtain a descent property. The proof of this lemma is given in Appendix \ref{apdx:le:constant_stepsize}.

%%% Lemma 3.2.
\begin{lemma}\label{le:constant_stepsize}
Under Assumption~\ref{as:A2} and $\hat{b} := \vert \hat{\Bc}_t^{(s)}\vert \in [n-1]$, let us choose $\eta_t = \eta > 0$ and $\gamma_t = \gamma > 0$ in Algorithm~\ref{alg:prox_sarah} such that
\begin{equation}\label{eq:constant_stepsize}
\gamma_t = \gamma := \frac{1}{L\sqrt{\omega m}}~~~\text{and}~~\eta_t = \eta := \frac{2\sqrt{\omega m}}{4\sqrt{\omega m} + 1},
\end{equation}
where $\omega := \frac{3}{2\hat{b}}$ if Algorithm~\ref{alg:prox_sarah} solves \eqref{eq:sopt_prob} and $\omega := \frac{3(n - \hat{b})}{2\hat{b}(n-1)}$ if Algorithm~\ref{alg:prox_sarah} solves \eqref{eq:finite_sum}.
Then
\begin{equation}\label{eq:key_est2}
\begin{array}{ll}
\Exp{F(w_{m+1}^{(s)})} &\leq {~} \Exp{F(w_0^{(s)})} - \displaystyle\frac{\gamma\eta^2}{2}\sum_{t=0}^m\Exp{\Vert G_{\eta}(w_t^{(s)})\Vert^2} \\
&  -  {~} \displaystyle\sum_{t=0}^m\Exp{\sigma_t^{(s)}} + \dfrac{\gamma\theta}{2}(m+1)\bar{\sigma}^{(s)},
\end{array}
\end{equation}
where $\theta := 1 + 2\eta^2 \leq \frac{3}{2}$.
\end{lemma}

%%%%%%%%%%%%%%%%%%%%%%%%%%%%%%%%%%%%%%
%%% 3.2. Convergence analysis for  composite finite-sum problem: Mini-batch.
\beforesubsec
\subsection{Convergence analysis for  the composite finite-sum problem \eqref{eq:finite_sum}}\label{subsec:mini_batch_case}
\aftersubsec
In this subsection, we specify Algorithm~\ref{alg:prox_sarah} to solve the composite finite-sum problem \eqref{eq:finite_sum}.
We replace $v^{(s)}_0$ at \textbf{Step~\ref{step:o2}} and $v^{(s)}_t$ at \textbf{Step~\ref{step:i2}} of Algorithm~\ref{alg:prox_sarah} by the following ones:
\begin{equation}\label{eq:vt_finite_sum}
v^{(s)}_0 := \frac{1}{b_s}\sum_{j\in\Bc_s}\nabla{f_{j}}(w^{(s)}_0),~~\text{and}~~
v^{(s)}_t := v_{t-1}^{(s)} + \frac{1}{\hat{b}_t^{(s)}}\sum_{i\in\hat{\Bc}_t^{(s)}}\left(\nabla{f_{i}}(w^{(s)}_t) - \nabla{f_{i}}(w^{(s)}_{t-1})\right), 
\end{equation}
where $\Bc_s$ is an outer mini-batch of a fixed size $b_s := \vert\Bc_s\vert = b$, and $\hat{\Bc}_t^{(s)}$ is an inner mini-batch of a fixed size $\hat{b}_t^{(s)} := \vert\hat{\Bc}_t^{(s)}\vert = \hat{b}$.
Moreover, $\Bc_s$ is independent of $\Bc_t^{(s)}$.

We consider two separate cases of this algorithmic variant: adaptive step-sizes and constant step-sizes, but with fixed inner mini-batch size $\hat{b} \in [n-1]$. 
The following theorem proves the convergence of the adaptive step-size variant, whose proof is postponed until Appendix~\ref{apdx:th:convergence_composite_finite_sum_b}.

%%% Theorem 5.
\begin{theorem}\label{th:convergence_composite_finite_sum_b}
Assume that we apply Algorithm~\ref{alg:prox_sarah} to solve \eqref{eq:finite_sum}, where the estimators $v_0^{(s)}$ and $v_t^{(s)}$ are defined by \eqref{eq:vt_finite_sum} such that $b_s = b \in [n]$ and $\hat{b}_t^{(s)} = \hat{b} \in [n-1]$.

Let $\eta_t := \eta \in (0, \frac{2}{3})$ be fixed, $\omega_{\eta} := \frac{(1 + 2\eta^2)(n-\hat{b})}{\hat{b}(n-1)}$, and $\delta := \frac{2}{\eta} - 3 > 0$.
Then, the sequence $\set{\gamma_t}_{t=0}^m$ updated in a backward mode by
\begin{equation}\label{eq:adaptive_update_param_minibatch} 
\gamma_m := \frac{\delta}{L},~~~\text{and}~~\gamma_t := \frac{\delta}{L\big[\eta + \omega_{\eta} L\sum_{j=t+1}^m\gamma_j\big]},~~t=0,\cdots, m-1,
\end{equation}
satisfies 
\begin{equation}\label{eq:Sigma_lower_bound}
{\!\!\!}\frac{\delta}{L(1 + \delta\omega_{\eta} m)} \leq \gamma_0 < \gamma_1 < \cdots < \gamma_m, ~~~\text{and}~~~
\Sigma_m := \sum_{t=0}^m\gamma_t \geq \frac{2\delta(m+1)}{L(\sqrt{2\delta\omega_{\eta} m + 1} + 1)}.
{\!\!\!}
\end{equation}
Moreover, under Assumptions~\ref{as:A1} and \ref{as:A2}, the following bound holds:
\begin{equation}\label{eq:mini_batch_bound}
\frac{1}{S\Sigma_m}\sum_{s=1}^S\sum_{t=0}^m\gamma_t\Exp{\Vert G_{\eta}(w_t^{(s)})\Vert^2}  \leq  \frac{2}{\eta^2S\Sigma_m}\big[F(\widetilde{w}_0) - F^{\star}\big]  + \frac{3\sigma_n^2}{2\eta^2S}\sum_{s=1}^S\frac{(n-b_s)}{nb_s}.
\end{equation}
If we choose $\eta := \frac{1}{2}$, $m := \big\lfloor \frac{n}{\hat{b}}\big\rfloor$, $b_s  := n$, and $\hat{b} \in [1, \sqrt{n}]$, then $\widetilde{w}_T$ chosen by  $\widetilde{w}_T ~\sim \Ub_p\big(\sets{w_t^{(s)}}_{t=0\to m}^{s=1 \to S}\big)$ such that
\begin{equation*}
\Prob{\widetilde{w}_T = w_t^{(s)}} = p_{(s-1)m+t} := \frac{\gamma_t}{S\Sigma_m}, 
\end{equation*}
satisfies
\begin{equation}\label{eq:grad_norm_bound1_c2}
\Exp{\norms{G_{\eta}(\widetilde{w}_T)}^2} \leq   \frac{4\sqrt{6}L\left[F(\widetilde{w}_0) - F^{\star}\right]}{S\sqrt{n}}.
\end{equation}
Consequently,  the number of outer iterations $S$ needed to obtain $\widetilde{w}_T$ such that $\Exp{\norms{G_{\eta}(\widetilde{w}_T)}^2} \leq \varepsilon^2$ is at most $S := \frac{4\sqrt{6}L\left[F(\widetilde{w}_0) - F^{\star}\right]}{\sqrt{n}\varepsilon^2}$.
Moreover, if $n \leq \frac{96L^2\left[F(\widetilde{w}_0) - F^{\star}\right]^2}{\varepsilon^4}$, then $S \geq 1$.

The number of individual stochastic gradient evaluations $\nabla{f_i}$ does not exceed 
\begin{equation*}
\Tc_{\mathrm{grad}} := \frac{20\sqrt{6}L\sqrt{n}\left[F(\widetilde{w}_0) - F^{\star}\right]}{\varepsilon^2} = \BigO{\frac{L\sqrt{n}}{\varepsilon^2}\left[F(\widetilde{w}_0) - F^{\star}\right]}.
\end{equation*}
The number of $\prox_{\eta\psi}$ operations does not exceed $\Tc_{\prox} :=  \frac{4\sqrt{6}(\sqrt{n} + 1)L\left[F(\widetilde{w}_0) - F^{\star}\right]}{\hat{b}\varepsilon^2}$.
\end{theorem}
%%% End of Theorem 3.2.

Alternatively,  Theorem \ref{th:convergence_composite_finite_sum_b2} below shows the convergence of Algorithm~\ref{alg:prox_sarah} for  the constant step-size case, whose proof is given in Appendix~\ref{apdx:th:convergence_composite_finite_sum_b2}.

%%% Theorem 3.2.
\begin{theorem}\label{th:convergence_composite_finite_sum_b2}
Assume that we apply Algorithm~\ref{alg:prox_sarah} to solve \eqref{eq:finite_sum}, where the estimators $v_0^{(s)}$ and $v_t^{(s)}$ are defined by \eqref{eq:vt_finite_sum} such that $b_s = b \in [n]$ and $\hat{b}_t^{(s)} = \hat{b} \in [n-1]$.

Let us choose constant step-sizes $\gamma_t = \gamma$ and $\eta_t = \eta$ as
\begin{equation}\label{eq:mini_batch_step_size}
\gamma := \frac{1}{L\sqrt{\omega m}}~~~\text{and}~~~\eta := \frac{2\sqrt{\omega m}}{4\sqrt{\omega m} + 1},~~\text{where}~~\omega :=  \frac{3(n-\hat{b})}{2\hat{b}(n-1)}~~\text{and}~~\hat{b} \in [1,\sqrt{n}].
\end{equation}
Then, under Assumptions~\ref{as:A1} and \ref{as:A2}, if we choose $m := \big\lfloor \frac{n}{\hat{b}}\big\rfloor$, $b_s  := n$, and $\widetilde{w}_T ~\sim \Ub\big(\sets{w_t^{(s)}}_{t=0\to m}^{s=1 \to S}\big)$, then
the number of outer iterations $S$ to achieve $\Exp{\norms{G_{\eta}(\widetilde{w}_T)}^2} \leq \varepsilon^2$ does not exceed 
\begin{equation*}
S := \frac{16\sqrt{3}L}{\sqrt{2n}\varepsilon^2}\big[F(\widetilde{w}_0) - F^{\star}\big].
\end{equation*}
Moreover, if $n \leq  \frac{384L^2}{\varepsilon^4}\big[F(\widetilde{w}_0) - F^{\star}\big]^2$, then $S \geq 1$.

Consequently, the number of stochastic gradient evaluations $\Tc_{\mathrm{grad}}$ does not exceed
\begin{equation*}
\Tc_{\mathrm{grad}} := \frac{16\sqrt{3}L\sqrt{n}}{\sqrt{2}\varepsilon^2}\big[F(\widetilde{w}_0) - F^{\star}\big] = \BigO{ \frac{L\sqrt{n}}{\varepsilon^2}\big[F(\widetilde{w}_0) - F^{\star}\big]}.
\end{equation*}
The number of $\prox_{\eta\psi}$ operations does not exceed $\Tc_{\prox} :=  \frac{16\sqrt{3}L(\sqrt{n}+1)}{\hat{b}\sqrt{2}\varepsilon^2}\big[F(\widetilde{w}_0) - F^{\star}\big]$.
\end{theorem}
%%% End of Theorem 3.2.

Note that the condition $n \leq \BigO{\varepsilon^{-4}}$ is to guarantee that $S \geq 1$ in Theorems~\ref{th:convergence_composite_finite_sum_b} and \ref{th:convergence_composite_finite_sum_b2}.
In this case, our complexity bound is $\BigO{n^{1/2}\varepsilon^{-2}}$.
Otherwise, i.e., $n > \BigO{\varepsilon^{-4}}$, then our complexity becomes  $\BigO{n +  n^{1/2}\varepsilon^{-2}}$ due to the full gradient snapshots.
In the non-composite setting, this complexity is the same as SPIDER \citep{fang2018spider}, and the range of our mini-batch size $\hat{b} \in [1, \sqrt{n}]$, which is the same as in SPIDER, instead of fixed $\hat{b} = \lfloor\sqrt{n}\rfloor$ as in SpiderBoost \citep{wang2018spiderboost}.
Note that we can extend our mini-batch size $\hat{b}$ such that $\sqrt{n} < \hat{b} \leq n-1$, but our complexity bound is no longer the best-known one.

The step-size $\eta$ in \eqref{eq:mini_batch_step_size} can be bounded by $\eta \in [\frac{2}{5}, \frac{1}{2}]$ for any batch size $\hat{b}$ and $m$ instead of fixing at $\eta = \frac{1}{2}$.
Nevertheless, this interval can be enlarged by choosing different $c_t$ and $r_t$ in Lemma~\ref{le:key_est1}.
For example, if we choose $c_t := \frac{1}{2}$ and $r_t := 2$ in Lemma~\ref{le:key_est1}, then $\eta$ can go up to $\frac{2}{3}$.
The step-size $\gamma \in (0, 1]$ can change from a small to a large value close to $1$ as the batch-size $\hat{b}$ and the epoch length $m$ change as we will discuss in Subsection~\ref{subsubsec:mini_batch_step_size}.

%%%%%%%%%%%%%%%%%%%%%%%%%%%%%%%
%%% 3.3. Lower bound complexity
\beforesubsec
\subsection{Lower-bound complexity for the finite-sum problem \eqref{eq:finite_sum}}
\aftersubsec
Let us analyze a special case of \eqref{eq:finite_sum} with $\psi = 0$.
We consider any stochastic first-order methods to generate an iterate sequence $\set{w_t}$ as follows: 
\begin{equation}\label{eq:first_order_alg}
[w_t, i_t] := \Ac^{t-1}\left(\omega, \nabla{f_{i_0}}(w^0), \nabla{f_{i_1}}(w^1), \cdots, \nabla{f_{i_{t-1}}}(w^{t-1})\right),~~t\geq 1,
\end{equation}
where $\Ac^{t-1}$ are measure mapping into $\R^{d+1}$, $f_{i_t}$ is an individual function chosen by $\Ac^{t-1}$ at iteration $t$, $\omega\sim\Ub([0, 1])$ is a random vector, and $[w^0, i_0] := \Ac^0(\omega)$.
Clearly, Algorithm~\ref{alg:prox_sarah} can be cast as a special case of \eqref{eq:first_order_alg}.
As shown in \citep[Theorem 3]{fang2018spider} and later in \citep[Theorem 4.5.]{zhou2019lower}, under Assumptions~\ref{as:A1} and \ref{as:A2}, for any $L > 0$ and $2 \leq n \leq \BigO{L^2\left[F(w^0) - F^{\star}\right]^2\varepsilon^{-4}}$, there exists a dimension $d = \widetilde{\mathcal{O}}(L^2\left[F(w^0) - F^{\star}\right]^2n^2\varepsilon^{-4})$ such that the lower-bound complexity of Algorithm~\ref{alg:prox_sarah} to produce an output $\widetilde{w}_T$ such that $\Exp{\norm{\nabla{f}(\widetilde{w}_T)}^2} \leq \varepsilon^2$ is $\Omega\left(\frac{L\left[F(w^0) - F^{\star}\right]\sqrt{n}}{\varepsilon^2}\right)$.
This lower-bound  clearly matches the upper bound $\Tc_{\mathrm{grad}}$ in Theorems~\ref{th:convergence_composite_finite_sum_b} and \ref{th:convergence_composite_finite_sum_b2} up to a given constant factor.
%Therefore, Algorithm~\ref{alg:prox_sarah} achieves nearly optimal complexity for solving \eqref{eq:finite_sum}. 

%%%%%%%%%%%%%%%%%%%%%%%%%%%
%%% 3.4 . Mini-batch size and learning rate trade-offs.
\beforesubsec
\subsection{Mini-batch size and learning rate trade-offs}\label{subsubsec:mini_batch_step_size}
\aftersubsec
Although our step-size defined by \eqref{eq:mini_batch_step_size} in the single sample case is much larger than that of ProxSVRG in \citep[Theorem 1]{reddi2016proximal}, it still depends on $\sqrt{m}$, where $m$ is the epoch length.
To obtain larger step-sizes, we can choose $m$ and the mini-batch size $\hat{b}$ using the same trick as in  \citep[Theorem 2]{reddi2016proximal}.
Let us first fix $\gamma := \bar{\gamma} \in (0, 1]$.
From \eqref{eq:mini_batch_step_size}, we have $\omega m = \frac{1}{L^2\bar{\gamma}^2}$.
It makes sense to choose $\bar{\gamma}$ close to $1$ in order to use new information from $\widehat{w}^{(s)}_{t+1}$ instead of the old one in $w^{(s)}_t$.

Our goal is to choose $m$ and $\hat{b}$ such that $\omega m = \frac{3(n-\hat{b})m}{2\hat{b}(n-1)} = \frac{1}{L^2\bar{\gamma}^2}$.
If we define $C := \frac{2}{3L^2\bar{\gamma}^2}$, then the last condition implies that $\hat{b} := \frac{mn}{Cn + m - C} \leq \frac{m}{C}$ provided that $m \geq C$.
Our suggestion is to choose
\vspace{-1ex}
\begin{equation}\label{eq:mini_batch_vs_stepsize}
\gamma := \bar{\gamma} \in (0, 1], ~~\hat{b} := \Big\lfloor\frac{mn}{Cn+m-C} \Big\rfloor,~~\text{and}~~\eta := \frac{2}{4 + L\bar{\gamma}}.
\vspace{-1ex}
\end{equation}
If we choose $m = \lfloor n^{1/3}\rfloor$, then $\hat{b} = \BigO{n^{1/3}} \leq \frac{n^{1/3}}{C}$.
This mini-batch size is much smaller than $\lfloor n^{2/3}\rfloor$ in ProxSVRG.
Note that, in ProxSVRG, they set $\gamma := 1$ and $\eta := \frac{1}{3L}$.

In ProxSpiderBoost \citep{wang2018spiderboost}, $m$ and the mini-batch size $\hat{b}$ were chosen as $m = \hat{b} = \lfloor n^{1/2}\rfloor$ so that they can use constant step-sizes $\gamma = 1$ and $\eta = \frac{1}{2L}$.
In our case, if $\gamma = 1$, then $\eta = \frac{2}{4 + L}$. Hence, if $L = 1$, then $\eta_{\mathrm{ProxSpiderBoost}} = \frac{1}{2} > \eta_{\mathrm{ProxSARAH}} = \frac{2}{5} > \eta_{\mathrm{ProxSVRG}} = \frac{1}{3}$.
But if $L > 4$, then our step-size $\eta_{\mathrm{ProxSARAH}}$ dominates $\eta_{\mathrm{ProxSpiderBoost}}$.
However, if we choose  $c_t := \frac{1}{2}$, $r_t := 2$ in Lemma~\ref{le:key_est1}, then $\eta_{\mathrm{ProxSARAH}}= \frac{2}{3} > \eta_{\mathrm{ProxSpiderBoost}} = \frac{1}{2}$.
%
%\nhanp{$\eta_{\mathrm{ProxSpiderBoost}} = \frac{1}{2} > \eta_{\mathrm{ProxSARAH}} = \frac{2}{5}$ is only true when we choose $c_t = r_t = 1$, if we choose $c_t = 0.5$, $r_t = 2$ then $\eta_{\mathrm{ProxSpiderBoost}} = \frac{1}{2} < \eta_{\mathrm{ProxSARAH}}= \frac{2}{3}$.}
%

If we choose $m = \BigO{n^{1/2}}$ and $\hat{b} = \BigO{n^{1/2}}$, then we maintain the same complexity bound $\BigO{n^{1/2}\varepsilon^{-2}}$ as in Theorems \ref{th:convergence_composite_finite_sum_b} and \ref{th:convergence_composite_finite_sum_b2}.
However, if we choose $m = \BigO{n^{1/3}}$ and $\hat{b} = \BigO{n^{1/3}}$, then the complexity bound becomes $\BigO{(n^{2/3} + n^{1/3})\varepsilon^{-2}}$, which is similar to ProxSVRG.

%%%%%%%%%%%%%%%%%%%%%%%%%%%%%%%%%%%%
%%%% 3.3. Convergence analysis for the composite expectation problem.
\beforesubsec
\subsection{Convergence analysis for the composite expectation problem \eqref{eq:sopt_prob}}
\aftersubsec
In this subsection, we apply Algorithm~\ref{alg:prox_sarah} to solve the general expectation setting \eqref{eq:sopt_prob}. 
In this case, we generate the snapshot at \textbf{Step~\ref{step:o2}} of Algorithm~\ref{alg:prox_sarah} as follows:
\begin{equation}\label{eq:snapshot_v0_2}
v_0^{(s)} := \frac{1}{b_s}\sum_{\zeta_i^{(s)}\in\Bc_s}\nabla_w{f}(w_0^{(s)}; \zeta_i^{(s)}), 
\end{equation}
where $\Bc_s := \set{\zeta_1^{(s)},\cdots, \zeta_{b_s}^{(s)}}$ is a mini-batch of i.i.d. realizations of $\xi$ at the $s$-th outer iteration and independent of $\xi_{t}$ from the inner loop, and $b_s := \vert \Bc_s\vert = b \geq 1$ is fixed.

Now, we analyze the convergence of Algorithm~\ref{alg:prox_sarah} for solving \eqref{eq:sopt_prob} using \eqref{eq:snapshot_v0_2} above.
For simplicity of discussion, we only consider the constant step-size case.
The adaptive step-size variant can be derived similarly as in Theorem~\ref{th:convergence_composite_finite_sum_b} and we omit the details.
The proof of the following theorem can be found in Appendix~\ref{apdx:th:convergence_composite_expectation}.

%%% Theorem 3.3.
\begin{theorem}\label{th:convergence_composite_expectation}
Let us apply Algorithm~\ref{alg:prox_sarah} to solve \eqref{eq:sopt_prob} using \eqref{eq:snapshot_v0_2} for $v_0^{(s)}$ at \textbf{Step~\ref{step:o2}} of Algorithm~\ref{alg:prox_sarah} with fixed outer loop batch-size $b_s = b \geq 1$ and inner loop batch-size $\hat{b} := \vert \Bc_t^{(s)}\vert \geq 1$.

If we choose fixed step-sizes $\gamma$ and $\eta$ as 
\begin{equation}\label{eq:step_size3}
\gamma := \frac{1}{L\sqrt{\bar{\omega} m}}~~~~~\text{and}~~~~\eta := \frac{2\sqrt{\bar{\omega} m}}{4\sqrt{\bar{\omega} m} + 1},~~~\text{with}~~\bar{\omega} := \frac{3}{2\hat{b}},
\end{equation}
then, under Assumptions~\ref{as:A1} and \ref{as:A2}, we have the following estimate:
\begin{equation}\label{eq:key_est6}
\frac{1}{(m+1)S}\sum_{s=1}^S\sum_{t=0}^m \Exp{\Vert G_{\eta}(w_t^{(s)})\Vert^2} \leq \frac{2}{\gamma\eta^2(m+1)S}\big[F(\widetilde{w}_0) - F^{\star}\big] + \frac{3 \sigma^2}{2\eta^2b}.
\end{equation}
In particular, if we choose $b := \left\lfloor\frac{75\sigma^2}{\varepsilon^2}\right\rfloor$ and $m := \left\lfloor\frac{\sigma^2}{\hat{b}\varepsilon^2}\right\rfloor$ for $\hat{b} \leq \frac{\sigma^2}{\varepsilon^2}$, then after at most
\begin{equation*}
S :=  \frac{32 L[F(\widetilde{w}_0) - F^{\star}]}{\sigma\varepsilon}
\end{equation*}
outer iterations, we obtain $\Exp{\norms{G_{\eta}(\widetilde{w}_T)}^2} \leq \varepsilon^2$, where  $\widetilde{w}_T ~\sim \Ub\big(\sets{w_t^{(s)}}_{t=0\to m}^{s=1\to S}\big)$.

Consequently, the number of individual stochastic gradient evaluations $\nabla_wf(w_t^{(s)}; \xi_t)$  and the number of proximal operations $\prox_{\eta\psi}$, respectively do not exceed:
\begin{equation*}
\Tc_{\mathrm{grad}} :=   \frac{2464\sigma L[F(\widetilde{w}_0) - F^{\star}]}{\varepsilon^3},~~~\text{and}~~\Tc_{\prox} := \frac{32\sigma L[F(\widetilde{w}_0) - F^{\star}]}{\hat{b}\varepsilon^2}.
\end{equation*}
\end{theorem}

Theorem~\ref{th:convergence_composite_expectation} achieves the best-known complexity $\BigO{ \sigma L \varepsilon^{-3}}$ for the composite expectation problem \eqref{eq:sopt_prob}  as long as $\sigma \leq \frac{32L[F(\widetilde{w}_0) - F^{\star}]}{\varepsilon^2}$. 
Otherwise, our complexity is  $\BigO{\sigma \varepsilon^{-3} + \sigma^2\varepsilon^{-2}}$ due to the snapshot gradient for evaluating $v_0^{(s)}$. 
This complexity is the same as  SPIDER \citep{fang2018spider} in the non-composite setting and ProxSpiderBoost \citep{wang2018spiderboost} in the mini-batch setting.
Note that our method does not require to perform mini-batch in the inner loop, i.e., it is independent of $\hat{\Bc}_t^{(s)}$, and the mini-batch is independent of the number of iterations $m$ of the inner loop, while in \citep{wang2018spiderboost}, the mini-batch size $\vert \hat{\Bc}_t^{(s)}\vert$ must be proportional to $\sqrt{\vert\Bc_s\vert} = \BigO{\varepsilon^{-1}}$, where $\Bc_s$ is the mini-batch of the outer loop.
This is perhaps the reason why ProxSpiderBoost  can take a large constant step-size $\eta = \frac{1}{2L}$ as discussed in Subsection~\ref{subsubsec:mini_batch_step_size}.

%%%% Remark 1.
\begin{remark}\label{re:constant_in_bounds}
We have not attempted to optimize the constants in the complexity bounds of all theorems above, Theorem \ref{th:convergence_composite_finite_sum_b}, Theorem~\ref{th:convergence_composite_finite_sum_b2}, and Theorem~\ref{th:convergence_composite_expectation}.
Our analysis can be refined to obtain smaller constants in these complexity bounds.
\end{remark}

%%%%%%%%%%%%%%%%%%%%%%%%%%%%%%%%%%%%%%%%%%%%%%%%
%%% 4. Special cases and extensions
%%%%%%%%%%%%%%%%%%%%%%%%%%%%%%%%%%%%%%%%%%%%%%%%
\beforesec
\section{Adaptive methods for non-composite problems}\label{sec:extensions}
\aftersec
In this section, we consider the non-composite settings of \eqref{eq:sopt_prob} and  \eqref{eq:finite_sum} as  special cases  of Algorithm~\ref{alg:prox_sarah}. 
Note that if we solely apply Algorithm~\ref{alg:prox_sarah} with constant stepsizes to solve the non-composite case of \eqref{eq:sopt_prob} and \eqref{eq:finite_sum} when $\psi \equiv 0$, then by using the same step-size as in Theorems~\ref{th:convergence_composite_finite_sum_b}, \ref{th:convergence_composite_finite_sum_b2}, and \ref{th:convergence_composite_expectation},  we can obtain the same complexity as stated in Theorems~\ref{th:convergence_composite_finite_sum_b}, \ref{th:convergence_composite_finite_sum_b2}, and \ref{th:convergence_composite_expectation}, respectively.
However, we will modify our proof of Theorem~\ref{th:convergence_composite_finite_sum_b} to take advantage of the extra term $\sum_{t=0}^m\Exp{\sigma_t^{(s)}}$ in Lemma~\ref{le:key_est1}. 
The proof of this theorem is given in Appendix~\ref{apdx:thm:noncomposite_convergence}. 

%%% Theorem 3.4.
\begin{theorem}\label{thm:noncomposite_convergence}
Let $\sets{w^{(s)}_t}$ be the sequence generated by a variant of Algorithm~\ref{alg:prox_sarah} to solve the non-composite instance of \eqref{eq:sopt_prob} or  \eqref{eq:finite_sum} using the following update:
\begin{equation}\label{eq:non_composite_update}
w_{t+1}^{(s)} := w_t^{(s)} - \hat{\eta}_tv_t^{(s)}
\end{equation}
for both Step~\ref{step:o3} and Step~\ref{step:i3} and using  \eqref{eq:snapshot_v0_2} for the expectation problem and \eqref{eq:vt_finite_sum} for the finite-sum problem.

Let $\rho := \frac{1}{\hat{b}}$ for the expectation problem and $\rho := \frac{n-\hat{b}}{\hat{b}(n-1)}$ for the finite-sum problem, and  the step-size $\hat{\eta}_t$ is computed recursively in a backward mode from $t=m$ down to $t=0$ as
\begin{equation}\label{eq:adaptive_stepsize}
\hat{\eta}_m = \frac{1}{L},~~\text{and}~~\hat{\eta}_{m-t} := \frac{1}{L\big(1 + \rho L\sum_{j=1}^t\hat{\eta}_{m-j+1}\big)},~~~\forall t=1,\cdots, m.
\end{equation}
Then, we have $\Sigma_m := \sum_{t=0}^m\hat{\eta}_t \geq \frac{2(m+1)}{(\sqrt{2\rho m+1}+1)L}$.

\noindent Suppose that Assumptions~\ref{as:A1} and \ref{as:A2} hold, and  $\widetilde{w}_T \sim \Ub_{p}\big(\sets{w_t^{(s)}}_{t=0\to m}^{s=1\to S}\big)$ such that
\begin{equation*} 
\Prob{\widetilde{w}_T = w_t^{(s)}} = p_{(s-1)m+t} := \frac{\hat{\eta}_t}{S\Sigma_m},~~~~\forall s=1,\cdots, S, ~t = 0,\cdots, m.
\end{equation*}
Then, we have
\begin{equation}\label{eq:key_est11}
\begin{array}{ll}
\Exp{\Vert \nabla{f}(\widetilde{w}_T)\Vert^2} &= \dfrac{1}{S\Sigma_m}\displaystyle\sum_{s=1}^S\displaystyle\sum_{t=0}^m\hat{\eta}_t\Exp{\Vert \nabla{f}(w_t^{(s)})\Vert^2} \vspace{0ex}\\
& \leq \dfrac{(\sqrt{2\nu m+1}+1)L}{S(m+1)} \big[f(\widetilde{w}_0) - f^{\star}\big] +  \dfrac{1}{S}\displaystyle\sum_{s=1}^S\hat{\sigma}_s,
\end{array}
\end{equation}
where  $\hat{\sigma}_s := \Exp{\norms{\nabla{f}(w_0^{(s)}) - v_0^{(s)}}^2}$.

We consider two cases:
\begin{itemize}
\vspace{-1ex}
\item[$\mathrm{(a)}$]\textbf{The finite-sum case:} 
If we apply this variant of Algorithm~\ref{alg:prox_sarah} to solve the non-composite instance of \eqref{eq:finite_sum} $($i.e. $\psi = 0$$)$ using full gradient snapshot $b_s := n$, $m := \lfloor \frac{n}{\hat{b}}\rfloor$, and $\hat{b} \in [1, \sqrt{n}]$, then
\begin{equation}\label{eq:key_est11a}
\Exp{\Vert \nabla{f}(\widetilde{w}_T)\Vert^2} \leq \frac{2L}{S\sqrt{n}}[ f(\widetilde{w}_0) - f^{\star}].
\end{equation}
Consequently, the total of outer iterations $S$ to achieve an $\varepsilon$-stationary point $\widetilde{w}_T$ such that $\Exp{\Vert \nabla{f}(\widetilde{w}_T)\Vert^2} \leq \varepsilon^2$ does not exceed $S := \frac{2L [f(\widetilde{w}_0) - f^{\star}]}{\sqrt{n}\varepsilon^2}$.
The number of individual stochastic gradient evaluations $\nabla{f_i}$ does not exceed $\Tc_{\mathrm{grad}} := \frac{10\sqrt{n}L [f(\widetilde{w}_0) - f^{\star}]}{\varepsilon^2}$.

\item[$\mathrm{(b)}$]\textbf{The expectation case:}
If we apply this variant of Algorithm~\ref{alg:prox_sarah} to solve the non-composite expectation instance of \eqref{eq:sopt_prob} $($i.e. $\psi = 0$$)$ using mini-batch size $b_s = b := \frac{2\sigma^2}{\varepsilon^2}$ for the outer-loop, $m :=  \frac{\sigma^2}{\hat{b}\varepsilon^2}$, and $\hat{b} \leq \frac{\sigma^2}{\varepsilon^2}$, then
\begin{equation}\label{eq:key_est11b}
\Exp{\Vert \nabla{f}(\widetilde{w}_T)\Vert^2} \leq \frac{2L}{S\sqrt{\hat{b}m}}\big[ f(\widetilde{w}_0) - f^{\star}\big] + \frac{\sigma^2}{b}.
\end{equation}
Consequently, the total of outer iterations $S$ to achieve an $\varepsilon$-stationary point $\widetilde{w}_T$ such that $\Exp{\Vert \nabla{f}(\widetilde{w}_T)\Vert^2} \leq \varepsilon^2$ does not exceed $S := \frac{4L[ f(\widetilde{w}_0) - f^{\star}]}{\sigma\varepsilon}$.
The number of individual stochastic gradient evaluations does not exceed $\Tc_{\mathrm{grad}} := \frac{16\sigma L[ f(\widetilde{w}_0) - f^{\star}]}{\varepsilon^3}$, provided that $\sigma \leq \frac{8L[ f(\widetilde{w}_0) - f^{\star}]}{\varepsilon}$.
\end{itemize}
\end{theorem} 
%%% End of Theorem 4.1.

Note that the first statement (a) of Theorem~\ref{thm:noncomposite_convergence} covers the nonconvex case of \citep{Nguyen2019_SARAH} by fixing step-size $\hat{\eta}_t = \hat{\eta} = \frac{2}{L(1 + \sqrt{4m+1})}$.
However, this constant step-size is rather small if $m = \BigO{n}$ is large.
Hence, it is better to update $\hat{\eta}_t$ adaptively increasing as in \eqref{eq:adaptive_stepsize}, where $\hat{\eta}_m = \frac{1}{L}$ is a large step-size.
In addition, \citep{Nguyen2019_SARAH} only studies the finite-sum problem. 

Again, by combining the  first statement (a) of Theorem~\ref{thm:noncomposite_convergence} and the lower-bound complexity in  \citep{fang2018spider}, we can conclude that this algorithmic variant still achieves a nearly-optimal complexity  $\BigO{n^{1/2}\varepsilon^{-2}}$ for the non-composite finite-sum problem in \eqref{eq:finite_sum} to find an $\varepsilon$-stationary point in expectation if $n \leq \BigO{\varepsilon^{-4}}$.
In Statement (b), if $\sigma > \frac{8L[ f(\widetilde{w}_0) - f^{\star}]}{\varepsilon}$, then the complexity of our method is $\BigO{\sigma^2\varepsilon^{-2} + \sigma\varepsilon^{-3}}$ due to the gradient snapshot of the size $b = \BigO{\sigma^2\varepsilon^{-2}}$ to evaluate $v_0^{(s)}$.

%%%%%%%%%%%%%%%%%%%%%%%%%%%%%%%%%%%%%%%%%%%%%%%%%%
%%%% 5. Numerical experiments
%%%%%%%%%%%%%%%%%%%%%%%%%%%%%%%%%%%%%%%%%%%%%%%%%%
\beforesec
\section{Numerical experiments}\label{sec:num_experiments}
\aftersec
We present three numerical examples to illustrate our theory and compare our methods with state-of-the-art algorithms in the literature.
We implement $8$ different variants of our ProxSARAH algorithm:
\begin{itemize}
\vspace{-1.5ex}
\item ProxSARAH-v1: Single sample and fixed step-sizes $\gamma \!:=\! \frac{\sqrt{2}}{L\sqrt{3m}}$ and $\eta \!:=\! \frac{2\sqrt{3m}}{4\sqrt{3m} \!+\! \sqrt{2}}$.
\vspace{-1.5ex}
\item ProxSARAH-v2:  $\gamma := 0.95$ and mini-batch size $\hat{b} := \big\lfloor\frac{\sqrt{n}}{C}\big\rfloor$ and $m := \lfloor\sqrt{n}\rfloor$.
\vspace{-1.5ex}
\item ProxSARAH-v3: $\gamma := 0.99$ and mini-batch size $\hat{b} := \big\lfloor\frac{\sqrt{n}}{C}\big\rfloor$ and $m := \lfloor\sqrt{n}\rfloor$.
\vspace{-1.5ex}
\item ProxSARAH-v4:  $\gamma := 0.95$ and mini-batch size $\hat{b} := \big\lfloor\frac{n^{\frac{1}{3}}}{C}\big\rfloor$ and $m := \lfloor n^{\frac{1}{3}}\rfloor$.
\vspace{-1.5ex}
\item ProxSARAH-v5: $\gamma := 0.99$ and mini-batch size $\hat{b} := \big\lfloor\frac{n^{\frac{1}{3}}}{C}\big\rfloor$ and $m := \lfloor n^{\frac{1}{3}}\rfloor$.
\vspace{-1.5ex}
\item ProxSARAH-A-v1:  Single sample (i.e., $\hat{b} = 1$), and adaptive step-sizes.
\vspace{-1.5ex}
\item ProxSARAH-A-v2: $\gamma_m := 0.99$ and mini-batch size $\hat{b} := \lfloor\sqrt{n}\rfloor$ and $m := \lfloor \sqrt{n}\rfloor$.
\vspace{-1.5ex}
\item ProxSARAH-A-v3: $\gamma_m := 0.99$ and mini-batch size $\hat{b} := \lfloor n^{\frac{1}{3}}\rfloor$ and $m := \lfloor n^{\frac{1}{3}}\rfloor$.
\vspace{-1.5ex}
\end{itemize}
Here, $C$ is given in Subsection~\ref{subsubsec:mini_batch_step_size}.
We also implement $4$ other algorithms:
\begin{itemize}
\vspace{-1.5ex}
\item ProxSVRG: The proximal SVRG algorithm in \citep{reddi2016proximal} for single sample with theoretical step-size $\eta = \frac{1}{3nL}$, and for the mini-batch case with $\hat{b} := \lfloor n^{2/3}\rfloor$, the epoch length $m := \lfloor n^{1/3}\rfloor$, and the step-size $\eta := \frac{1}{3L}$.
\vspace{-1.5ex}
\item ProxSpiderBoost: The proximal SpiderBoost method in \citep{wang2018spiderboost} with $\hat{b} := \lfloor\sqrt{n}\rfloor$, $m := \lfloor\sqrt{n}\rfloor$, and step-size $\eta := \frac{1}{2L}$.
\vspace{-1.5ex}
\item ProxSGD: Proximal Stochastic Gradient Descent scheme \citep{ghadimi2013stochastic} with step-size  $\eta_t \!:=\! \frac{\eta_0}{1 + \tilde{\eta}\lfloor t/n\rfloor}$, where $\eta_0 \! > \! 0$ and $\tilde{\eta}  \!\geq\! 0$ will be given in each example.
\vspace{-1.5ex}
\item ProxGD: Standard Proximal Gradient Descent algorithm with step-size $\eta := \frac{1}{L}$.
\vspace{-1.5ex}
\end{itemize}
All the algorithms are implemented in Python running on a single node of a Linux server (called Longleaf) with configuration: 3.40GHz Intel processors, 30M cache, and 256GB RAM.
For the last example, we implement these algorithms in \href{https://www.tensorflow.org}{\textbf{TensorFlow} (https://www.tensorflow.org)} running on a GPU system.
Our code is available online at
\begin{center}
\vspace{-0.5ex}
{\color{blue}\href{https://github.com/unc-optimization/StochasticProximalMethods}{https://github.com/unc-optimization/StochasticProximalMethods}.}
\vspace{-0.5ex}
\end{center}
To be fair for comparison, we compute the norm of gradient mapping $\Vert G_{\eta}(w^{(s)}_t)\Vert$ for visualization at the same value $\eta := 0.5$ in all methods.
\nhanp{We run the first and second examples for $20$ and $30$ epochs, respectively whereas we increase it up to $150$ and $300$ epochs in the last example.}
Several datasets used in this paper are from \citep{CC01a}, which are available online at \href{https://www.csie.ntu.edu.tw/~cjlin/libsvm/}{https://www.csie.ntu.edu.tw/{$\sim$}cjlin/libsvm/}.
Two other well-known datasets are \texttt{mnist} and \texttt{mnist\_fashion} (\href{http://yann.lecun.com/exdb/mnist/}{http://yann.lecun.com/exdb/mnist/}).

%%%%%%%%%%%%%%%%%%%%%%%%%%%%%%%%%%%%%%%%%%
%%% 5.1. Nonnegative principal component analysis
\beforesubsec
\subsection{Nonnegative principal component analysis}\label{subsec:nn_pca_exam}
\aftersubsec
We reconsider the problem of non-negative principal component analysis (NN-PCA) studied in \citep{reddi2016proximal}. 
More precisely, for a given set of samples $\set{z_i}_{i=1}^n$ in $\R^d$,  we solve the following  constrained nonconvex problem:
\begin{equation}\label{eq:nn_pca}
f^{\star} := \min_{w\in\R^d}\Big\{ f(w) := -\frac{1}{2n}\sum_{i=1}^nw^{\top}(z_iz_i^{\top})w \mid \norms{w} \leq 1, ~w \geq 0 \Big\}.
\end{equation}
By defining $f_i(w) := -\frac{1}{2}w^{\top}(z_iz_i^{\top})w$ for $i=1,\cdots, n$, and $\psi(w) := \delta_{\Xc}(w)$, the indicator of $\Xc := \set{w\in\R^d \mid \norms{w} \leq 1, w \geq 0}$, we can formulate \eqref{eq:nn_pca} into \eqref{eq:finite_sum}.
Moreover, since $z_i$ is normalized, the Lipschitz constant of $\nabla{f}_i$ is $L = 1$ for $i=1,\cdots, n$.

\beforepara
\paragraph{\textbf{Small and medium datasets:}}
We test all the algorithms on three different well-known datasets:  \texttt{mnist} ($n=60000$, $d = 784$), \texttt{rcv1-binary} ($n=20242$, $d=47236$), and \texttt{real-sim} ($n=72309$, $d=20958$). 
In ProxSGD, we set $\eta_0 := 0.1$ and $\tilde{\eta} := 1.0$ that allow us  to obtain good performance.

%%% Test 1.
We first verify our theory by running $5$ algorithms with single sample (i.e. $\hat{b} = 1$).
The relative objective residuals and the absolute norm of gradient mappings of these algorithms after $20$ epochs are plotted in Figure~\ref{fig:nn_pca_tes1}.

\begin{figure}[ht!]
%\vspace{-1ex}
\begin{center}
\includegraphics[width = 1\textwidth]{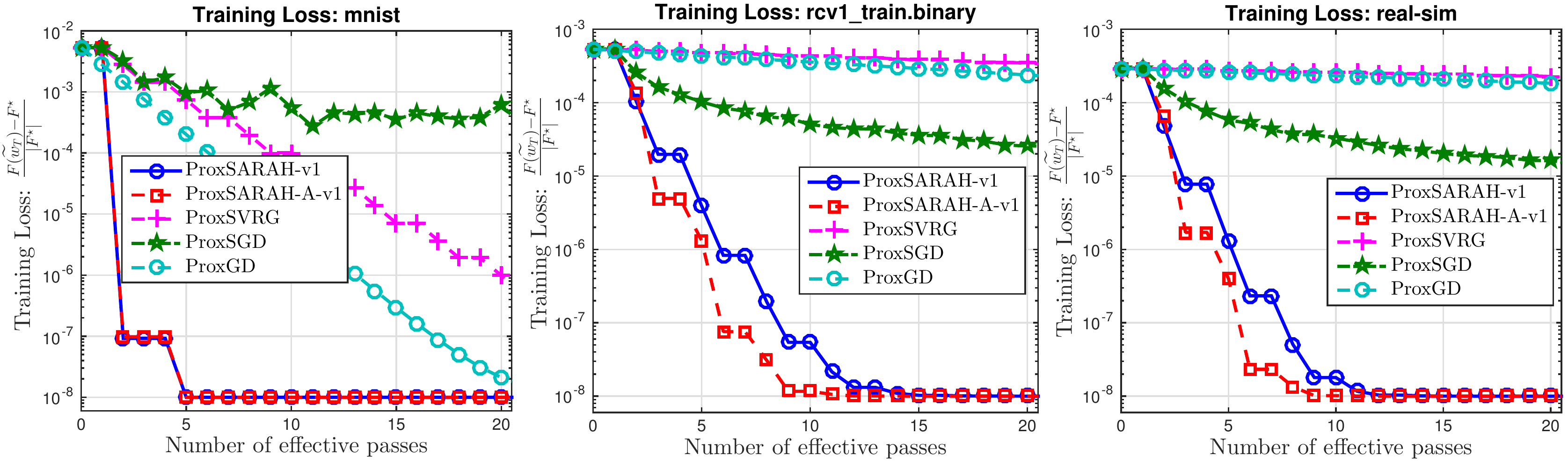} 
\includegraphics[width = 1\textwidth]{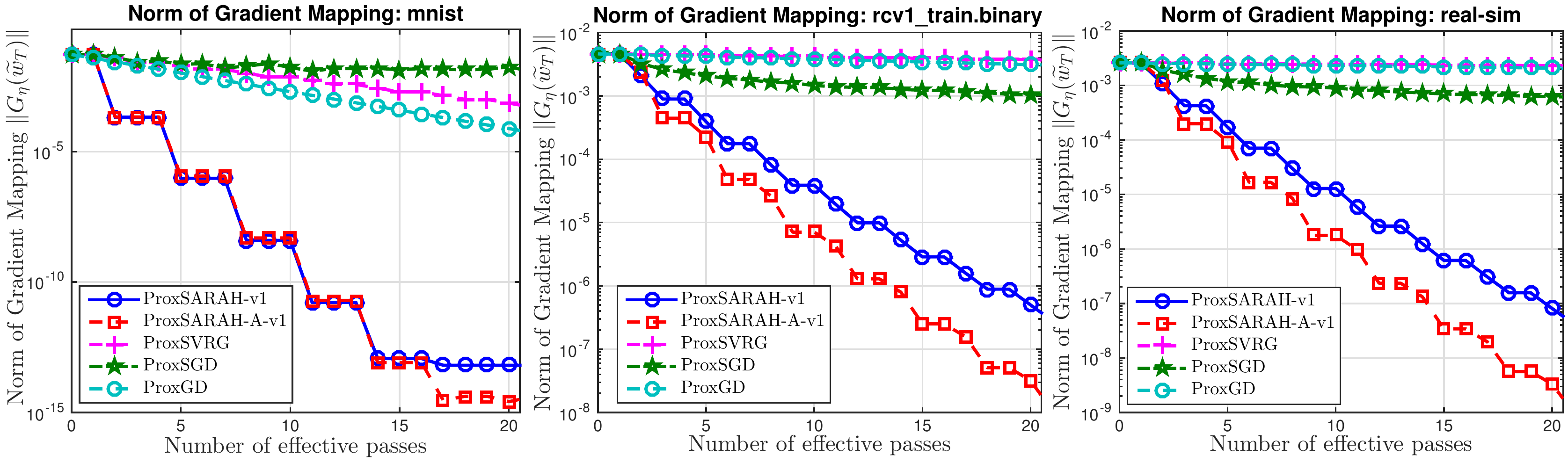}
\vspace{-4ex}
\caption{The objective value residuals and gradient mapping norms of \eqref{eq:nn_pca} on three datasets: \texttt{mnist}, \texttt{rcv1-binary}, and \texttt{real-sim}.}\label{fig:nn_pca_tes1}
\end{center}
\vspace{-3ex}
\end{figure}

Figure~\ref{fig:nn_pca_tes1} shows that both \nhanp{ProxSARAH-v1} and its adaptive variant work really well and dominate all other methods.
ProxSARAH-A-v1 is still better than \nhanp{ProxSARAH-v1}. 
ProxSVRG is slow since its theoretical step-size $\frac{1}{3nL}$ is too small.

%%% Test 2.
Now, we consider the mini-batch case.
In this test, we run all the mini-batch variants of the methods described above. The relative objective residuals and the norms of gradient mapping are plotted in Figure~\ref{fig:nn_pca_tes2}.

\begin{figure}[!hpt]
\vspace{-0ex}
\begin{center}
\includegraphics[width = 1\textwidth]{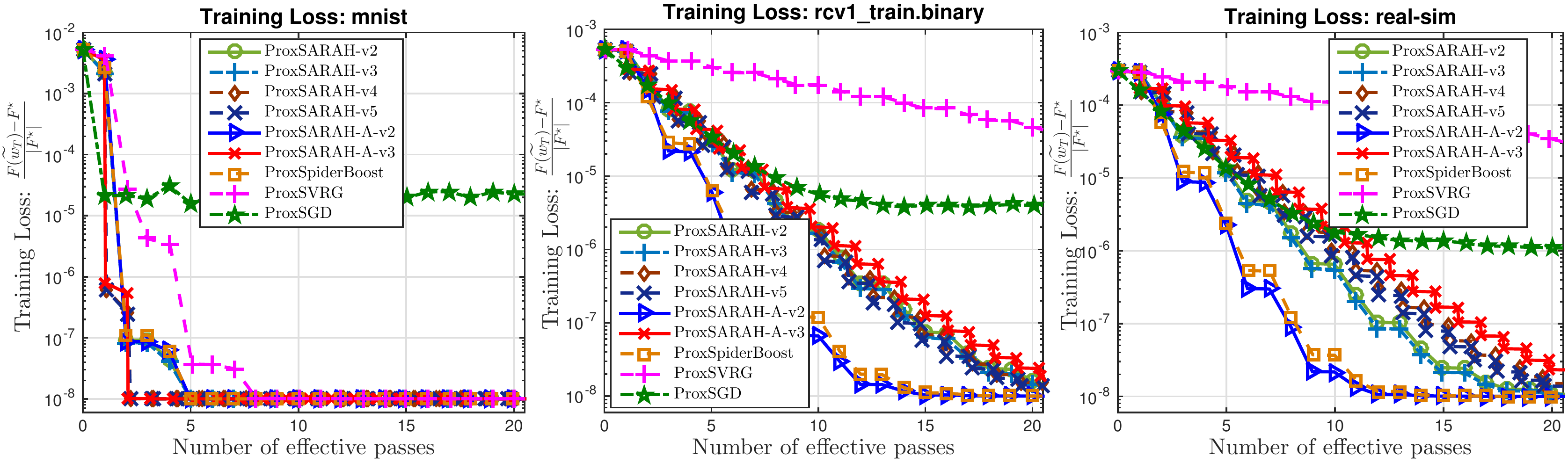}  
\includegraphics[width = 1\textwidth]{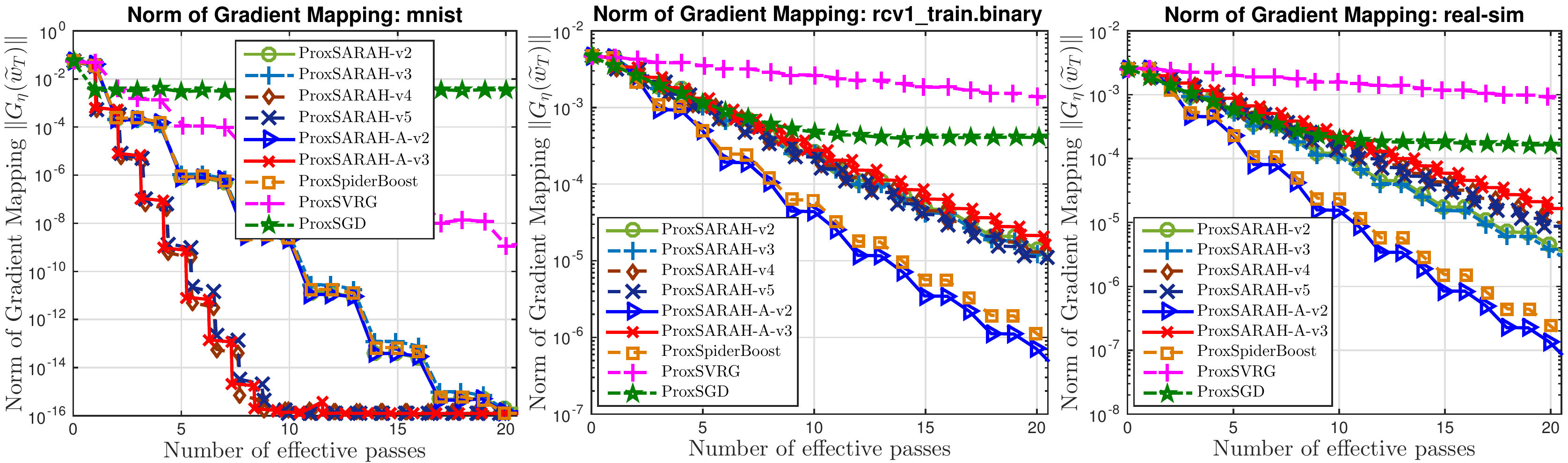}
\vspace{-3ex}
\caption{The relative objective residuals and the norms of gradient mappings of $9$ algorithms for solving \eqref{eq:nn_pca} on three datasets: \texttt{mnist}, \texttt{rcv1-binary}, and \texttt{real-sim}.}\label{fig:nn_pca_tes2}
\end{center}
\vspace{-3ex}
\end{figure}

From Figure~\ref{fig:nn_pca_tes2}, we observe that ProxSpiderBoost works well since it has a large step-size $\eta = \frac{1}{2L}$, and it is comparable with ProxSARAH-A-v2.
Other ProxSARAH variants also work well, and their performance depends on datasets.
Although ProxSVRG takes $\eta = \frac{1}{3L}$, its choice of batch size and epoch length also affects the performance resulting in a slower convergence.
%Other variants of ProxSARAH also work well, but ProxSARAH-v4 and ProxSARAH-v5 are slower than other variants.
ProxSGD works well but then its relative objective residual is saturated around $10^{-5}$ accuracy.
However, its gradient mapping norms do not significantly decrease as in ProxSARAH variants or ProxSpiderBoost.
Note that ProxSARAH variants with large step-size $\gamma$ (e.g., $\gamma = 0.99$) are very similar to ProxSpiderBoost which results in resemblance in their performance. 

%%%% Test 3.
\beforepara
\paragraph{\textbf{Large datasets:}}
Now, we test these algorithms on larger datasets: \texttt{url\_combined} ($n = 2,396,130; d = 3,231,961$), \texttt{news20.binary} ($n=19,996; d=1,355,191$), and \texttt{avazu-app} ($n = 14,596,137; d = 999,990$). 
The relative objective residuals and the absolute norms of gradient mapping of this \nhanp{experiment} are \nhanp{depicted} in Figure~\ref{fig:nn_pca_tes3}.

\begin{figure}[hpt!]
\vspace{-1ex}
\begin{center}
\includegraphics[width = 1\textwidth]{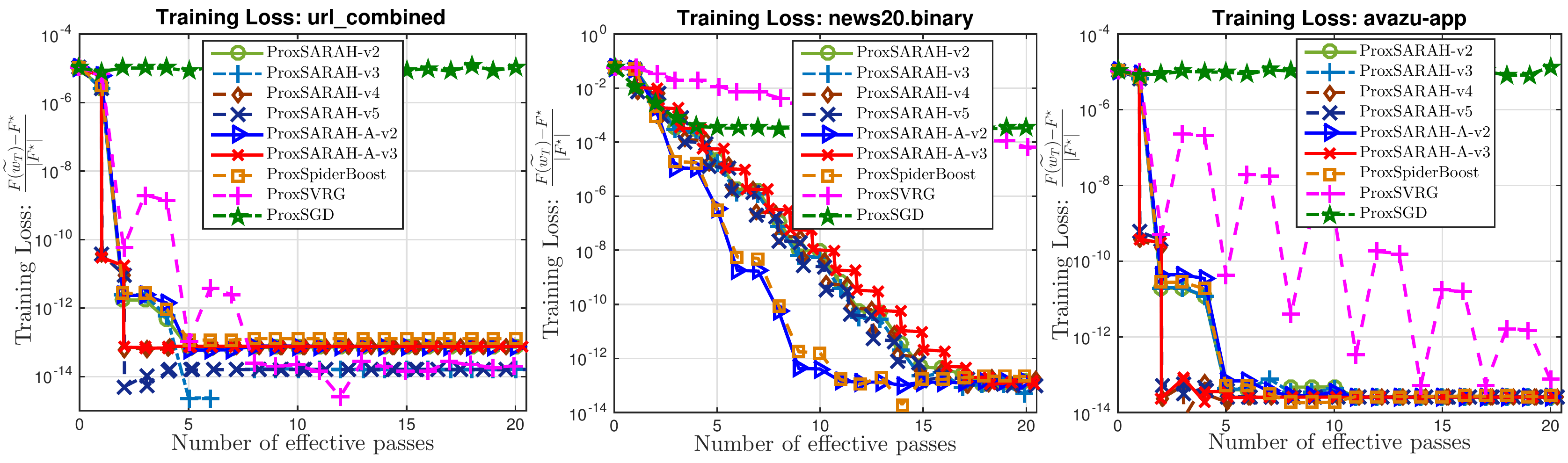}  
\includegraphics[width = 1\textwidth]{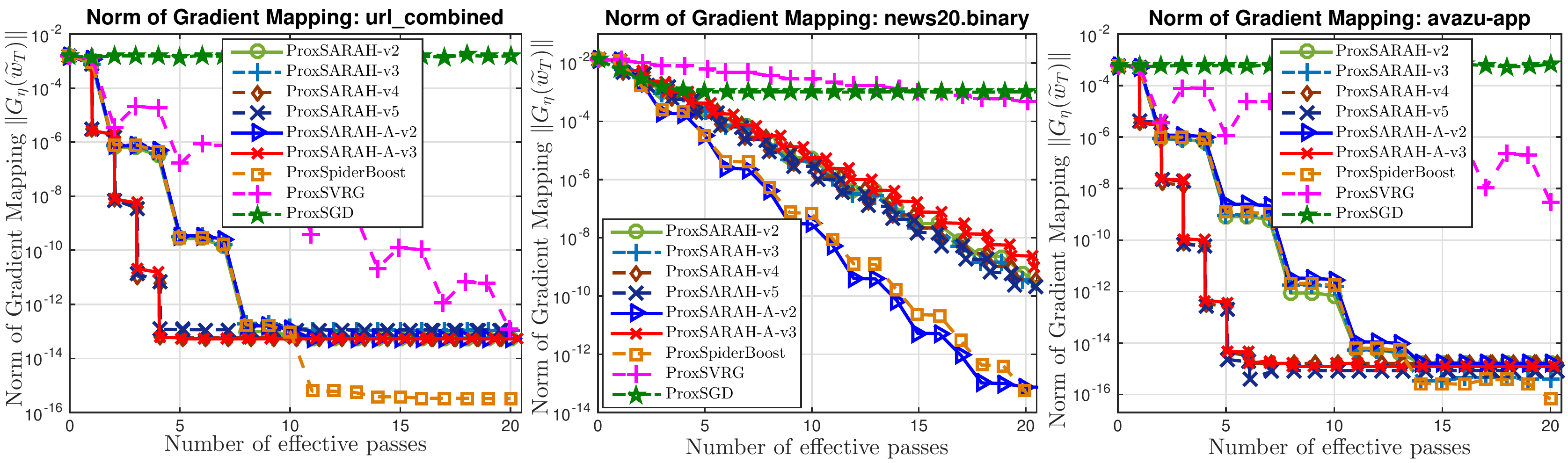}
\vspace{-3ex}
\caption{The relative objective residuals and the gradient mapping norms of $9$ algorithms for solving \eqref{eq:nn_pca} on three datasets: \texttt{url\_combined}, \texttt{news20.binary}, and \texttt{avazu-app}.}\label{fig:nn_pca_tes3}
\end{center}
\vspace{-3ex}
\end{figure}

\nhanp{Figure~\ref{fig:nn_pca_tes3} shows that ProxSARAH variants still work well and depend on the dataset in which ProxSARAH-A-v2 or the variants with $\hat{b} = \mathcal{O}(n^{\frac{1}{3}})$ dominates other algorithms.}
In this experiment, ProxSpiderBoost gives smaller gradient mapping norms for \texttt{url\_combined} and \texttt{avazu-app} in the last epochs than the others. 
However, these algorithms have achieved up to $10^{-13}$ accuracy in absolute values, the improvement of ProxSpiderBoost may not be necessary.
With the same step-size as in the previous test, ProxSGD performs quite poorly in these three datasets.
ProxSVRG does not work well on the \texttt{news20.binary} dataset, but becomes comparable with other methods on \texttt{url\_combined} and \texttt{avazu-app}.

%%%%%%%%%%%%%%%%%%%%%%%%%%%%%%%%%%%%%%%%%%
%%%% 5.2. Sparse binary classification with nonconvex loss
\beforesubsec 
\subsection{Sparse binary classification with nonconvex losses}
\aftersubsec
We consider the following sparse binary classification involving nonconvex loss function:
\vspace{-0.5ex}
\begin{equation}\label{eq:sparse_BinClass_ncvx}
\min_{w\in\R^d}\set{ F(w) := \frac{1}{n}\sum_{i=1}^n\ell(a_i^{\top}w, b_i) + \lambda\norms{w}_1},
\vspace{-0.5ex}
\end{equation}
where $\set{(a_i, b_i)}_{i=1}^n \subset\R^d\times \set{-1,1}^n$ is a given training dataset, $\lambda > 0$ is a regularization parameter, and $\ell(\cdot,\cdot)$ is a given smooth and nonconvex loss function as studied in \citep{zhao2010convex}.
By setting $f_i(w) := \ell(a_i^{\top}w, b_i)$ and $\psi(w) := \lambda\norms{w}_1$ for $i=1,\cdots, n$, we obtain the form \eqref{eq:finite_sum}.

The loss function $\ell$ is chosen from one of the following three cases \citep{zhao2010convex}:
\begin{enumerate}
\vspace{-1ex}
\item \textbf{Normalized sigmoid loss:}
$\ell_1(s, \tau) := 1 - \tanh(\omega \tau s)$ for a given $\omega > 0$.
Since $\left\vert \frac{d^2\ell_1(s,\tau)}{ds^2}\right\vert \leq \frac{8(2+\sqrt{3})(1+\sqrt{3})\omega^2\tau^2}{(3+\sqrt{3})^2}$ and $\vert\tau\vert = 1$, we can show that $\ell_1(\cdot,\tau)$ is $L$-smooth with respect to $s$, where $L :=  \frac{8(2+\sqrt{3})(1+\sqrt{3})\omega^2}{(3+\sqrt{3})^2} \approx 0.7698\omega^2$.

\vspace{-1.5ex}
\item \textbf{Nonconvex loss in 2-layer neural networks:}
$\ell_2(s, \tau) := \left(1 - \frac{1}{1 + \exp(-\tau s)}\right)^2$.
For this function, we have $\left\vert \frac{d^2\ell_2(s,\tau)}{ds^2} \right\vert \leq 0.15405\tau^2$.
If $\vert \tau \vert = 1$, then this function is also $L$-smooth with $L = 0.15405$.

\vspace{-1.5ex}
\item \textbf{Logistic difference loss:} $\ell_3(s,\tau) := \ln(1 + \exp(-\tau s)) - \ln(1 + \exp(-\tau s - \omega))$ for some $\omega > 0$.
With $\omega = 1$, we have  $\vert \frac{d^2\ell_3(s,\tau)}{ds^2}\vert \leq 0.092372\tau^2$.
Therefore, if $\vert\tau\vert = 1$, then this function is also $L$-smooth with $L = 0.092372$.
\vspace{-1ex}
\end{enumerate}
We set the regularization parameter \nhanp{$\lambda := \frac{1}{n}$} in all the tests, which gives us relatively sparse solutions.
We test the above algorithms on different scenarios ranging from small to large datasets.

%%%% Test 1.
\beforepara
\paragraph{\textbf{Small and medium  datasets:}}
We consider three small to medium datasets: \texttt{rcv1.binary} ($n=20,242$, $d=47,236$), \texttt{real-sim} $(n = 72,309$, $d = 20,958$), and \texttt{epsilon} ($n = 400,000$, $d =2,000$).

\begin{figure}[htp!]
%\vspace{-1ex}
\begin{center}
\includegraphics[width = 1\textwidth]{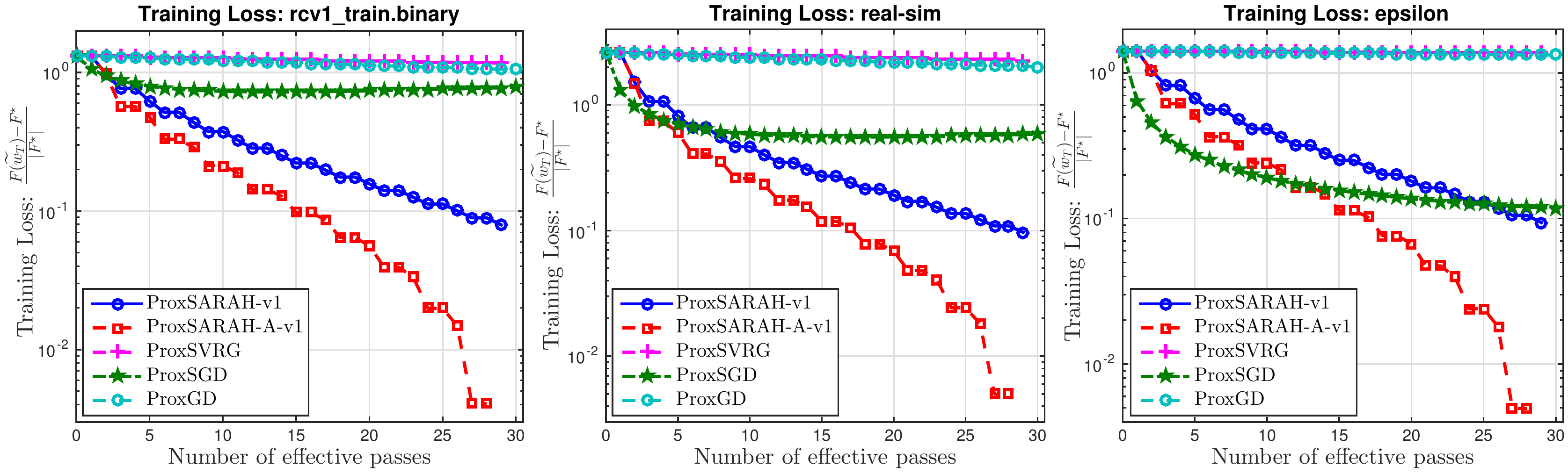}\\
\includegraphics[width = 1\textwidth]{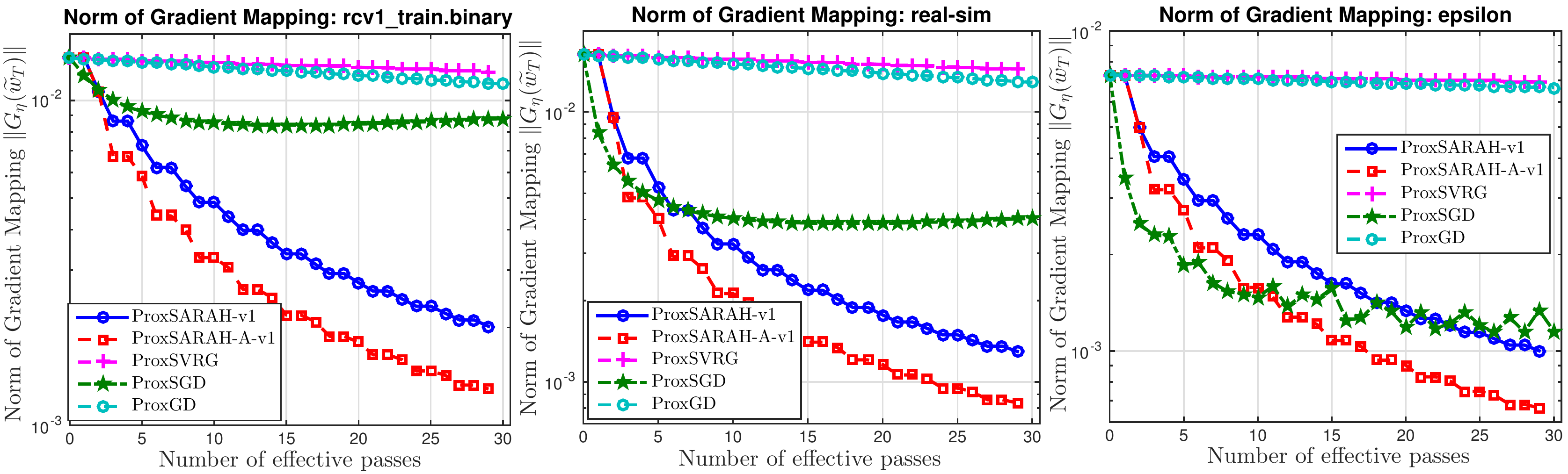}
\vspace{-3ex}
\caption{The relative objective residuals and gradient mapping norms  of \eqref{eq:sparse_BinClass_ncvx} on three datasets using the loss $\ell_2(s, \tau)$ - The single sample case.}\label{fig:bin_class_test1}
\end{center}
\vspace{-3ex}
\end{figure}

Figure \ref{fig:bin_class_test1} shows the relative objective residuals and the gradient mapping norms on these three datasets for the loss function $\ell_2(\cdot)$ in the single sample case.
Similar to the first example, ProxSARAH-v1 and its adaptive variant work well, whereas ProxSARAH-A-v1 is better.
ProxSVRG is still slow due to small step-size.
ProxSGD appears to be better than ProxSVRG and ProxGD within $30$ epochs.

%%%% Test 2.
Now, we test the loss function $\ell_2(\cdot)$ with the mini-batch variants using the same three datasets.
Figure~\ref{fig:bin_class_test1b} shows the results of $9$ algorithms on these datasets.
\begin{figure}[htp!]
%\vspace{-1ex}
\begin{center}
\includegraphics[width = 1\textwidth]{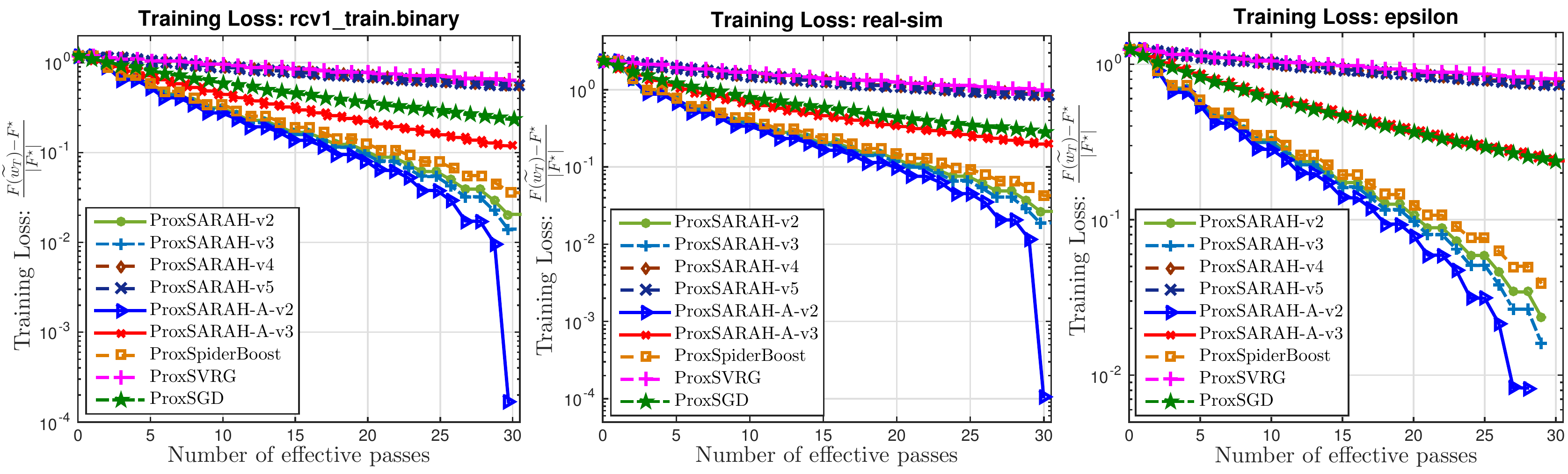}\\
\includegraphics[width = 1\textwidth]{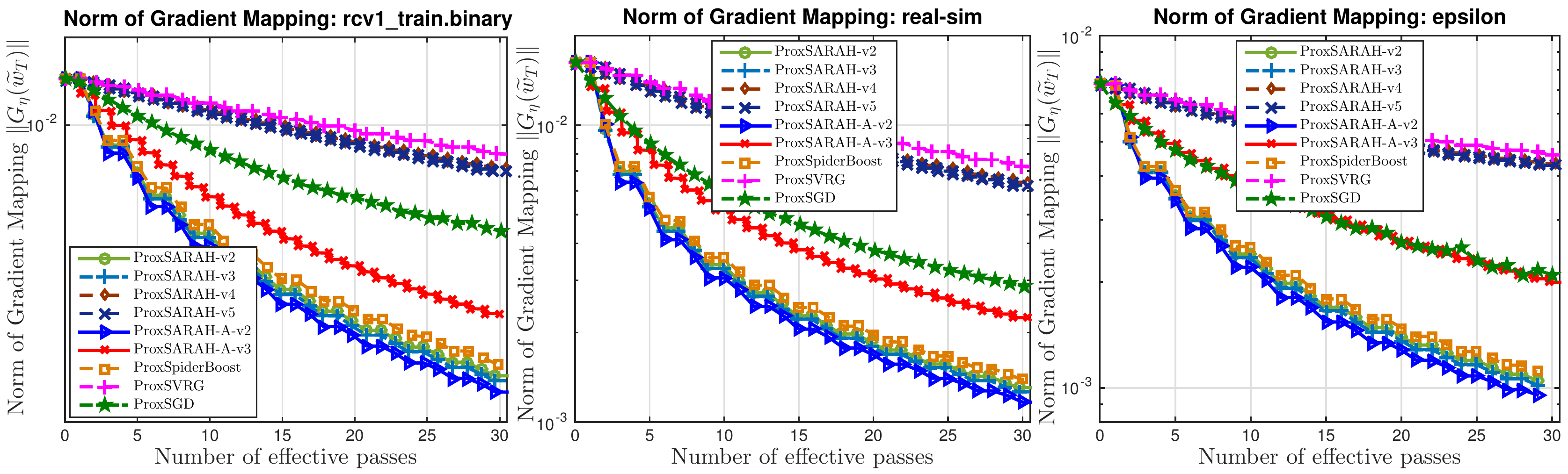}
\vspace{-3ex}
\caption{The relative objective residuals and gradient mapping norms  of \eqref{eq:sparse_BinClass_ncvx} on three datasets using the loss $\ell_2(s, \tau)$ - The mini-batch case.}\label{fig:bin_class_test1b}
\end{center}
\vspace{-3ex}
\end{figure}

We can see that ProxSARAH-A-v2 is the most effective algorithm whereas ProxSpiderBoost also performs well due to large step-size as discussed.
ProxSVRG remains slow in this test, and has similar performance as  ProxSARAH-v4 and -v5 since they all use the same epoch length. \nhanp{Notice that ProxSARAH adaptive variants normally work better than their corresponding fixed step-size variants in this experiment. Additionally,} ProxSARAH-A-v2 still preserves the best-known complexity $\BigO{n + n^{1/2}\varepsilon^{-2}}$.

%%%% Test 3.
\beforepara
\paragraph{\textbf{Large datasets:}}
Next, we test  these algorithms on three large datasets: \texttt{url\_combined} ($n = 2,396,130$, $d= 3,231,961$), \texttt{avazu-app} ($n = 14,596,137$, $d = 999,990$), and \texttt{kddb-raw} ($n = 19,264,097$, $d = 3,231,961$).
Figure~\ref{fig:bin_class_test1c} \nhanp{presents} the results of different algorithms on \nhanp{these} datasets.

\begin{figure}[htp!]
\begin{center}
\includegraphics[width = 1\textwidth]{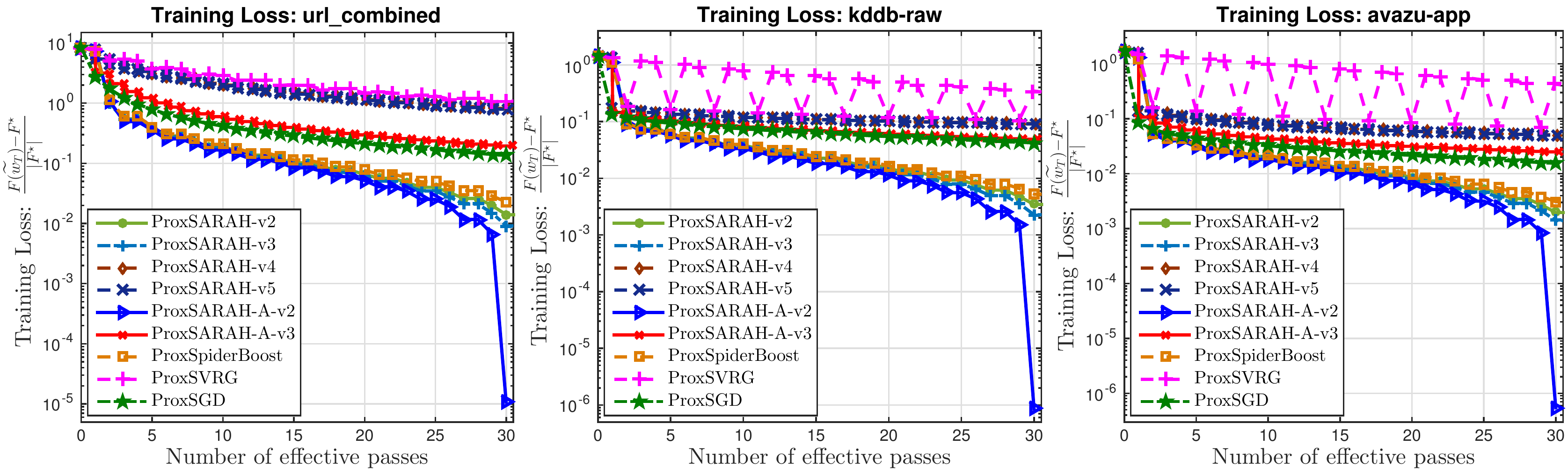}
\includegraphics[width = 1\textwidth]{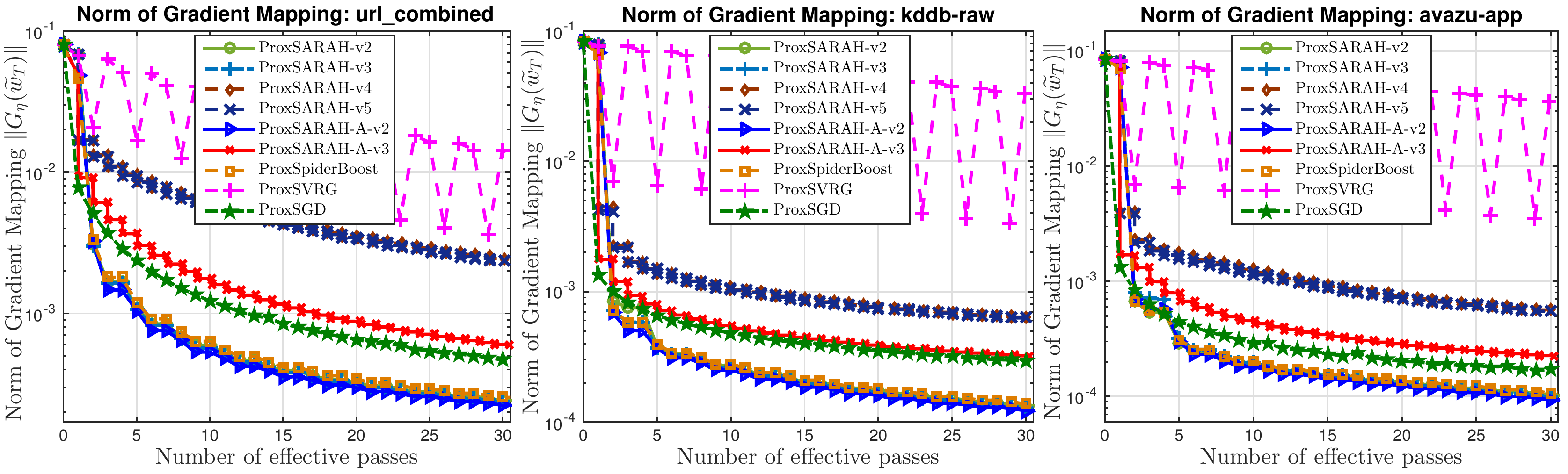}
\vspace{-1ex}
\caption{The relative objective residuals and gradient mapping norms  of \eqref{eq:sparse_BinClass_ncvx} on three large datasets using the loss $\ell_2(s, \tau)$ - The mini-batch case.}\label{fig:bin_class_test1c}
\end{center}
\vspace{-4ex}
\end{figure}

Again, we can observe from Figure~\ref{fig:bin_class_test1c} that, ProxSARAH-A-v2 achieves the best performance.
ProxSpiderBoost also works well in this experiment while ProxSVRG are comparable with ProxSARAH-v1 and ProxSARAH-v2.
ProxSGD also has similar performance as in ProxSARAH-A-v3.

The complete results of $9$ algorithms on these three datasets with three loss functions are presented in Table~\ref{tbl:binary_classification_test}.
Apart from the relative objective residuals and gradient mapping norms, it consists of both training and test accuracies where we use $10\%$ of the dataset to evaluate the test accuracy.

\begin{table}[!t] %hpt!]
\newcommand{\cellb}[1]{{}{{\bf\color{blue}#1}}{}}
\newcommand{\cellr}[1]{{}{\color{red}#1}{}}
\newcommand{\cell}[1]{{}#1{}}
\begin{center}
\caption{The results of $9$ algorithms on three datasets: \textrm{url\_combined}, \textrm{avazu-app}, and \textrm{kddb-raw}.}\label{tbl:binary_classification_test}
\resizebox{\textwidth}{!}{	
\begin{tabular}{| l | r r r | r r r | r r r | r r r |} \toprule
\multirow{2}{*}{\text{Algorithms}} & \multicolumn{3}{c|}{$\Vert G_{\eta}(\widetilde{w}_T)\Vert^2$ } & \multicolumn{3}{c|}{$(F(w_T) - F^{\star})/\vert F^{\star}\vert$} & \multicolumn{3}{c|}{Training Accuracy} & \multicolumn{3}{c|}{Test Accuracy} \\ 
\cmidrule{2-13}
		& \cell{$\ell_1$-Loss} & \cell{$\ell_2$-Loss} & \cell{$\ell_3$-Loss} & \cell{$\ell_1$-Loss} & \cell{$\ell_2$-Loss} & \cell{$\ell_3$-Loss} &  \cell{$\ell_1$-Loss} & \cell{$\ell_2$-Loss} & \cell{$\ell_3$-Loss} &  \cell{$\ell_1$-Loss} & \cell{$\ell_2$-Loss} & \cell{$\ell_3$-Loss}  \\ 
%		\cmidrule{2-13}
\midrule
 \multicolumn{13}{|c|}{\texttt{url\_combined} ($n = 2,396,130$, $d= 3,231,961$)}\\ 
\midrule
\cell{ProxSARAH-v2} & \cell{2.534e-06} & \cell{5.827e-08} & \cell{1.181e-07}  & \cell{1.941e-01} & \cell{1.397e-02} & \cell{8.092e-02} & \cell{0.965} & \cell{0.9684} & \cell{0.9657} & \cell{0.9636} & \cell{0.9672} & \cell{0.9646} \\ 
\cell{ProxSARAH-v3} & \cell{2.772e-06} & \cell{5.515e-08} & \cell{1.110e-07}  & \cell{2.065e-01} & \cell{9.149e-03} & \cell{7.399e-02} & \cell{0.965} & \cell{0.9685} & \cell{0.9658} & \cell{0.9635} & \cell{0.9673} & \cell{0.9647} \\ 
\cell{ProxSARAH-v4} & \cell{1.252e-05} & \cell{6.003e-06} & \cell{1.433e-05}  & \cell{4.749e-01} & \cell{8.210e-01} & \cell{1.597e+00} & \cell{0.962} & \cell{0.9617} & \cell{0.9558} & \cell{0.9614} & \cell{0.9607} & \cell{0.9528} \\ 
\cell{ProxSARAH-v5} & \cell{1.182e-05} & \cell{5.595e-06} & \cell{1.346e-05}  & \cell{4.617e-01} & \cell{7.931e-01} & \cell{1.546e+00} & \cell{0.962} & \cell{0.9617} & \cell{0.9568} & \cell{0.9615} & \cell{0.9609} & \cell{0.9537} \\ 
\cell{ProxSARAH-A-v2} & \cell{1.115e-06} & \cellb{4.969e-08} & \cellb{5.215e-08}  & \cell{9.225e-02} & \cellb{1.076e-05} & \cellb{1.268e-05} & \cellb{0.966} & \cellb{0.9687} & \cellb{0.9672} & \cell{0.9645} & \cellb{0.9676} & \cellb{0.9662} \\ 
\cell{ProxSARAH-A-v3} & \cell{1.034e-05} & \cell{3.639e-07} & \cell{4.555e-07}  & \cell{4.325e-01} & \cell{1.946e-01} & \cell{2.619e-01} & \cell{0.962} & \cell{0.9644} & \cell{0.9634} & \cell{0.9616} & \cell{0.9631} & \cell{0.9625} \\ 
\cell{ProxSpiderBoost} & \cell{1.375e-06} & \cell{6.454e-08} & \cell{7.158e-08}  & \cell{1.178e-01} & \cell{2.274e-02} & \cell{2.947e-02} & \cell{0.965} & \cell{0.9681} & \cell{0.9664} & \cell{0.9641} & \cell{0.9669} & \cell{0.9653} \\ 
\cell{ProxSVRG} & \cell{7.391e-03} & \cell{2.043e-04} & \cell{2.697e-04}  & \cell{2.196e+00} & \cell{1.091e+00} & \cell{1.490e+00} & \cell{0.958} & \cell{0.9601} & \cell{0.9595} & \cell{0.9570} & \cell{0.9585} & \cell{0.9579} \\ 
\cell{ProxSGD} & \cellb{5.005e-07} & \cell{2.340e-07} & \cell{5.963e-07}  & \cellb{4.446e-03} & \cell{1.406e-01} & \cell{3.062e-01} & \cell{0.968} & \cell{0.9651} & \cell{0.9633} & \cellb{0.9667} & \cell{0.9637} & \cell{0.9624} \\ 
\midrule
\multicolumn{13}{|c|}{\texttt{avazu-app} ($n = 14,596,137$, $d = 999,990$)}\\ 
\midrule
\cell{ProxSARAH-v2} & \cell{8.647e-09} & \cell{1.053e-08} & \cell{5.074e-10}  & \cell{4.354e-04} & \cell{1.958e-03} & \cell{1.687e-04} & \cell{0.883} & \cell{0.8843} & \cell{0.8834} & \cell{0.8615} & \cell{0.8617} & \cell{0.8615} \\ 
\cell{ProxSARAH-v3} & \cell{9.757e-09} & \cell{9.792e-09} & \cell{4.776e-10}  & \cell{4.615e-04} & \cell{1.397e-03} & \cell{1.554e-04} & \cell{0.883} & \cellb{0.8844} & \cell{0.8834} & \cell{0.8615} & \cell{0.8617} & \cell{0.8615} \\ 
\cell{ProxSARAH-v4} & \cell{9.087e-08} & \cell{3.179e-07} & \cell{1.841e-07}  & \cell{1.738e-03} & \cell{5.102e-02} & \cell{9.816e-03} & \cell{0.883} & \cell{0.8834} & \cell{0.8834} & \cell{0.8615} & \cell{0.8615} & \cell{0.8615} \\ 
\cell{ProxSARAH-v5} & \cell{8.568e-08} & \cell{3.029e-07} & \cell{1.702e-07}  & \cell{1.675e-03} & \cell{5.036e-02} & \cell{9.433e-03} & \cell{0.883} & \cell{0.8834} & \cell{0.8834} & \cell{0.8615} & \cell{0.8615} & \cell{0.8615} \\ 
\cell{ProxSARAH-A-v2} & \cell{3.062e-09} & \cellb{8.724e-09} & \cellb{1.814e-10}  & \cell{2.046e-04} & \cellb{5.467e-07} & \cellb{1.388e-08} & \cell{0.883} & \cellb{0.8844} & \cell{0.8834} & \cell{0.8615} & \cell{0.8617} & \cell{0.8615} \\ 
\cell{ProxSARAH-A-v3} & \cell{7.784e-08} & \cell{5.124e-08} & \cell{4.405e-09}  & \cell{1.604e-03} & \cell{2.499e-02} & \cell{1.223e-03} & \cell{0.883} & \cell{0.8834} & \cell{0.8834} & \cell{0.8615} & \cell{0.8615} & \cell{0.8615} \\ 
\cell{ProxSpiderBoost} & \cell{4.050e-09} & \cell{1.152e-08} & \cell{2.579e-10}  & \cell{2.626e-04} & \cell{3.090e-03} & \cell{5.073e-05} & \cell{0.883} & \cell{0.8842} & \cell{0.8834} & \cell{0.8615} & \cell{0.8617} & \cell{0.8615} \\ 
\cell{ProxSVRG} & \cell{4.218e-03} & \cell{1.309e-03} & \cell{1.202e-04}  & \cell{3.137e-01} & \cell{4.287e-01} & \cell{2.031e-01} & \cell{0.883} & \cell{0.8648} & \cell{0.8834} & \cell{0.8615} & \cell{0.8146} & \cell{0.8615} \\ 
\cell{ProxSGD} & \cellb{9.063e-10} & \cell{2.839e-08} & \cell{3.150e-09}  & \cellb{6.449e-06} & \cell{1.595e-02} & \cell{9.536e-04} & \cell{0.883} & \cell{0.8835} & \cell{0.8834} & \cell{0.8615} & \cell{0.8616} & \cell{0.8615} \\ 
\midrule
\multicolumn{13}{|c|}{\texttt{kddb-raw} ($n = 19,264,097$, $d = 3,231,961$)}\\ 
\midrule
\cell{ProxSARAH-v2} & \cell{2.013e-08} & \cell{1.770e-08} & \cell{5.688e-09}  & \cell{7.235e-04} & \cell{3.455e-03} & \cell{4.295e-03} & \cell{0.862} & \cell{0.8654} & \cell{0.8619} & \cell{0.8531} & \cell{0.8560} & \cell{0.8534} \\ 
\cell{ProxSARAH-v3} & \cell{2.168e-08} & \cell{1.669e-08} & \cell{6.105e-09}  & \cell{7.903e-04} & \cell{2.275e-03} & \cell{3.741e-03} & \cell{0.862} & \cell{0.8655} & \cell{0.8619} & \cell{0.8530} & \cell{0.8561} & \cell{0.8534} \\ 
\cell{ProxSARAH-v4} & \cell{2.265e-07} & \cell{4.066e-07} & \cell{2.796e-07}  & \cell{3.862e-03} & \cell{9.196e-02} & \cell{2.203e-02} & \cell{0.862} & \cell{0.8617} & \cell{0.8615} & \cell{0.8530} & \cell{0.8533} & \cell{0.8531} \\ 
\cell{ProxSARAH-v5} & \cell{2.127e-07} & \cell{3.943e-07} & \cell{2.600e-07}  & \cell{3.725e-03} & \cell{9.098e-02} & \cell{2.152e-02} & \cell{0.862} & \cell{0.8617} & \cell{0.8615} & \cell{0.8530} & \cell{0.8533} & \cell{0.8531} \\ 
\cell{ProxSARAH-A-v2} & \cell{7.955e-09} & \cellb{1.490e-08} & \cellb{2.830e-09}  & \cell{2.106e-04} & \cellb{8.502e-07} & \cell{2.829e-03} & \cell{0.862} & \cellb{0.8656} & \cellb{0.8621} & \cell{0.8531} & \cellb{0.8562} & \cellb{0.8536} \\ 
\cell{ProxSARAH-A-v3} & \cell{1.951e-07} & \cell{1.036e-07} & \cell{9.293e-09}  & \cell{3.539e-03} & \cell{4.887e-02} & \cell{9.223e-03} & \cell{0.862} & \cell{0.8627} & \cell{0.8616} & \cell{0.8530} & \cell{0.8544} & \cell{0.8531} \\ 
\cell{ProxSpiderBoost} & \cell{9.867e-09} & \cell{1.906e-08} & \cell{6.889e-09}  & \cell{3.082e-04} & \cell{5.249e-03} & \cellb{5.026e-07} & \cell{0.862} & \cell{0.8652} & \cell{0.8619} & \cell{0.8531} & \cell{0.8559} & \cell{0.8534} \\ 
\cell{ProxSVRG} & \cell{1.225e-02} & \cell{1.105e-03} & \cell{5.040e-04}  & \cell{3.541e-01} & \cell{3.471e-01} & \cell{2.780e-01} & \cell{0.860} & \cell{0.8611} & \cell{0.8599} & \cell{0.8518} & \cell{0.8529} & \cell{0.8519} \\ 
\cell{ProxSGD} & \cellb{6.027e-09} & \cell{8.899e-08} & \cell{1.331e-08}  & \cellb{2.593e-05} & \cell{4.320e-02} & \cell{9.937e-03} & \cell{0.862} & \cell{0.8629} & \cell{0.8616} & \cell{0.8530} & \cell{0.8546} & \cell{0.8531} \\ 
\bottomrule
\end{tabular}}
\end{center}
\vspace{-4ex}
\end{table}

Among three loss functions, the loss $\ell_2$ gives the best training and testing accuracy.
The accuracy is consistent with the result reported in \cite{zhao2010convex}.
ProxSGD seems to give a good results on the $\ell_1$-loss, but ProxSARAH-A-v2 is the best for the $\ell_2$ and \nhanp{$\ell_3$}-losses in the majority of the test.

%%%%%%%%%%%%%%%%%%%%%%%%%%%%%%%%%%%%%%%%%%
%%%% 5.3. Feedforward Neural networks
\beforesubsec
\subsection{Feedforward Neural Network Training problem}
\aftersubsec
We consider the following composite nonconvex optimization model arising from a feedforward neural network configuration:
\begin{equation}\label{eq:fwnn_exam}
\min_{w\in\R^d}\set{ F(w) := \frac{1}{n}\sum_{i=1}^n\ell\big( h(w, a_i), b_i\big) +  \psi(w)},
\end{equation}
where we concatenate all the weight matrices and bias vectors of the neural network in one vector of variable $w$, $\set{(a_i, b_i)}_{i=1}^n$ is a training dataset, $h(\cdot)$ is a composition between all linear transforms and activation functions as $h(w, a) := \bsigma_l(W_l\bsigma_{l-1}(W_{l-1}\bsigma_{l-2}(\cdots \bsigma_0(W_0a + \mu_0) \cdots ) + \mu_{l-1}) + \mu_l)$, where $W_i$ is a weight matrix, $\mu_i$ is a bias vector, $\bsigma_i$ is an activation function, $l$ is the number of layers,  $\ell(\cdot)$ is the soft-max cross-entropy loss, and $\psi$ is a convex regularizer (e.g., $\psi(w) := \lambda\norms{w}_1$ for some $\lambda > 0$ to obtain sparse weights).
Again, by defining $f_i(w) := \ell(h(w, a_i), b_i)$ for $i=1,\cdots, n$, we can bring \eqref{eq:fwnn_exam} into the same composite finite-sum setting \eqref{eq:finite_sum}.

We implement our algorithms and other methods in TensorFlow and use two datasets \texttt{mnist} and \texttt{fashion\_mnist} to evaluate their performance.
In the first experiment, we use a one-hidden-layer fully connected neural network: $784 \times 100 \times 10$ for both \texttt{mnist} and \texttt{fashion\_mnist}.
The activation function $\bsigma_i$ of the hidden layer is ReLU and the loss function is soft-max cross-entropy.
To estimate the Lipschitz constant $L$, we normalize the input data. 
The regularization parameter $\lambda$ is set at \nhanp{$\lambda := \frac{1}{n}$} and $\psi(\cdot) := \lambda\norm{\cdot}_1$.

We first test ProxSARAH, ProxSVRG, ProxSpiderBoost, and ProxSGD using mini-batch.
For ProxSGD, we use the mini-batch $\hat{b} = 245$, $\eta_0 = 0.1$, and $\tilde{\eta} = 0.5$ for both datasets.
For the  \texttt{mnist} dataset, we tune $L = 1$ then follow the configuration in Subsection~\ref{subsubsec:mini_batch_step_size} to  choose $\eta$, $\gamma$, $m$, and $\hat{b}$ for ProxSARAH variants. 
We also tune the learning rate for ProxSVRG at $\eta = 0.2$, and for ProxSpiderBoost at $\eta = 0.12$.
However, for the \texttt{fashion\_mnist} dataset, it requires a smaller learning rate.
Therefore, we choose \nhanp{$L = 4$} for ProxSARAH and follow the theory in Subsection~\ref{subsubsec:mini_batch_step_size} to set $\eta$, $\gamma$, $m$, and $\hat{b}$.
We also tune the learning rate for ProxSVRG and ProxSpiderBoost until they are stabilized to obtain the best possible step-size in this example as \nhanp{$\eta_{\mathrm{ProxSVRG}} = 0.11$} and \nhanp{$\eta_{\mathrm{ProxSpiderBoost}} = 0.15$}, respectively.

Figure~\ref{fig:nn_exam1_mnist_150} shows the convergence of different variants of ProxSARAH, ProxSpiderBoost, ProxSVRG, and ProxSGD on three criteria for  \texttt{mnist}: training loss values, the absolute norm of gradient mapping, and the test accuracy.

\begin{figure}[hpt!]
\vspace{-1ex}
\begin{center}
\includegraphics[width = 1\textwidth]{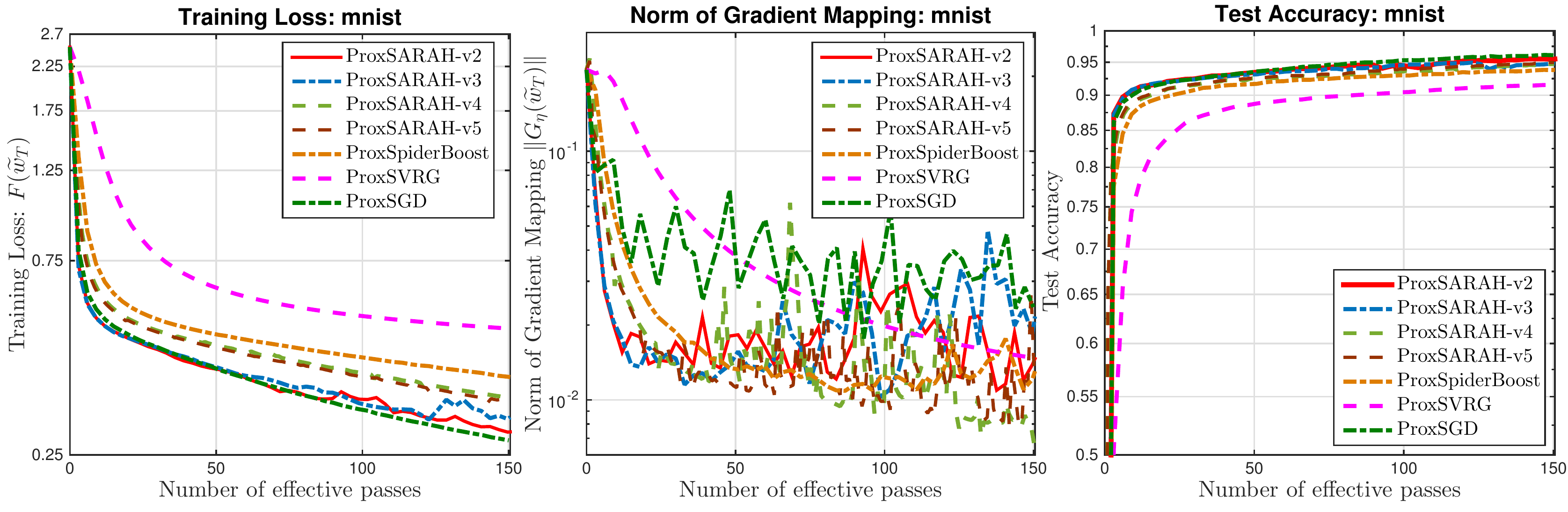}  
\includegraphics[width = 1\textwidth]{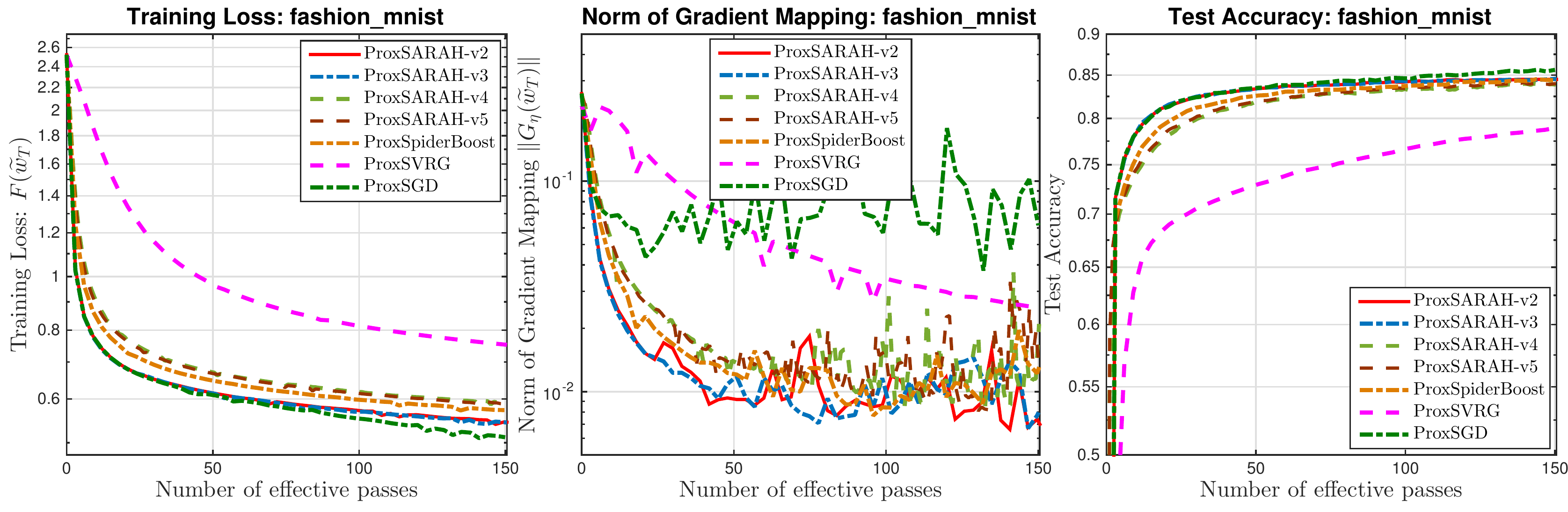} 
\vspace{-2ex}
\caption{The training loss, gradient mapping, and test accuracy on \texttt{mnist} $($top line$)$ and \texttt{fashion\_mnist} $($bottom line$)$ of $7$ algorithms.}\label{fig:nn_exam1_mnist_150}
\end{center}
\vspace{-3ex}
\end{figure}

In this example, ProxSGD appears to be the best in terms of training loss and test accuracy. However, the norm of gradient mapping is rather different from others, relatively large, and oscillated.
ProxSVRG is clearly slower than ProxSpiderBoost due to smaller learning rate.
The four variants of ProxSARAH perform relatively well, but the first and second variants seem to be slightly better. 
Note that the norm of gradient mapping tends to be decreasing but still oscillated since perhaps we are taking the last iterate instead of a random choice of intermediate iterates as stated in the theory.

Finally, we test the above algorithm on \texttt{mnist} using a $784\times 800\times 10$ network as known to give a better test accuracy.
We run all 7 algorithms for 300 epochs and the result is given in Figure~\ref{fig:nn_exam3_mnist_150}.

\begin{figure}[hpt!]
\begin{center}
\includegraphics[width = 1\textwidth]{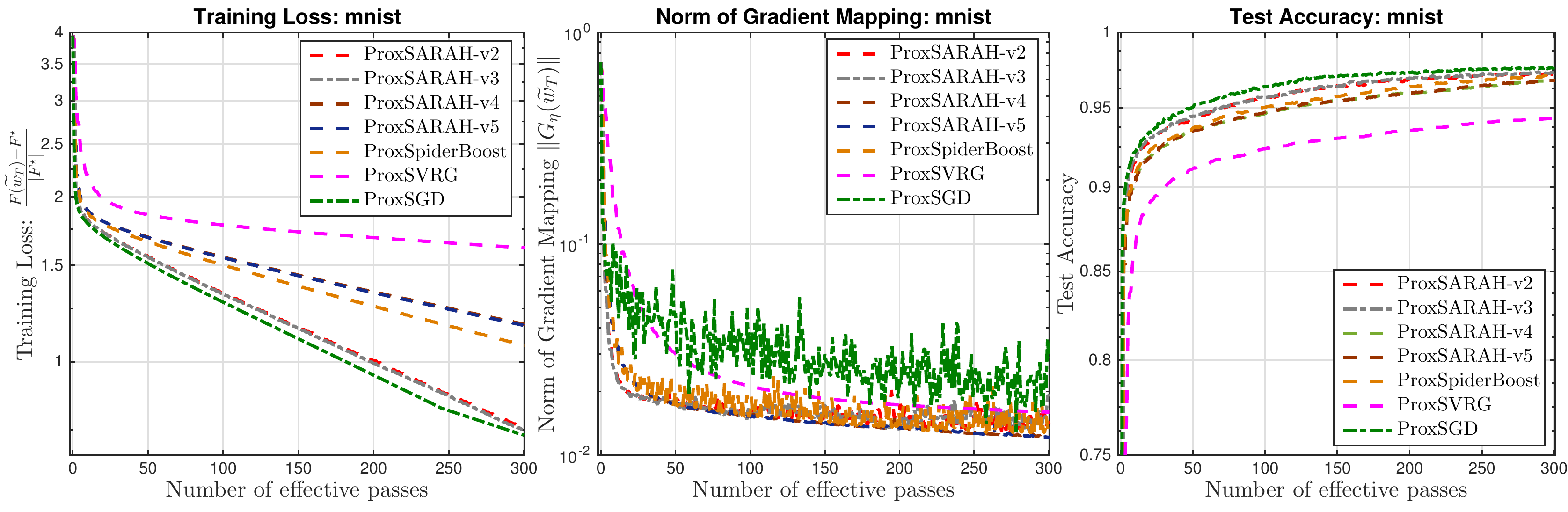}
\vspace{-2ex}
\caption{The training loss, gradient mapping, and test accuracy on \texttt{mnist} of $7$ algorithms on a $784\times 800\times 10$ neural network $($See \href{http://yann.lecun.com/exdb/mnist/}{http://yann.lecun.com/exdb/mnist/}$)$.}\label{fig:nn_exam3_mnist_150}
\end{center}
\vspace{-3ex}
\end{figure}

As we can see from Figure~\ref{fig:nn_exam3_mnist_150} that ProxSARAH-v2, ProxSARAH-v3, and ProxSGD performs really well in terms of training loss and test accuracy.
However, our method can achieve lower as well as less oscillated gradient mapping norm than ProxSGD.
Also, ProxSpiderBoost has similar performance to ProxSARAH-v4 and ProxSARAH-v5.
ProxSVRG again does not have a good performance in this example in terms of loss and test accuracy but is slightly better than ProxSGD regarding gradient mapping norm.

\beforesec
\section{Conclusions}\label{sec:concluding}
\aftersec
We have proposed a unified stochastic proximal-gradient framework using the SARAH estimator to solve both the composite expectation problem \eqref{eq:sopt_prob} and the composite finite sum problem \eqref{eq:finite_sum}.
Our algorithm is different from existing stochastic proximal gradient-type methods such as ProxSVRG and ProxSpiderBoost at which we have  an additional averaging step.
Moreover, it can work with both single sample and mini-batch using either constants or adaptive step-sizes.
Our adaptive step-size is updated in an increasing fashion as opposed to a diminishing step-size in ProxSGD.
We have established the best-known complexity bounds for all cases.
We believe that our methods give more flexibility to trade-off between step-sizes and mini-batch in order to obtain good performance in practice.  
The numerical experiments have shown that our methods are comparable or even outperform existing methods, especially in the single sample case.

% Acknowledgements should go at the end, before appendices and references
\beforesec
%\acks{
\section*{Acknowledgements}
\aftersec
We would like to acknowledge the support for this project from the National Science Foundation (NSF grant DMS-1619884).
%}

%%%%%%%%%%%%%%%%%%%%%%%%%%%%%%%%%%%%%%%%%%%%%%%%%%%
%%% Appendix: Technical proofs.
%%%%%%%%%%%%%%%%%%%%%%%%%%%%%%%%%%%%%%%%%%%%%%%%%%%
%\newpage

\appendix
\beforesec
\section{Technical lemmas}\label{apdx:le:mini_batch}
\aftersec
This appendix provides the missing proofs of Lemma~\ref{le:mini_batch} and one elementary result, Lemma~\ref{le:adaptive_step_size}, used in our analysis in the sequel.

%%% Lemma 11.
\begin{lemma}\label{le:adaptive_step_size}
Given three positive constants $\nu$, $\delta$, and $L$, let $\set{\gamma_t}_{t=0}^m$ be a positive sequence satisfying the following conditions:
\begin{equation}\label{eq:key_cond} 
\left\{\begin{array}{ll}
L\gamma_m - \delta &\leq 0, \vspace{1ex}\\
\nu L^2\gamma_{t}\sum_{j=t+1}^m\gamma_j - \delta + L\gamma_t & \leq 0,~~ t=0,\cdots, m-1.
\end{array}\right.
\end{equation}
Then, the following statements hold:
\begin{itemize}
\vspace{-1.25ex}
\item[$\mathrm{(a)}$]
The sequence $\set{\gamma_t}_{t=0}^m$ computed recursively in a backward mode as
\begin{equation}\label{eq:update_gamma_t}
\gamma_m := \frac{\delta}{L},~~~\text{and}~~\gamma_t := \frac{\delta}{L\big[1 + \nu L\sum_{j=t+1}^m\gamma_j\big]},~~t=0,\cdots, m-1,
\end{equation}
tightly satisfies \eqref{eq:key_cond}.
Moreover, we have $\frac{\delta}{L(1 + \delta\nu m)} < \gamma_0 < \gamma_1 < \cdots < \gamma_m$ and 
\begin{equation}\label{eq:lower_bound}
\Sigma_m := \sum_{t=0}^m\gamma_t \geq \frac{2\delta(m+1)}{L\big[\sqrt{1 + 2\delta\nu m} + 1\big]}.
\end{equation}
\vspace{-1.5ex}
\item[$\mathrm{(b)}$]
The constant sequence $\set{\gamma_t}_{t=0}^m$ with $\gamma_t := \frac{2\delta}{L(\sqrt{1 + 4\delta\nu m} + 1)}$ satisfies \eqref{eq:key_cond}.
\vspace{-1ex}
\end{itemize}
\end{lemma}
%\nhanp{wrong label for equation (58)?}

%% Beginning of the proof.
\begin{proof}
(a)~The sequence $\set{\gamma_t}_{t=0}^m$ given by \eqref{eq:update_gamma_t} is in fact computed from \eqref{eq:key_cond} by setting all the inequalities ``$\leq$'' to equalities ``$=$''.
Hence, it automatically satisfies \eqref{eq:key_cond}.
Moreover, it is obvious that $\gamma_0 < \gamma_1 < \cdots < \gamma_m$.
Since $\sum_{t=1}^m\gamma_t < m\gamma_m = \frac{m\delta}{L}$, we have $\gamma_0 > \frac{\delta}{L(1 + \delta\nu m)}$.

Let $\Sigma_m := \sum_{t=0}^m\gamma_t$.
Using $\Sigma_m$ into \eqref{eq:key_cond} with all equalities, we can rewrite it as
\begin{equation*}
\left\{\begin{array}{lll}
\nu L^2\gamma_m\Sigma_m &= \delta - L\gamma_m &+ {~} \nu L^2(\gamma_m^2 + \gamma_m\gamma_{m-1} +  \gamma_m\gamma_{m-2} + \cdots + \gamma_m\gamma_0) \vspace{1ex}\\
\nu L^2\gamma_{m-1}\Sigma_m &= \delta - L\gamma_{m-1} &+ {~} \nu L^2(\gamma_{m-1}^2 + \gamma_{m-1}\gamma_{m-2} +  \gamma_{m-1}\gamma_{m-3} +  \cdots + \gamma_{m-1}\gamma_0) \vspace{1ex}\\
\cdots & \cdots & \cdots \vspace{1ex}\\ 
\nu L^2\gamma_1\Sigma_m &= \delta - L\gamma_1 &+ {~} \nu L^2(\gamma_1^2 + \gamma_1\gamma_0) \vspace{1ex}\\
\nu L^2\gamma_0\Sigma_m &= \delta - L\gamma_0 &+ {~} \nu L^2\gamma_0^2.
\end{array}\right.
\end{equation*}
Summing up both sides of these equations, and using the definition of $\Sigma_m$ and $S_m^2 := \sum_{t=0}^m\hat{\eta}_t^2$, we obtain 
\begin{equation*}
\nu L^2\Sigma_m^2 = (m+1)\delta - L\Sigma_m + \frac{\nu L^2}{2}(\Sigma_m^2 + S_m^2).
\end{equation*}
Since $(m+1)S_m^2 \geq \Sigma_m^2$ by the Cauchy-Schwarz inequality, the last expression leads to 
\begin{equation*}
\nu L^2\Sigma_m^2 + 2L\Sigma_m - 2\delta (m+1) = \nu L^2S_m^2 \geq \frac{\nu L^2\Sigma_m^2}{m+1}.
\end{equation*}
Therefore,  by solving the quadratic inequation $\nu m L^2\Sigma_m^2 + 2(m+1)L\Sigma_m - 2\delta (m+1)^2 \geq 0$ in $\Sigma_m$ with $\Sigma_m > 0$, we obtain
\begin{equation*}
\Sigma_m \geq \frac{2\delta(m+1)}{L\big[1 + \sqrt{1 + 2\delta\nu m }\big]},
\end{equation*}
which is exactly \eqref{eq:lower_bound}.

\vspace{1ex}
\noindent (b)~Let $\gamma_t := \gamma > 0$ for $t=0,\cdots, m$. 
Then \eqref{eq:key_cond} holds if $\nu L^2\gamma^2m - \delta + L\gamma = 0$.
Solving this quadratic equation in $\gamma$ and noting that $\gamma > 0$, we obtain $\gamma = \frac{2\delta}{L(\sqrt{1 + 4\delta\nu m} + 1)}$.
\end{proof}
%% End of the proof.

%%%%% The proof of Lemma 2.
\beforepara
\begin{proof}(\textbf{\textit{The proof of Lemma~\ref{le:mini_batch}: Properties of stochastic estimators}}):
We only prove \eqref{eq:mini_batch_est2}, since other statements were proved in \citep{harikandeh2015stopwasting,lohr2009sampling,Nguyen2017_sarahnonconvex,Nguyen2018sgd_dnn}.
The proof of \eqref{eq:mini_batch_est2} for $\vert \Omega\vert = n$ was also given in \citep{Nguyen2018sgd_dnn} but under the $L$-smoothness of each $f_i$, we conduct this proof here by following the same path as in \citep{Nguyen2018sgd_dnn} for completeness.

Our goal is to prove \eqref{eq:mini_batch_est2b} by upper bounding the following quantity:
\begin{equation}\label{eq:At_term}
\Ac_t := \Exp{\Vert v_t - v_{t-1}\Vert^2 \mid \Fc_t} - \Vert \nabla{f}(w_t) - \nabla{f}(w_{t-1})\Vert^2.
\end{equation}
Let $\Fc_t := \sigma(w_0^{(s)}, \Bc_1, \cdots, \Bc_{t-1})$ be the $\sigma$-field generated by $w_0^{(s)}$ and mini-batches $\Bc_1, \cdots, \Bc_{t-1}$, and $\Fc_0 = \Fc_1 = \sigma(w_0^{(s)})$.
If we define $\Xi_{i} := \nabla{f_i}(w_t) -  \nabla{f_i}(w_{t-1})$, then using the update rule \eqref{eq:mini_batch}, we can upper bound $\Ac_t$ in \eqref{eq:At_term} as
\begin{equation*}
\begin{array}{ll}
\Ac_t &= \Exp{\Vert \frac{1}{b_t}\sum_{i\in\Bc_t}\Xi_i\Vert^2 \mid \Fc_t} - \Vert \frac{1}{n}\sum_{i=1}^n\Xi_i\Vert^2 \vspace{1ex}\\
&= \frac{1}{b_t^2}\Exp{\sum_{i\in\Bc_t}\sum_{j\in\Bc_t}\iprods{\Xi_i,\Xi_j} \mid \Fc_t} - \frac{1}{n^2}\sum_{i=1}^n\sum_{j=1}^n\iprods{\Xi_i,\Xi_j} \vspace{1ex}\\
&= \frac{1}{b_t^2}\Exp{\sum_{i,j\in\Bc_t,i\neq j}\iprods{\Xi_i,\Xi_j} + \sum_{i\in\Bc_t}\Vert\Xi_i \Vert^2 \mid \Fc_t} - \frac{1}{n^2}\sum_{i=1}^n\sum_{j=1}^n\iprods{\Xi_i,\Xi_j} \vspace{1ex}\\
&= \frac{1}{b_t^2}\Big[\frac{b_t(b_t-1)}{n(n-1)}\sum_{i,j=1,i\neq j}^n\iprods{\Xi_i,\Xi_j} + \frac{b_t}{n}\sum_{i=1}^n\Vert\Xi_i \Vert^2\Big]  - \frac{1}{n^2}\sum_{i=1}^n\sum_{j=1}^n\iprods{\Xi_i,\Xi_j} \vspace{1ex}\\
&= \frac{(b_t-1)}{b_tn(n-1)}\sum_{i,j=1}^n\iprods{\Xi_i,\Xi_j} +  \frac{(n - b_t)}{b_tn(n-1)}\sum_{i=1}^n\Vert\Xi_i \Vert^2 - \frac{1}{n^2}\sum_{i=1}^n\sum_{j=1}^n\iprods{\Xi_i,\Xi_j} \vspace{1ex}\\
&= \frac{(n - b_t)}{b_tn(n-1)}\sum_{i=1}^n\Vert \Xi_i\Vert^2 - \frac{(n - b_t)}{(n-1)b_t}\Vert \frac{1}{n} \sum_{i=1}^n\Xi_i\Vert^2 \vspace{1ex}\\
&= \frac{(n - b_t)}{b_t(n-1)}\frac{1}{n}\sum_{i=1}^n\Vert \nabla{f_i}(w_t) - \nabla{f_i}(w_{t-1})\Vert^2 - \frac{(n - b_t)}{(n-1)b_t}\Vert \nabla{f}(w_t) - \nabla{f}(w_{t-1})\Vert^2,
\end{array}
\end{equation*}
where we use the facts that 
\begin{equation*}
\begin{array}{ll}
& \Exp{\sum_{i,j\in\Bc_t,i\neq j}\iprods{\Xi_i,\Xi_j} \mid \Fc_t} = \frac{b_t(b_t-1)}{n(n-1)}\sum_{i,j=1,i\neq j}^n\iprods{\Xi_i,\Xi_j}\vspace{1ex}\\
\text{and}~& \Exp{ \sum_{i\in\Bc_t}\Vert \Xi_i\Vert^2  \mid \Fc_t} = \frac{b_t}{n}\sum_{i=1}^n\Vert\Xi_i\Vert^2
\end{array}
\end{equation*}
in the third line of the above derivation.
Rearranging the estimate $\Ac_t$, we obtain \eqref{eq:mini_batch_est2}.

To prove \eqref{eq:mini_batch_est2b}, we define $\Xi_i := \nabla_w{f}(w_t;\xi_i) - \nabla_w{f}(w_{t-1};\xi_i)$.
Clearly, $\Exp{\Xi_i \mid \Fc_t} = \nabla{f}(w_t) - \nabla{f}(w_{t-1})$ and  $v_t - v_{t-1} = \frac{1}{b_t}\sum_{i\in\Bc_t}\Xi_i$.
Similar to \eqref{eq:mini_batch_est1}, we have
\begin{equation*}
\begin{array}{ll}
\Exp{\Vert (v_t - v_{t-1}) - \Exp{\Xi_i\mid \Fc_t} \Vert^2 \mid \Fc_t} &= \frac{1}{b_t}\Exp{ \Vert \Xi_i - \Exp{\Xi_i\mid \Fc_t} \Vert^2 \mid \Fc_t}.
\end{array}
\end{equation*}
Using the fact that $\Exp{\Vert X - \Exp{X}\Vert^2} = \Exp{\Vert X\Vert^2} - \Vert\Exp{X}\Vert^2$, after rearranging, we obtain from the last expression that 
\begin{equation*}
\begin{array}{ll}
\Exp{\Vert v_t - v_{t-1}\Vert^2 \mid \Fc_t} &= \left(1 - \frac{1}{b_t}\right)\Vert \nabla{f}(w_t) - \nabla{f}(w_{t-1})\Vert^2 \vspace{1ex}\\
& + {~} \frac{1}{b_t}\Exp{\Vert \nabla_w{f}(w_t;\xi) - \nabla_w{f}(w_{t-1};\xi)\Vert^2 \mid \Fc_t},
\end{array}
\end{equation*}
which is indeed \eqref{eq:mini_batch_est2b}.
\end{proof}
%%% End of the proof.

%%%%%%%%%%%%%%%%%%%%%%%%%%%%%%%%%%%%%%%%%%%%%%%
%%% Appendix B. The proof of technical results in Section
%%%%%%%%%%%%%%%%%%%%%%%%%%%%%%%%%%%%%%%%%%%%%%%
\beforesec
\section{The proof of technical results in Section~\ref{sec:alg_section}}
\aftersec
We provide the full proof of the results in Section~\ref{sec:alg_section}.

%%% The proof of Lemma 3.1.
\beforesubsec
\subsection{The proof of Lemma~\ref{le:key_est1}: The analysis of the inner loop}\label{apdx:le:key_est1}
\aftersubsec
From the update $w_{t+1}^{(s)} := (1-\gamma_t)w_t^{(s)} + \gamma_t\widehat{w}_{t+1}^{(s)}$, we have $w_{t+1}^{(s)} - w_t^{(s)} = \gamma_t(\widehat{w}_{t+1}^{(s)} - w_t^{(s)})$.
Firstly, using the $L$-smoothness of $f$ from \eqref{eq:upper_bound_fw} of Assumption~\ref{as:A2}, we can derive
\begin{equation}\label{eq:lm31_est1}
\begin{array}{ll}
f(w_{t+1}^{(s)}) &\leq f(w_t^{(s)}) + \iprods{\nabla{f}(w_t^{(s)}), w_{t+1}^{(s)} - w_t^{(s)}} + \frac{L}{2}\Vert w_{t+1}^{(s)} - w_t^{(s)}\Vert^2 \vspace{1ex}\\
&= f(w_t^{(s)}) + \gamma_t\iprods{\nabla{f}(w_t^{(s)}), \widehat{w}_{t+1}^{(s)} - w_t^{(s)}} + \frac{L\gamma_t^2}{2}\Vert \widehat{w}_{t+1}^{(s)} - w_t^{(s)}\Vert^2.
\end{array}
\end{equation}
Next, using the convexity of $\psi$, one can show that
\begin{equation}\label{eq:lm31_est2}
\psi(w_{t+1}^{(s)}) \leq (1-\gamma_t)\psi(w_t^{(s)}) + \gamma_t\psi(\widehat{w}_{t+1}^{(s)}) \leq \psi(w_t^{(s)}) + \gamma_t\iprods{\nabla{\psi}(\widehat{w}_{t+1}^{(s)}), \widehat{w}_{t+1}^{(s)} - w_t^{(s)}},
\end{equation}
where $\nabla{\psi}(\widehat{w}_{t+1}^{(s)}) \in \partial{\psi}(\widehat{w}_{t+1}^{(s)})$.

By the optimality condition of $\widehat{w}_{t+1}^{(s)} := \prox_{\eta_t\psi}(w_t^{(s)} - \eta_tv_t^{(s)})$, we have $\nabla{\psi}(\widehat{w}_{t+1}^{(s)}) = -v_t^{(s)} - \frac{1}{\eta_t}(\widehat{w}_{t+1}^{(s)} - w_t^{(s)})$ for some $\nabla{\psi}(\widehat{w}_{t+1}^{(s)}) \in \partial{\psi}(\widehat{w}_{t+1}^{(s)})$.
Substituting this expression into \eqref{eq:lm31_est2}, we obtain
\begin{equation}\label{eq:lm31_est3}
\psi(w_{t+1}^{(s)}) \leq \psi(w_t^{(s)}) + \gamma_t\iprods{v_t^{(s)}, w_t^{(s)} - \widehat{w}_{t+1}^{(s)}} - \frac{\gamma_t}{\eta_t}\Vert\widehat{w}_{t+1}^{(s)} - w_t^{(s)} \Vert^2.
\end{equation}
Combining \eqref{eq:lm31_est1} and \eqref{eq:lm31_est3}, and then using $F(w) := f(w) + \psi(w)$ yields
\begin{equation}\label{eq:lm31_est4}
F(w_{t+1}^{(s)}) \leq F(w_t^{(s)}) + \gamma_t\iprods{\nabla{f}(w_t^{(s)}) - v_t^{(s)}, \widehat{w}_{t+1}^{(s)} - w_t^{(s)}} - \Big(\frac{\gamma_t}{\eta_t} - \frac{L\gamma_t^2}{2}\Big)\Vert\widehat{w}_{t+1}^{(s)} - w_t^{(s)}\Vert^2.
\end{equation}
Now, for any $c_t > 0$, we have
\begin{equation*}
\begin{array}{ll}
\iprods{\nabla{f}(w_t^{(s)}) - v_t^{(s)}, \widehat{w}_{t+1}^{(s)} - w_t^{(s)}} &= \frac{1}{2c_t}\Vert\nabla{f}(w_t^{(s)}) - v_t^{(s)} \Vert^2 + \frac{c_t}{2}\Vert \widehat{w}_{t+1}^{(s)} - w_t^{(s)}\Vert^2 \vspace{1ex}\\
&- {~} \frac{1}{2c_t}\Vert \nabla{f}(w_t^{(s)}) - v_t^{(s)} - c_t(\widehat{w}_{t+1}^{(s)} - w_t^{(s)})\Vert^2.
\end{array}
\end{equation*}
Utilizing this inequality, we can rewrite \eqref{eq:lm31_est4} as
\begin{equation*} 
F(w_{t+1}^{(s)}) \leq F(w_t^{(s)}) + \frac{\gamma_t}{2c_t}\Vert\nabla{f}(w_t^{(s)}) - v_t^{(s)} \Vert^2 - \Big(\frac{\gamma_t}{\eta_t} - \frac{L\gamma_t^2}{2}- \frac{\gamma_tc_t}{2}\Big)\Vert\widehat{w}_{t+1}^{(s)} - w_t^{(s)}\Vert^2 - \sigma_t^{(s)},
\end{equation*}
where $\sigma_t^{(s)} := \frac{\gamma_t}{2c_t}\Vert \nabla{f}(w_t^{(s)}) - v_t^{(s)} - c_t(\widehat{w}_{t+1}^{(s)} - w_t^{(s)})\Vert^2 \geq 0$.

Taking expectation both sides of this inequality over the entire history, we obtain
\begin{equation}\label{eq:lm31_est6} 
\begin{array}{ll}
\Exp{F(w_{t+1}^{(s)})} &\leq \Exp{F(w_t^{(s)})} + \frac{\gamma_t}{2c_t}\Exp{\Vert\nabla{f}(w_t^{(s)}) - v_t^{(s)}\Vert^2} \vspace{1ex}\\
& - {~} \Big(\frac{\gamma_t}{\eta_t} - \frac{L\gamma_t^2}{2}- \frac{\gamma_tc_t}{2}\Big)\Exp{\Vert\widehat{w}_{t+1}^{(s)} - w_t^{(s)}\Vert^2} - \Exp{\sigma_t^{(s)}}.
\end{array}
\end{equation}
Next, recall from \eqref{eq:gradient_mapping} that $G_{\eta}(w) := \frac{1}{\eta}\big(w - \mathrm{prox}_{\eta\psi}(w - \eta\nabla{f}(w))\big)$ is the gradient mapping of $F$.
In this case, it is obvious that
\begin{equation*}
 \eta_t\Vert G_{\eta_t}(w_t^{(s)})\Vert = \Vert w_t^{(s)} - \mathrm{prox}_{\eta_t\psi}(w_t^{(s)} - \eta_t\nabla{f}(w_t^{(s)}))\Vert.   
\end{equation*}
Using this definition, the triangle inequality, and the nonexpansive property $\Vert \prox_{\eta\psi}(z) - \prox_{\eta\psi}(w)\Vert \leq \Vert z - w\Vert$ of $\prox_{\eta\psi}$, we can derive that
\begin{equation*}
\begin{array}{ll}
\eta_t\Vert G_{\eta_t}(w_t^{(s)})\Vert &\leq \Vert \widehat{w}_{t+1}^{(s)} - w_t^{(s)}\Vert + \Vert \mathrm{prox}_{\eta_t\psi}(w_t^{(s)} - \eta_t\nabla{f}(w_t^{(s)})) - \widehat{w}_{t+1}^{(s)}\Vert \vspace{1ex}\\
& = \Vert \widehat{w}_{t+1}^{(s)} - w_t^{(s)}\Vert + \Vert \prox_{\eta_t\psi}(w_t^{(s)} - \eta_t\nabla{f}(w_t^{(s)})) - \prox_{\eta_t\psi}(w_t^{(s)} - \eta_tv^{(s)}_t)\Vert \vspace{1ex}\\
& \leq \Vert\widehat{w}_{t+1}^{(s)} - w_t^{(s)}\Vert + \eta_t\Vert \nabla{f}(w_t^{(s)}) - v_t^{(s)}\Vert.
\end{array}
\end{equation*}
Now, for any $r_t > 0$, the last estimate leads to
\begin{equation*} 
\eta_t^2\Exp{\Vert G_{\eta_t}(w_t^{(s)})\Vert^2} \leq \left(1+\tfrac{1}{r_t}\right)\Exp{\Vert\widehat{w}_{t+1}^{(s)} - w_t^{(s)}\Vert^2} +  (1+r_t)\eta_t^2\Exp{\Vert \nabla{f}(w_t^{(s)}) - v_t^{(s)}\Vert^2}.
\end{equation*}
Multiplying this inequality by $\frac{s_t}{2} > 0$ and adding the result to \eqref{eq:lm31_est6}, we finally get
\begin{equation*} 
\begin{array}{ll}
\Exp{F(w_{t+1}^{(s)})} &\leq \Exp{F(w_t^{(s)})} - \frac{s_t\eta_t^2}{2}\Exp{\Vert G_{\eta_t}(w_t^{(s)})\Vert^2} \vspace{1ex}\\
& +  {~} \frac{1}{2}\Big[\frac{\gamma_t}{c_t} +  (1+r_t)s_t\eta_t^2\Big] \Exp{\Vert\nabla{f}(w_t^{(s)}) - v_t^{(s)}\Vert^2} \vspace{1ex}\\
&- {~} \frac{1}{2}\Big[\frac{2\gamma_t}{\eta_t} - L\gamma_t^2 -  \gamma_tc_t - s_t\left(1+\frac{1}{r_t}\right)\Big] \Exp{\Vert\widehat{w}_{t+1}^{(s)} - w_t^{(s)}\Vert^2}  - \Exp{\sigma_t^{(s)}}.
\end{array}
\end{equation*}
Summing up this inequality from $t=0$ to $t=m$, we obtain
\begin{equation}\label{eq:lm31_est7}
{\!\!\!\!}\begin{array}{ll}
\Exp{F(w_{m+1}^{(s)})} &\leq \Exp{F(w_0^{(s)})} +\dfrac{1}{2}\displaystyle\sum_{t=0}^m \Big[\frac{\gamma_t}{c_t} + (1+r_t)s_t\eta_t^2\Big] \Exp{\Vert\nabla{f}(w_t^{(s)}) - v_t^{(s)}\Vert^2}  \vspace{1ex}\\
& - {~} \dfrac{1}{2}\displaystyle\sum_{t=0}^m\Big[\frac{2\gamma_t}{\eta_t} - L\gamma_t^2 - \gamma_tc_t - s_t\Big(1+\frac{1}{r_t}\Big)\Big] \Exp{\Vert\widehat{w}_{t+1}^{(s)} - w_t^{(s)}\Vert^2} \vspace{1ex}\\
&- {~} \displaystyle\sum_{t=0}^m\frac{s_t\eta_t^2}{2}\Exp{\Vert G_{\eta_t}(w_t^{(s)})\Vert^2} - \displaystyle\sum_{t=0}^m\Exp{\sigma_t^{(s)}}.
\end{array}{\!\!\!\!}
\end{equation}
We consider two cases:

%%% Case 1.
\noindent\textbf{Case 1:}
If $\vert \Omega \vert = n$, i.e. Algorithm~\ref{alg:prox_sarah} solves \eqref{eq:finite_sum}, then 
from \eqref{eq:mini_batch_est2} of Lemma~\ref{le:mini_batch}, the $L$-smoothness condition \eqref{eq:L_smooth_fi} in Assumption~\ref{as:A2}, the choice $\hat{b}_t^{(s)} = \hat{b} \geq 1$, and $w_j^{(s)} - w_{j-1}^{(s)} = \gamma_{j-1}(\widehat{w}_j^{(s)} - w_{j-1}^{(s)})$, we can estimate
\begin{equation*}
\begin{array}{ll}
\cExp{\norms{v^{(s)}_j - v^{(s)}_{j-1}}^2}{\Fc_{j}}  &\overset{\tiny\eqref{eq:mini_batch_est2} }{=} \frac{n(\hat{b}-1)}{\hat{b}(n-1)} \norms{\nabla{f}(w_j) - \nabla{f}(w_{j-1})}^2 \vspace{1ex}\\
& + {~} \frac{n-\hat{b}}{\hat{b}(n-1)}\frac{1}{n}\sum_{i=1}^n\norms{\nabla{f_i}(w_j^{(s)}) - \nabla{f_i}(w^{(s)}_{j-1})}^2  \vspace{1ex}\\
&\overset{\tiny\eqref{eq:L_smooth_fi}}{\leq} \norms{\nabla{f}(w_j) - \nabla{f}(w_{j-1})}^2 +  \frac{(n-\hat{b})L^2}{\hat{b}(n-1)}\norms{w_j^{(s)} - w_{j-1}^{(s)}}^2 \vspace{1ex}\\
&=  \norms{\nabla{f}(w_j) - \nabla{f}(w_{j-1})}^2 +  \frac{(n-\hat{b}) L^2\gamma_{j-1}^2}{\hat{b}(n-1)}\norms{\widehat{w}_j^{(s)} - w_{j-1}^{(s)}}^2.
\end{array}
\end{equation*}
%%% Case 2.
\noindent\textbf{Case 2:}
If $\vert \Omega \vert \neq n$, i.e. Algorithm~\ref{alg:prox_sarah} solves \eqref{eq:sopt_prob}, then from \eqref{eq:mini_batch_est2b} of Lemma~\ref{le:mini_batch}, we have
\begin{equation*}
\begin{array}{ll}
\cExp{\norms{v^{(s)}_j - v^{(s)}_{j-1}}^2}{\Fc_{j}}  &\overset{\tiny\eqref{eq:mini_batch_est2b} }{=} \left(1 - \frac{1}{\hat{b}}\right)\Vert \nabla{f}(w_j) - \nabla{f}(w_{j-1})\Vert^2 \vspace{1ex}\\
&  + {~}  \frac{1}{\hat{b}}\Exp{\norms{\nabla_w{f}(w_j; \xi) - \nabla_w{f}(w_{j-1};\xi)}^2  \mid \Fc_j}  \vspace{1ex}\\
& \overset{\tiny\eqref{eq:L_smooth}}{\leq} \Vert \nabla{f}(w_j) - \nabla{f}(w_{j-1})\Vert^2 + \frac{L^2}{\hat{b}}\norms{w_j^{(s)} - w_{j-1}^{(s)}}^2 \vspace{1ex}\\
&=\Vert \nabla{f}(w_j) - \nabla{f}(w_{j-1})\Vert^2 +  \frac{L^2\gamma_{j-1}^2}{\hat{b}}\norms{\widehat{w}_j^{(s)} - w_{j-1}^{(s)}}^2.
\end{array}
\end{equation*}
Using either one of the two last inequalities and \eqref{eq:biased_sum}, then taking the full expectation, we can derive
\begin{equation}\label{eq:lm31_est8}
{\!\!\!\!\!\!\!\!}\begin{array}{ll}
\Exp{\Vert\nabla{f}(w_t^{(s)}) - v_t^{(s)}\Vert^2} {\!\!\!}&= \Exp{\norms{\nabla{f}(w_0^{(s)}) - v_0^{(s)}}^2}  \sum_{j=1}^{t}\Exp{\Vert v_j^{(s)} - v_{j-1}^{(s)}\Vert^2} \vspace{1ex}\\
& - {~} \sum_{j=1}^t \Exp{\Vert \nabla{f}(w_j) - \nabla{f}(w_{j-1})\Vert^2} \vspace{1ex}\\
&\leq  \Exp{\norms{\nabla{f}(w_0^{(s)}) - v_0^{(s)}}^2}  + \rho L^2\sum_{j=1}^t\gamma_{j-1}^2\Exp{\norms{\widehat{w}_j^{(s)} - w_{j-1}^{(s)}}^2} \vspace{1ex}\\
&=  \bar{\sigma}^{(s)} +  \rho L^2\sum_{j=1}^{t}\gamma_{j-1}^2\Exp{\Vert\widehat{w}_j^{(s)} - w_{j-1}^{(s)}\Vert^2},
\end{array}{\!\!\!\!\!\!\!\!}
\end{equation}
where $\bar{\sigma}^{(s)}  := \Exp{\norms{\nabla{f}(w_0^{(s)}) - v_0^{(s)}}^2} \geq 0$, and $\rho := \frac{1}{\hat{b}}$ if Algorithm~\ref{alg:prox_sarah} solves \eqref{eq:sopt_prob}, and $\rho := \frac{n-\hat{b}}{\hat{b}(n-1)}$ if Algorithm~\ref{alg:prox_sarah} solves \eqref{eq:finite_sum}.

\noindent Substituting the estimate \eqref{eq:lm31_est8} into \eqref{eq:lm31_est7}, we finally arrive at
\begin{equation*} 
\begin{array}{ll}
\Exp{F(w_{m+1}^{(s)})} {\!\!\!}&\leq \Exp{F(w_0^{(s)})} + \frac{\rho L^2}{2}\displaystyle\sum_{t=0}^m\Big[\frac{\gamma_t}{c_t} + (1+r_t)s_t\eta_t^2\Big]\displaystyle\sum_{j=1}^{t}\gamma_{j-1}^2\Exp{\Vert\widehat{w}_j^{(s)} - w_{j-1}^{(s)}\Vert^2} \vspace{1ex}\\
& - {~} \displaystyle\frac{1}{2}\sum_{t=0}^m\Big[\frac{2\gamma_t}{\eta_t} - L\gamma_t^2 - \gamma_t c_t  -  s_t\Big(1+\frac{1}{r_t}\Big)\Big]\Exp{\Vert\widehat{w}_{t+1}^{(s)} - w_t^{(s)}\Vert^2} \vspace{1ex}\\
& - {~} \displaystyle\sum_{t=0}^m\frac{s_t\eta_t^2}{2}\Exp{\Vert G_{\eta_t}(w_t^{(s)})\Vert^2} - \displaystyle\sum_{t=0}^m\Exp{\sigma_t^{(s)}} + \displaystyle\frac{1}{2}\sum_{t=0}^m\Big[\dfrac{\gamma_t}{c_t} +  (1+r_t)s_t\eta_t^2\Big]\bar{\sigma}^{(s)},
\end{array}
\end{equation*}
which is exactly \eqref{eq:key_est1}.
\Eproof
%% End of the proof.

%%% The proof of Lemma 3.2.
%%%%%%%%%%%%%%%%
\beforesubsec
\subsection{The proof of Lemma~\ref{le:constant_stepsize}: The selection of constant step-sizes}\label{apdx:le:constant_stepsize}
\aftersubsec
Let us first fix all the parameters and step-sizes as constants as follows:
\begin{equation*} 
c_t := 1, ~\gamma_t := \gamma \in (0, 1], ~~\eta_t := \eta > 0, ~~r_t := 1, ~~\text{and}~~s_t := \gamma > 0.
\end{equation*}
We also denote  $a_t^{(s)} := \Exp{\Vert \widehat{w}_{t+1}^{(s)} - w_t^{(s)}\Vert^2} \geq 0$.

\noindent 
Let  $\rho := \frac{1}{\hat{b}}$ if Algorithm~\ref{alg:prox_sarah} solves \eqref{eq:sopt_prob} and $\rho := \frac{n-\hat{b}}{\hat{b}(n-1)}$ if Algorithm~\ref{alg:prox_sarah} solves \eqref{eq:finite_sum}.
Using these expressions into \eqref{eq:key_est1}, we can easily show that
\begin{equation}\label{eq:est12}
\begin{array}{ll}
\Exp{F(w_{m+1}^{(s)})} &\leq \Exp{F(w_0^{(s)})} + \frac{\rho L^2\gamma^3}{2}\big[1 + 2\eta^2\big]\displaystyle\sum_{t=0}^m\sum_{j=1}^{t}a_{j-1}^{(s)} \vspace{1ex}\\
& - {~} \frac{\gamma}{2}\left[\frac{2}{\eta} - L\gamma - 3\right]\displaystyle\sum_{t=0}^ma_t^{(s)} - \displaystyle\tfrac{\gamma\eta^2}{2} \sum_{t=0}^m\Exp{\Vert G_{\eta_t}(w_t^{(s)})\Vert^2 } \vspace{1ex}\\
& + {~} \tfrac{\gamma}{2}\left[1 + 2\eta^2\right](m+1)\bar{\sigma}^{(s)}  -  \displaystyle\sum_{t=0}^m\Exp{\sigma_t^{(s)}}  \vspace{1ex}\\
&= \Exp{F(w_0^{(s)})}  -  \frac{\gamma\eta^2}{2}\displaystyle\sum_{t=0}^m\Exp{\Vert G_{\eta_t}(w_t^{(s)})\Vert^2} - \displaystyle\sum_{t=0}^m\Exp{\sigma_t^{(s)}} \vspace{1ex}\\
& + {~} \tfrac{\gamma}{2}\left[1 + 2\eta^2\right](m+1)\bar{\sigma}^{(s)}  + \Tc_m,
\end{array}
\end{equation}
where $\Tc_m$ is defined as
\begin{equation*}
\Tc_m :=  \frac{\rho L^2\gamma^3\left(1 + 2\eta^2\right)}{2}\sum_{t=0}^m\sum_{j=1}^{t}a_{j-1}^{(s)} - \frac{\gamma}{2}\left[\frac{2}{\eta} - L\gamma - c - \left(1 + \frac{1}{r}\right)\right]\sum_{t=0}^ma_t^{(s)}.   
\end{equation*}
Our goal is to choose $\eta > 0$, and $\gamma\in (0, 1]$ such that $\Tc_m \leq 0$.
We first rewrite $\Tc_m$ as follows:
\begin{equation*}
\begin{array}{ll}
\Tc_m &=  \frac{\rho L^2\gamma^3\left(1 + 2\eta^2\right)}{2}\Big[m a_0^{(s)} + (m-1)a_1^{(s)} + \cdots + 2a_{m-2}^{(s)} + a_{m-1}^{(s)}\Big]  \vspace{1ex}\\
& - {~} \frac{\gamma}{2}\Big[\frac{2}{\eta} - L\gamma - 3\Big] \left[a_0^{(s)} + a_1^{(s)} + \cdots + a_m^{(s)}\right].
\end{array}
\end{equation*}
By synchronizing the coefficients of the terms $a_0^{(s)}, a_1^{(s)}, \cdots, a_m^{(s)}$, to guarantee $\Tc_m \leq 0$, we need to \nhanp{satisfy}
\begin{equation}\label{eq:key_cond_new}
\left\{\begin{array}{ll}
\rho\left(1 + 2\eta^2\right)L^2\gamma^2m - \left[\frac{2}{\eta} - L\gamma - 3\right] &\leq 0, \vspace{1ex}\\ 
\frac{2}{\eta} - L\gamma - 3 &\geq 0.
\end{array}\right.
\end{equation}
Assume that $\frac{2}{\eta} - L\gamma - 3 = 1 > 0$. 
This implies that $\eta = \frac{2}{L\gamma + 4}$.
Next, since $L\gamma > 0$, we have $\eta \leq \frac{1}{2}$. Therefore, we can upper bound
\begin{equation*}
\rho L^2\gamma^2m(1 + 2\eta^2) - \left[\frac{2}{\eta} - L\gamma - 3\right] \leq  \frac{3\rho L^2\gamma^2m}{2} - 1 = 0.
\end{equation*}
The last equation and $\eta = \frac{2}{L\gamma + 4}$ lead to
\begin{equation*}
\gamma := \frac{1}{L\sqrt{\omega m}}~~~\text{and}~~\eta := \frac{2\sqrt{\omega m}}{4\sqrt{\omega m} + 1},
\end{equation*}
which is exactly \eqref{eq:constant_stepsize}, where $\omega := \frac{3(n - \hat{b})}{2\hat{b}(n-1)}$ for \eqref{eq:finite_sum} and $\omega := \frac{3}{2\hat{b}}$ for \eqref{eq:sopt_prob}.

Finally, using this choice \eqref{eq:constant_stepsize} of the step-sizes, we can derive that
\begin{equation}\label{eq:key_est}
{\!\!\!}\Exp{F(w_{m+1}^{(s)}} \leq \Exp{F(w_0^{(s)})} - \frac{\gamma\eta^2}{2}\sum_{t=0}^m\Exp{\Vert G_{\eta}(w_t^{(s)})\Vert^2} - \sum_{t=0}^m\Exp{\sigma_t^{(s)}} + \frac{\gamma\theta}{2}(m+1)\bar{\sigma}^{(s)},{\!\!\!}
\end{equation}
which is exactly \eqref{eq:key_est2}, where $\theta := 1 + 2\eta^2 \leq  \frac{3}{2}$.
\Eproof
%%% End of the proof.

%% Proof of Theorem 3.2.
\beforesubsec
\subsection{The proof of Theorem~\ref{th:convergence_composite_finite_sum_b}: The adaptive step-size case}\label{apdx:th:convergence_composite_finite_sum_b}
\aftersubsec
Let $\beta_t := \frac{\gamma_t}{c_t} + (1+r_t)s_t\eta_t^2$ and $\kappa_t := \frac{2\gamma_t}{\eta_t} -  L\gamma_t^2- \gamma_t c_t - s_t\left(1+\frac{1}{r_t} \right)$ be defined in Lemma~\ref{le:key_est1}.
From \eqref{eq:key_est1} of Lemma~\ref{le:key_est1} we have
\begin{equation}\label{eq:key_est1_new}
\begin{array}{ll}
\Exp{F(w_{m+1}^{(s)})} {\!\!\!\!}&\leq  \Exp{F(w_0^{(s)})}    - \displaystyle\sum_{t=0}^m\frac{s_t\eta_t^2}{2}\Exp{\Vert G_{\eta_t}(w_t^{(s)})\Vert^2}\vspace{1ex}\\
& + {~} \dfrac{1}{2}\bar{\sigma}^{(s)}\Big(\displaystyle\sum_{t=0}^m\beta_t\Big)  -  \displaystyle\sum_{t=0}^m\Exp{\sigma_t^{(s)}} + \Tc_m,
\end{array}
\end{equation}
where
\begin{equation*}
\Tc_m := \displaystyle\frac{L^2(n-\hat{b})}{2\hat{b}(n-1)}\sum_{t=0}^m\beta_t\displaystyle\sum_{j=1}^{t}\gamma_{j-1}^2\Exp{\Vert\widehat{w}_j^{(s)} - w_{j-1}^{(s)}\Vert^2} - \displaystyle\frac{1}{2}\sum_{t=0}^m\kappa_t \Exp{\Vert\widehat{w}_{t+1}^{(s)} - w_t^{(s)}\Vert^2}.
\end{equation*}
Now, to guarantee $\Tc_m \leq 0$, let us choose all the parameters such that
\begin{equation}\label{eq:param_cond10b}
\left\{\begin{array}{ll}
\kappa_m &= 0, \vspace{1ex}\\
\frac{(n-\hat{b})}{\hat{b}(n-1)}L^2\gamma_{t}^2\sum_{j=t+1}^m\beta_j - \kappa_t & = 0,~~ t=0,\cdots, m-1.
\end{array}\right.
\end{equation}
Then, the above inequality reduces to
\begin{equation}\label{eq:another_est} 
\Exp{F(w_{m+1}^{(s)})} \leq \Exp{F(w_0^{(s)})}   - \displaystyle\sum_{t=0}^m\frac{s_t\eta_t^2}{2}\Exp{\Vert G_{\eta_t}(w_t^{(s)})\Vert^2}   +  \frac{1}{2}\sum_{t=0}^m\beta_t\bar{\sigma}^{(s)}.
\end{equation}
If we choose $c_t = r_t = 1$, $s_t = \gamma_t$, fix $\eta_t = \eta \in (0, \frac{2}{3})$, and define $\delta := \frac{2}{\eta} - 3 > 0$, then \eqref{eq:param_cond10b} reduces to
\begin{equation}\label{eq:param_cond10c}
\left\{\begin{array}{ll}
\delta - L\gamma_m &= 0, \vspace{1ex}\\
\frac{L^2(n-\hat{b})(1+2\eta^2)}{\hat{b}(n-1)}\gamma_{t}\sum_{j=t+1}^m\gamma_j - \delta + L\gamma_t & = 0,~~ t=0,\cdots, m-1.
\end{array}\right.
\end{equation}
Applying Lemma~\ref{le:adaptive_step_size}(a) with $\nu = \omega_{\eta} := \frac{(n-\hat{b})(1+2\eta^2)}{\hat{b}(n-1)}$, 
we obtain from \eqref{eq:param_cond10c} that
\begin{equation}\label{eq:update_param20b} 
\gamma_m := \frac{\delta}{L},~~~\text{and}~~\gamma_t := \frac{\delta}{L\big[1 + \omega_{\eta} L \sum_{j=t+1}^m\gamma_j\big]},~~t=0,\cdots, m-1.
\end{equation}
Moreover, we have
\begin{equation*} 
\frac{\delta}{L(1+ \omega_{\eta}\delta m)} < \gamma_0 < \gamma_1 < \cdots < \gamma_m,~~~\text{and}~~~\Sigma_m := \sum_{t=0}^m\gamma_t \geq \frac{2\delta(m+1)}{L(\sqrt{2\omega_{\eta}\delta m + 1} + 1)},
\end{equation*}
which proves \eqref{eq:Sigma_lower_bound}.

On the other hand, by using \eqref{eq:mini_batch_est1b}, the estimate \eqref{eq:another_est} leads to
\begin{equation*}
\frac{1}{S\Sigma_m}\sum_{s=1}^S\sum_{t=0}^m\gamma_t\Exp{\Vert G_{\eta}(w_t^{(s)})\Vert^2}  \leq  \frac{2}{\eta^2S\Sigma_m}\big[F(\widetilde{w}_0) - F^{\star}\big]  + \frac{3\sigma_n^2}{2\eta^2S}\sum_{s=1}^S\frac{(n-b_s)}{nb_s},
\end{equation*}
which is exactly \eqref{eq:mini_batch_bound}.

Now, let us choose $\eta := \frac{1}{2} \in (0, \frac{2}{3})$.
Then, we have $\delta = 1$, $\omega_{\eta} = \frac{3(n-\hat{b})}{2\hat{b}(n-1)}$, and $\Sigma_m \geq \frac{2\delta(m+1)}{L(\sqrt{2\omega_{\eta} m + 1} + 1)}$.
Using these facts,  $\widetilde{w}_T ~\sim \Ub_p\big(\sets{w_t^{(s)}}_{t=0\to m}^{s=1\to S}\big)$ with  $\Prob{\widetilde{w}_T = w_t^{(s)}} = p_{(s-1)m+t} := \frac{\gamma_t}{S\Sigma_m}$, and $b_s = n$, we obtain from \eqref{eq:mini_batch_bound} that
\begin{equation*}
\Exp{\norms{G_{\eta}(\widetilde{w}_T)}^2} = \frac{1}{S\Sigma_m}\sum_{s=1}^S\sum_{t=0}^m\gamma_t\Exp{\Vert G_{\eta}(w_t^{(s)})\Vert^2} \leq  \frac{4L(\sqrt{2\omega m + 1} + 1)}{S(m+1)}\big[F(\widetilde{w}_0) - F^{\star}\big].
\end{equation*}
Next, using $m = \lfloor\frac{n}{\hat{b}}\rfloor$ and $\omega := \omega_{\eta} = \frac{3(n-\hat{b})}{2\hat{b}(n-1)}$, if $\hat{b} \leq \sqrt{n}$, then we can bound
\begin{equation*}
\frac{\sqrt{2\omega m + 1} + 1}{m+1} \leq \frac{2\sqrt{\omega}}{\sqrt{m+1}} \leq \frac{\sqrt{6}}{\sqrt{n}}.
\end{equation*}
Using this bound, we can further bound the above estimate \nhanp{obtained from}
\eqref{eq:mini_batch_bound} \nhanp{as}
\begin{equation*}
\Exp{\norms{G_{\eta}(\widetilde{w}_T)}^2}  \leq  \frac{4\sqrt{6}L\left[F(\widetilde{w}_0) - F^{\star}\right]}{S\sqrt{n}},
\end{equation*}
which is \eqref{eq:grad_norm_bound1_c2}

To achieve $\Exp{\norms{G_{\eta}(\widetilde{w}_T)}^2} \leq \varepsilon^2$, we impose $\frac{4\sqrt{6}L\left[F(\widetilde{w}_0) - F^{\star}\right]}{S\sqrt{n}} = \varepsilon^2$, which shows that the number of outer iterations $S := \frac{4\sqrt{6}L\left[F(\widetilde{w}_0) - F^{\star}\right]}{\sqrt{n}\varepsilon^2}$.
To guarantee $S\geq 1$, we need $n \leq \frac{96L^2\left[F(\widetilde{w}_0) - F^{\star}\right]^2}{\varepsilon^4}$.

Hence, we can estimate the number of gradient evaluations $\Tc_{\mathrm{grad}}$ by 
\begin{equation*}
\Tc_{\mathrm{grad}} = Sn + 2S(m+1)\hat{b} \leq 5Sn  =  \frac{20\sqrt{6}L\sqrt{n}\left[F(\widetilde{w}_0) - F^{\star}\right]}{\varepsilon^2}.
\end{equation*}
We can conclude that the number of stochastic gradient evaluations does not exceed $\Tc_{\mathrm{grad}} = \BigO{\frac{L\sqrt{n}\left[F(\widetilde{w}_0) - F^{\star}\right]}{\varepsilon^2}}$.
The number of proximal operations $\prox_{\eta\psi}$ does not exceed $\Tc_{\prox} :=  S(m+1) \leq  \frac{4\sqrt{6}(\sqrt{n} + 1)L\left[F(\widetilde{w}_0) - F^{\star}\right]}{\hat{b}\varepsilon^2}$.
\Eproof
%%% End of the proof.

%% Proof of Theorem 3.2.
\beforesubsec
\subsection{The proof of Theorem~\ref{th:convergence_composite_finite_sum_b2}: The constant step-size case}\label{apdx:th:convergence_composite_finite_sum_b2}
If we choose $(\gamma_t, \eta_t) = (\gamma, \eta) > 0$ for all $t = 0, \cdots, m$, then, by applying Lemma~\ref{le:constant_stepsize}, we can update 
\begin{equation*}
\gamma := \frac{1}{L\sqrt{\omega m}}~~~\text{and}~~~\eta := \frac{2\sqrt{\omega m}}{4\sqrt{\omega m} + 1},
\end{equation*}
which is exactly \eqref{eq:mini_batch_step_size}, where $\omega :=  \frac{3(n-\hat{b}) }{2(n-1)\hat{b}}$.
With this update, we can simplify \eqref{eq:key_est2} as
\begin{equation*} 
\Exp{F(w_{m+1}^{(s)})} \leq \Exp{F(w_0^{(s)})} - \frac{\gamma\eta^2}{2}\sum_{t=0}^m\Exp{\Vert G_{\eta}(w_t^{(s)})\Vert^2} + \frac{3\gamma}{4}(m+1)\bar{\sigma}^{(s)}.
\end{equation*}
With the same argument as above, we obtain
\begin{equation*}
\frac{1}{(m+1)S}\sum_{s=1}^S\sum_{t=0}^m\Exp{\Vert G_{\eta}(w_t^{(s)})\Vert^2}  \leq  \frac{2}{\gamma\eta^2(m+1)S}\big[F(\widetilde{w}_0) - F^{\star}\big]  + \frac{3\sigma_n^2}{2\eta^2S}\sum_{s=1}^S\frac{(n-b_s)}{nb_s}.
\end{equation*}
For $\widetilde{w}_T ~\sim \Ub\big(\sets{w_t^{(s)}}_{t=0\to m}^{s=1\to S}\big)$ with $T := (m+1)S$ \nhanp{and $b_s = n$}, the last estimate implies
\begin{equation*}
\Exp{\Vert G_{\eta}(\widetilde{w}_T)\Vert^2 } = \frac{1}{(m+1)S}\sum_{s=1}^S\sum_{t=0}^m\Exp{\Vert G_{\eta}(w_t^{(s)})\Vert^2}  \leq  \frac{2}{\gamma\eta^2(m+1)S}\big[F(\widetilde{w}_0) - F^{\star}\big].
\end{equation*}
By the update rule of $\eta$ and $\gamma$, we can easily show that $\gamma\eta^2 \geq \frac{4\sqrt{\omega m}}{L(4\sqrt{\omega m} + 1)^2}$.
Therefore, using $m := \lfloor \frac{n}{\hat{b}} \rfloor$, we can overestimate
\begin{equation*}
\frac{1}{\gamma\eta^2(m+1)} \leq \frac{L(4\sqrt{\omega m} + 1)^2}{4\sqrt{\omega m}(m+1)} \leq \frac{8L\sqrt{\omega}}{\sqrt{m}} \leq \frac{8\sqrt{3}L}{\sqrt{2n}}.
\end{equation*}
Using this upper bound, to guarantee $\Exp{\Vert G_{\eta}(\widetilde{w}_T)\Vert^2 } \leq \varepsilon^2$, we choose $S$ and $m$ such that $\frac{16\sqrt{3}L}{S\sqrt{2n}}\big[F(\widetilde{w}_0) - F^{\star}\big] = \varepsilon^2$, which leads to $S := \frac{16\sqrt{3}L}{\sqrt{2n}\varepsilon^2}\big[F(\widetilde{w}_0) - F^{\star}\big]$ as the number of outer iterations.
To guarantee $S \geq 1$, we need to choose $n \leq  \frac{384L^2}{\varepsilon^4}\big[F(\widetilde{w}_0) - F^{\star}\big]^2$.

Finally, we can estimate the number of stochastic gradient evaluations $\Tc_{\mathrm{grad}}$ as
\begin{equation*}
\Tc_{\mathrm{grad}} = Sn + 2S(m+1) \leq 5Sn = \frac{16\sqrt{3}L\sqrt{n}}{\sqrt{2}\varepsilon^2}\big[F(\widetilde{w}_0) - F^{\star}\big] = \BigO{ \frac{L\sqrt{n}}{\varepsilon^2}\big[F(\widetilde{w}_0) - F^{\star}\big]}.
\end{equation*}
The number of $\prox_{\eta\psi}$ is $\Tc_{\prox} = S(m+1) \leq \frac{16\sqrt{3}L(\sqrt{n}+1)}{\hat{b}\sqrt{2}\varepsilon^2}\big[F(\widetilde{w}_0) - F^{\star}\big]$.
\Eproof
%% End of the proof.

%%%% Proof of Theorem 7.
\beforesubsec
\subsection{The proof of Theorem~\ref{th:convergence_composite_expectation}: The expectation problem}\label{apdx:th:convergence_composite_expectation}
\aftersubsec
Summing up \eqref{eq:key_est2} from $s=1$ to $s = S$, using $w_0^{(0)} = \widetilde{w}_0$, and ignoring the nonnegative term $\Exp{\sigma_t^{(s)}}$, we obtain
\begin{equation}\label{eq:thm33_est1} 
\frac{\gamma\eta^2}{2}\sum_{s=1}^S \sum_{t=0}^m \Exp{\Vert G_{\eta}(w_t^{(s)})\Vert^2} \leq  F(\widetilde{w}_0) - \Exp{F(w_{m+1}^{(S)})} + \frac{\gamma\theta(m+1)}{2}\sum_{s=1}^S\bar{\sigma}^{(s)}.
\end{equation}
Note that $\Exp{F(w_{m+1}^{(S)})} \geq F^{\star}$ by Assumption~\ref{as:A1}.
Moreover, by \eqref{eq:mini_batch_est1}, we have
\begin{equation*}
\bar{\sigma}^{(s)} := \Exp{\norms{v_0^{(s)} - \nabla{f}(w_0^{(s)})}^2} = \Exp{\norms{\widetilde{\nabla}f_{\Bc_s}(w_0^{(s)}) - \nabla{f}(w_0^{(s)})}^2} \leq \frac{\sigma^2}{b_s} = \frac{\sigma^2}{b}.
\end{equation*}
Let us fix $c_t = r_t = 1$ in Lemma~\ref{le:constant_stepsize}. 
Moreover, $\rho := \frac{1}{\hat{b}}$.
Therefore, we have $\theta = 1 + \frac{8\bar{\omega} m}{(1+4\sqrt{\bar{\omega} m})^2} < \frac{3}{2}$, where $\bar{\omega} := \frac{3}{2\hat{b}}$.
Using these estimates into \eqref{eq:thm33_est1}, we obtain \eqref{eq:key_est6}.

Now, since  $\widetilde{w}_T ~\sim \Ub\big(\sets{w_t^{(s)}}_{t=0\to m}^{s=1\to S}\big)$ for $T := S(m+1)$, we have
\begin{align*} 
\Exp{\Vert G_{\eta}(\widetilde{w}_T)\Vert^2} &=  \frac{1}{(m+1)S}\sum_{s=1}^S\sum_{t=0}^m \Exp{\Vert G_{\eta}(w_t^{(s)})\Vert^2} \nonumber\\
& \leq \frac{2}{\gamma\eta^2(m+1)S}[F(\widetilde{w}_0) - F^{\star}] + \frac{3 \sigma^2}{2\eta^2b}.
\end{align*}
Since $\eta = \frac{2\sqrt{\bar{\omega} m}}{4\sqrt{\bar{\omega} m} + 1} \geq  \frac{2}{5}$ and $\frac{1}{\gamma\eta^2(m+1)} \leq \frac{25L\sqrt{\bar{\omega}m}}{4(m+1)} \leq \frac{8L}{\sqrt{\hat{b}m}}$ as proved above, to guarantee $\Exp{\Vert G_{\eta}(\widetilde{w}_T)\Vert^2} \leq \varepsilon^2$, we need to set 
\begin{equation*}
\frac{16L}{S\sqrt{\hat{b}m}}[F(\widetilde{w}_0) - F^{\star}] + \frac{75\sigma^2}{8b} = \varepsilon^2.
\end{equation*}
Let us choose $b$ such that  $\frac{75\sigma^2}{8b} = \frac{\varepsilon^2}{2}$, which leads to $b := \frac{75\sigma^2}{8\varepsilon^2}$.
We also choose $m := \frac{\sigma^2}{\hat{b}\varepsilon^2}$.
To guarantee $m \geq 1$, we have $\hat{b} \leq \frac{\sigma^2}{\varepsilon^2}$.
Then, since $\frac{1}{\sqrt{\hat{b}m}} = \frac{\varepsilon}{\sigma}$, the above condition is equivalent to $\frac{16L\varepsilon}{S\sigma}[F(\widetilde{w}_0) - F^{\star}] = \frac{\varepsilon^2}{2}$, which leads to 
\begin{equation*}
S := \frac{32L}{\sigma\varepsilon}[F(\widetilde{w}_0) - F^{\star}].
\end{equation*}
To guarantee $S \geq 1$, we need to choose $\varepsilon \leq   \frac{32L}{\sigma}[F(\widetilde{w}_0) - F^{\star}]$ if $\sigma$ is sufficiently large.

Now, we estimate the total number of stochastic gradient evaluations as 
\begin{equation*}
\begin{array}{ll}
\Tc_{\mathrm{grad}} &=  \sum_{s=1}^Sb_s + 2m\hat{b}S = (b + 2m\hat{b})S = \frac{32L}{\sigma\varepsilon}[F(\widetilde{w}_0) - F^{\star}]\left(\frac{75\sigma^2}{\varepsilon^2} + \frac{2\sigma^2}{\hat{b}\varepsilon^2}\hat{b}\right) \vspace{1ex}\\
&= \frac{2464 L\sigma}{\varepsilon^3}[F(\widetilde{w}_0) - F^{\star}].
\end{array}
\end{equation*}
Hence, the number of gradient evaluations is $\BigO{ \frac{ L\sigma [F(\widetilde{w}_0) - F^{\star}] }{\varepsilon^3} }$, and the number of proximal operator calls is also $\Tc_{\prox} := S(m+1) = \frac{32\sigma L}{\hat{b}\varepsilon^2}[F(\widetilde{w}_0) - F^{\star}]$.
\Eproof
%%% End of proof.

%% The proof of Theorem 3.4.
\beforesec
\section{The proof of Theorem \ref{thm:noncomposite_convergence}: The non-composite cases}\label{apdx:thm:noncomposite_convergence}
\aftersec
Since $\psi = 0$, we have $\widehat{w}_{t+1}^{(s)} = w_t^{(s)} - \eta_tv_t^{(s)}$.
Therefore, $\widehat{w}_{t+1}^{(s)} - w_t^{(s)} = - \eta_tv_t^{(s)}$ and $w_{t+1}^{(s)} = (1-\gamma_t)w_t^{(s)} + \gamma_t\widehat{w}_{t+1}^{(s)} = w_t^{(s)} - \gamma_t\eta_tv_t^{(s)} = w_t^{(s)} - \hat{\eta}_tv_t^{(s)}$, where $\hat{\eta}_t := \gamma_t\eta_t$.
Using these relations and choose $c_t = \frac{1}{\eta_t}$, we can easily show that
\begin{equation*}
\left\{\begin{array}{ll}
&\Exp{\Vert\widehat{w}_{t+1}^{(s)} - w_t^{(s)}\Vert^2} = \eta_t^2\Exp{\Vert v_t^{(s)}\Vert^2}, \vspace{1ex}\\
&\sigma_t^{(s)} := \frac{\gamma_t}{2c_t}\Vert \nabla{f}(w_t^{(s)}) - v_t^{(s)} - c_t(\widehat{w}_{t+1}^{(s)} - w_t^{(s)})\Vert^2 = \frac{\hat{\eta}_t}{2}\Vert \nabla{f}(w_t^{(s)})\Vert^2.
\end{array}\right.
\end{equation*}
Substituting these estimates into \eqref{eq:lm31_est6} and noting that $f = F$ and $\hat{\eta}_t := \gamma_t\eta_t$, we obtain
\begin{equation}\label{eq:co41_est1} 
\begin{array}{ll}
\Exp{f(w_{t+1}^{(s)})} &\leq \Exp{f(w_t^{(s)})} + \frac{\hat{\eta}_t}{2}\Exp{\Vert\nabla{f}(w_t^{(s)}) - v_t^{(s)}\Vert^2} \vspace{1ex}\\
& - {~} \frac{\hat{\eta}_t}{2}\big(1 - L\hat{\eta}_t\big) 
\Exp{\Vert v_t^{(s)}\Vert^2} - \frac{\hat{\eta}_t}{2}\Exp{\Vert \nabla{f}(w_t^{(s)})\Vert^2}.
\end{array}
\end{equation}
On the other hand, from \eqref{eq:biased_sum}, by Assumption~\ref{as:A2}, \eqref{eq:sarah_estimator}, and $w_{t+1}^{(s)} := w_t^{(s)} - \hat{\eta}_tv_t^{(s)}$, we can derive
\begin{equation*} 
\begin{array}{ll}
\Exp{\norms{\nabla{f}(w_t^{(s)}) - v_t^{(s)}}^2} &\leq \Exp{\norms{\nabla{f}(w_0^{(s)}) - v_0^{(s)}}^2} + \sum_{j=1}^{t}\Exp{\norms{v_j^{(s)} - v_{j-1}^{(s)}}^2} \vspace{1ex}\\
&\leq \Exp{\norms{\nabla{f}(w_0^{(s)}) - v_0^{(s)}}^2} \vspace{1ex}\\
& + {~}  \rho \sum_{j=1}^{t}\Exp{\norms{\nabla_w{f}(w_j^{(s)};\xi_j^{(s)}) - \nabla_w{f}(w_{j-1}^{(s)}; \xi_j^{(s)}) }^2} \vspace{1ex}\\
&\leq \Exp{\norms{\nabla{f}(w_0^{(s)}) - v_0^{(s)}}^2} +  \rho L^2\sum_{j=1}^{t}\Exp{\norms{w_j^{(s)} -  w_{j-1}^{(s)} }^2} \vspace{1ex}\\
&\leq \Exp{\norms{\nabla{f}(w_0^{(s)}) - v_0^{(s)}}^2} +  \rho L^2\sum_{j=1}^{t}\hat{\eta}_{j-1}^2\Exp{\norms{v_{j-1}^{(s)} }^2},
\end{array}
\end{equation*}
where $\rho := \frac{1}{\hat{b}}$ if Algorithm~\ref{alg:prox_sarah} solves \eqref{eq:sopt_prob} and $\rho := \frac{n-\hat{b}}{\hat{b}(n-1)}$ if Algorithm~\ref{alg:prox_sarah} solves \eqref{eq:finite_sum}.

Substituting this estimate into \eqref{eq:co41_est1}, and summing up the result from $t=0$ to $t=m$, we eventually get
\begin{align}\label{eq:co41_est2}
{\!\!}\Exp{f(w_{m+1}^{(s)})} &\leq  \Exp{f(w_0^{(s)})}   - \displaystyle\sum_{t=0}^m\frac{\hat{\eta}_t}{2}\Exp{\Vert \nabla{f}(w_t^{(s)})\Vert^2}  + \frac{1}{2}\Big(\sum_{t=0}^m\hat{\eta}_t\Big)\Exp{\norms{\nabla{f}(w_0^{(s)}) - v_0^{(s)}}^2} \nonumber\\
&+ \frac{\rho L^2}{2}\displaystyle\sum_{t=0}^m\hat{\eta}_t\displaystyle\sum_{j=1}^{t}\hat{\eta}_{j-1}^2\Exp{\Vert v_{j-1}^{(s)}\Vert^2} - \displaystyle\sum_{t=0}^m\frac{\hat{\eta}_t(1 - L\hat{\eta}_t)}{2} \Exp{\Vert v_t^{(s)}\Vert^2}. {\!\!}
\end{align}
Our next step is to choose $\hat{\eta}_t$ such that 
\begin{equation*}
\rho L^2\displaystyle\sum_{t=0}^m\hat{\eta}_t\displaystyle\sum_{j=1}^{t}\hat{\eta}_{j-1}^2\Exp{\Vert v_{j-1}^{(s)}\Vert^2} - \displaystyle\sum_{t=0}^m\hat{\eta}_t(1 - L\hat{\eta}_t) \Exp{\Vert v_t^{(s)}\Vert^2} \leq 0.
\end{equation*}
This condition can be rewritten explicitly as
\begin{equation*}
\begin{array}{ll}
&\big[ \rho L^2\hat{\eta}_0^2(\hat{\eta}_1+\cdots + \hat{\eta}_m) - \hat{\eta}_0(1-L\hat{\eta}_0)\big]\Exp{\Vert v_0^{(s)}\Vert^2} \vspace{1ex}\\
& + {~} \big[ \rho L^2\hat{\eta}_1^2(\hat{\eta}_2 + \cdots + \hat{\eta}_m)  - \hat{\eta}_1(1-L\hat{\eta}_1)\big]\Exp{\Vert v_1^{(s)}\Vert^2} + \cdots \vspace{1ex}\\
&+ {~} \big[ \rho L^2\hat{\eta}_{m-1}^2\hat{\eta}_m - \hat{\eta}_{m-1}(1-L\hat{\eta}_{m-1})\big]\Exp{\Vert v_{m-1}^{(s)}\Vert^2} -  \hat{\eta}_m(1 - L\hat{\eta}_m)\Exp{\Vert v_m^{(s)}\Vert^2} \leq 0.
\end{array}
\end{equation*}
Similar to \eqref{eq:key_cond}, to guarantee the last inequality, we impose the following conditions
\begin{equation}\label{eq:stepsize_conds}
\left\{\begin{array}{ll}
-\hat{\eta}_m(1- L\hat{\eta}_m) &\leq 0, \vspace{1ex}\\
\rho L^2\hat{\eta}_t^2\sum_{j=t+1}^m\hat{\eta}_j - \hat{\eta}_0(1-L\hat{\eta}_0) &\leq 0.
\end{array}\right.
\end{equation}
Applying Lemma~\ref{eq:key_cond} (a) with $\nu = \rho$ and $\delta = 1$, we obtain
\begin{equation*}
\hat{\eta}_m = \frac{1}{L},~~\text{and}~~\hat{\eta}_{m-t} := \frac{1}{L\big(1 + \rho L\sum_{j=1}^t\hat{\eta}_{m-j+1}\big)},~~~\forall t=1,\cdots, m,
\end{equation*}
which is exactly \eqref{eq:adaptive_stepsize}.
With this update, we have $\frac{1}{L(1+\rho m)} < \hat{\eta}_0 < \hat{\eta}_1 < \cdots < \hat{\eta}_m$ and $\Sigma_m \geq \frac{2(m+1)}{L(\sqrt{2\rho m + 1} + 1)}$.

Using the update \eqref{eq:adaptive_stepsize}, we can simplify \eqref{eq:co41_est2} as follows:
\begin{equation*} 
{\!\!\!}\begin{array}{ll}
\Exp{f(w_{m+1}^{(s)})} \leq  \Exp{f(w_0^{(s)})}   - \displaystyle\sum_{t=0}^m\frac{\hat{\eta}_t}{2}\Exp{\Vert \nabla{f}(w_t^{(s)})\Vert^2}  + \frac{\sum_{t=0}^m\hat{\eta}_t}{2}\Exp{\norms{\nabla{f}(w_0^{(s)}) - v_0^{(s)}}^2}.
\end{array}{\!\!\!}
\end{equation*}
Let us define $\hat{\sigma}_s := \Exp{\norms{\nabla{f}(w_0^{(s)}) - v_0^{(s)}}^2}$ and noting that $f^{\star} := F^{\star} \leq \Exp{f(w_{m+1}^{(S)})}$ and $\widetilde{w}_0 := w_0^{(0)}$. 
Summing up the last inequality from $s=1$ to $S$ and using these relations, we can further derive
\begin{equation*} 
\sum_{s=1}^S\sum_{t=0}^m\hat{\eta}_t\Exp{\Vert \nabla{f}(w_t^{(s)})\Vert^2}  \leq 2\big[ f(\widetilde{w}_0) - f^{\star}\big] +   \Big(\sum_{t=0}^m\hat{\eta}_t\Big) \sum_{s=1}^S\hat{\sigma}_s.
\end{equation*}
Using the lower bound of $\Sigma_m$ as $\Sigma_m \geq \frac{2(m+1)}{L(\sqrt{2\rho m + 1} + 1)}$, the above inequality leads to
\begin{equation}\label{eq:co41_est2b}
\frac{1}{S\Sigma_m}\sum_{s=1}^S\sum_{t=0}^m\hat{\eta}_t\Exp{\Vert \nabla{f}(w_t^{(s)})\Vert^2}  \leq \frac{(\sqrt{2\rho m+1}+1)L}{S(m+1)}\big[ f(\widetilde{w}_0) - f^{\star}\big] +  \frac{1}{S}\sum_{s=1}^S\hat{\sigma}_s.
\end{equation}
Since $\Prob{\widetilde{w}_T = w_t^{(s)}} = p_{(s-1)m+t}$ with $p_{(s-1)m+t} = \frac{\hat{\eta}_t}{S\Sigma_m}$ for $s=1,\cdots, S$ and $t = 0,\cdots, m$, we have
\begin{equation*}
\Exp{\Vert \nabla{f}(\widetilde{w}_T)\Vert^2} = \frac{1}{S\Sigma_m}\sum_{s=1}^S\sum_{t=0}^m\hat{\eta}_t\Exp{\Vert \nabla{f}(w_t^{(s)})\Vert^2}.
\end{equation*}
Substituting this estimate into \eqref{eq:co41_est2b}, we obtain \eqref{eq:key_est11}.

Now, we consider two cases:

\noindent\textbf{Case (a):}
If we apply this algorithm variant to solve the non-composite finite-sum problem of \eqref{eq:finite_sum} $($i.e. $\psi = 0$$)$ using the full-gradient snapshot for the outer-loop with $b_s = n$, then $v_0^{(s)} = \nabla{f}(w^{(s)}_0)$, which leads to $\hat{\sigma}_s = 0$.
By the choice of epoch length $m = \lfloor \frac{n}{\hat{b}}\rfloor$ and $\hat{b} \leq \sqrt{n}$, we have $\frac{\sqrt{2\rho m + 1} + 1}{m+1} \leq \frac{2}{\sqrt{n}}$.
Using these facts into \eqref{eq:key_est11}, we obtain
\begin{equation*}
\Exp{\Vert \nabla{f}(\widetilde{w}_T)\Vert^2} \leq \frac{2L}{S\sqrt{n}}\big[ f(\widetilde{w}_0) - f^{\star}\big],
\end{equation*}
which is exactly \eqref{eq:key_est11a}.

To achieve $\Exp{\Vert \nabla{f}(\widetilde{w}_T)\Vert^2} \leq \varepsilon^2$, we impose $\frac{2L}{S\sqrt{n}}\big[ f(\widetilde{w}_0) - f^{\star}\big] = \varepsilon^2$.
Hence, the maximum number of outer iterations is at most $S = \frac{2L}{\sqrt{n}\varepsilon^2}[f(\widetilde{w}_0) - f^{\star}]$.
The number of gradient evaluations $\nabla{f_i}$ is at most $\Tc_{\mathrm{grad}} := nS + 2(m+1)\hat{b}S \leq 5nS =  \frac{10L\sqrt{n}}{\varepsilon^2}[f(\widetilde{w}_0) - f^{\star}]$.

\vspace{1ex}
\noindent\textbf{Case (b):} 
Let us apply this algorithm variant to solve the non-composite expectation problem of \eqref{eq:sopt_prob} $($i.e. $\psi = 0$$)$.
Then, by using $\rho := \frac{1}{\hat{b}}$ and $\hat{\sigma}_s :=  \Exp{\norms{\nabla{f}(w_0^{(s)}) - v_0^{(s)}}^2} \leq \frac{\sigma^2}{b_s} = \frac{\sigma^2}{b}$, we have from \eqref{eq:key_est11} that
\begin{equation*}
\Exp{\Vert \nabla{f}(\widetilde{w}_T)\Vert^2} \leq \frac{2L}{S\sqrt{\hat{b}m}}\big[ f(\widetilde{w}_0) - f^{\star}\big] + \frac{\sigma^2}{b}.
\end{equation*}
This is exactly \eqref{eq:key_est11b}.
Using the mini-batch $b := \frac{2\sigma^2}{\varepsilon^2}$ for the outer-loop and $m := \frac{\sigma^2}{\hat{b}\varepsilon^2}$, 
we can show that the number of outer iterations $S := \frac{4L}{\sigma\varepsilon}\big[ f(\widetilde{w}_0) - f^{\star}\big]$.
The number of stochastic gradient evaluations is at most $\Tc_{\mathrm{grad}} := Sb + 2S(m+1)\hat{b} =  \frac{4S\sigma^2}{\varepsilon^2} = \frac{16L\sigma}{\varepsilon^3}\big[ f(\widetilde{w}_0) - f^{\star}\big]$.
This holds if $ \frac{2\sigma^2}{\varepsilon^2} \leq \frac{4S\sigma^2}{\varepsilon^2} = \frac{16L\sigma}{\varepsilon^3}\big[ f(\widetilde{w}_0) - f^{\star}\big]$ leading to $\sigma \leq \frac{8L}{\varepsilon}\big[ f(\widetilde{w}_0) - f^{\star}\big]$.
\Eproof
\bibliographystyle{unsrtnat}
%\bibliography{/Users/quoctd/Dropbox/E-Books/tran_bibtex_new}

\end{document}